\documentclass[12pt,a4paper,reqno]{amsart}
\usepackage[utf8]{inputenc}
\usepackage[T1]{fontenc}
\usepackage{indentfirst}
\usepackage[english]{babel}
\usepackage{enumitem}
\usepackage{amsmath}
\usepackage{amsthm}
\usepackage{physics}
\usepackage{braket}
\usepackage{dsfont}

\usepackage{stmaryrd}
\usepackage[dvipsnames]{xcolor}
\usepackage{tikz,bm,csquotes}
\usetikzlibrary{patterns,arrows.meta}
\usepackage{pgf,interval}
\usepackage{genyoungtabtikz}
\usepackage[colorlinks=true,citecolor=DarkOrchid,linkcolor=NavyBlue,urlcolor=NavyBlue]{hyperref}
\usepackage[capitalize]{cleveref}
\usepackage{graphicx,caption}
\usepackage{afterpage}

\usepackage[a4paper,twoside,margin=0.7in]{geometry}
    \newtheorem{definition}{Definition}
    \newtheorem{lemma}[definition]{Lemma}
    
    \newtheorem{theorem}[definition]{Theorem}
    \newtheorem{proposition}[definition]{Proposition}
    \newtheorem{corollary}[definition]{Corollary}
    
    \newtheorem{remark}[definition]{Remark}
    \theoremstyle{remark}

    \makeatletter
    \def\thm@space@setup{\thm@preskip=0.5cm   \thm@postskip=0.5cm}
    \makeatother

\newcommand{\N}{\mathbb{N}}
\newcommand{\Z}{\mathbb{Z}}
\newcommand{\Q}{\mathbb{Q}}
\newcommand{\R}{\mathbb{R}}

\newcommand{\E}{\mathrm{e}}
\newcommand{\I}{\mathrm{i}}
\newcommand{\esper}{\mathbb{E}}
\newcommand{\proba}{\mathbb{P}}

\newcommand{\eps}{\varepsilon}

\renewcommand{\Re}{\mathfrak R}
\renewcommand{\Im}{\mathfrak{I}}
\newcommand{\gz}{\gamma_Z}
\newcommand{\gw}{\gamma_W}
\newcommand{\gzn}{\gamma_Z^{\text{new}}}
\newcommand{\gwn}{\gamma_W^{\text{new}}}
\newcommand{\gwno}{\gamma_{W,1}^{\text{new}}}
\newcommand{\gwnt}{\gamma_{W,2}^{\text{new}}}

\newcommand{\gwnr}{\gamma_W^{\text{new,right}}}
\newcommand{\gwnl}{\gamma_W^{\text{new,left}}}
\newcommand{\bx}{\overline x}

\def\One{\bm{1}}
\def\la{\lambda}
\def\a{\alpha}
\def\be{\beta}
\newcommand{\hh}{h^{(x_1,t_1)}_{(x_2,t_2)}}
\def\DS{D_S}

\DeclareMathOperator\Arg{Arg}

\DeclareMathOperator{\Disc}{Disc}
\DeclareMathOperator{\Int}{Int}
\DeclareMathOperator{\LHS}{LHS}
\DeclareMathOperator{\RHS}{RHS}

\usepackage{todonotes}

\setlist[enumerate]{itemsep=10pt,topsep=10pt}
\setlist[itemize]{itemsep=5pt,topsep=5pt}

\title[DPP and random Young tableaux]
{A determinantal point process approach to scaling \\ and local limits
of random Young tableaux}

\author{Jacopo Borga}
	\address[JB]{Stanford University, Department of Mathematics, 50 Jane Stanford Way, Building 380, Stanford, CA 94305-2125, USA}
	\email{jborga@stanford.edu}
\author{Cédric Boutillier}
	\address[CB]{LPSM, Sorbonne Université, Case courrier 158, 4 Place Jussieu, 75252 Paris Cedex 05, France}
	\email{cedric.boutillier@sorbonne-universite.fr}
\author{Valentin Féray}
	\address[VF]{Université de Lorraine, CNRS, IECL, F-54000 Nancy, France}
	\email{valentin.feray@univ-lorraine.fr}
\author{Pierre-Loïc Méliot}
 	\address[PLM]{Université Paris-Saclay, Faculté des Sciences d’Orsay, Institut de mathématiques d'Orsay, Bâtiment 307, F-91405 Orsay, France.}
 	\email{pierre-loic.meliot@universite-paris-saclay.fr}

\begin{document}

\begin{abstract}
  We obtain scaling and local limit results for large random Young tableaux of fixed shape $\lambda^0$ via the asymptotic analysis of a determinantal point process due to Gorin and Rahman (2019). More precisely, we prove: \begin{itemize}
  	\item an explicit description of the limiting surface of a uniform random Young tableau of shape $\lambda^0$, based on solving a complex-valued polynomial equation; 
  	\item a simple criteria to determine if the limiting surface is continuous in the whole domain;
  	\item and a local limit result in the bulk of a random Poissonized Young tableau of shape $\lambda^0$.
  \end{itemize} 
  
  Our results have several consequences, for instance: they lead to explicit formulas for the limiting surface of $L$-shaped tableaux, generalizing the results of Pittel and Romik (2007) for rectangular shapes; they imply that the limiting surface for $L$-shaped tableaux is discontinuous for almost-every $L$-shape; and they give a new one-parameter family of infinite random Young tableaux, constructed from the so-called \emph{random infinite bead process}.
\end{abstract}
\maketitle

\thispagestyle{empty}

\vspace{-1.1cm}

\tableofcontents

\section{Introduction}

\subsection{Overview}
Random Young diagrams form a classical theme in probability theory, starting with the work of
Logan--Shepp and Vershik--Kerov on the Plancherel measure \cite{LoganShepp1977,VershikKerov1977},
motivated by Ulam's problem on the typical length of the longest increasing subsequence in a uniform random permutation.
The topic also has connections with random matrix theory and particle systems,
and has known an increase of interest
after the discovery of an underlying determinantal point process
for a Poissonized version of the Plancherel measure \cite{BorodinOkounkovOlshanski2000}.
It would be vain to do a complete review of the related literature,
we only refer to \cite{RomikLIS,HoraBookYoung} for books on the topic. We also refer the reader to \cref{sect:yt_pyt_bc} for precise definitions of the objects mentioned in this section.

In comparison, the random Young tableaux, which are in essence dynamic versions of random diagrams, have a shorter history.
Motivations to study random Young tableaux range from asymptotic representation theory
to connections with other models of combinatorial probability, such as random permutations
with short monotone subsequences \cite{romik2006RandomES} 
or most notably random sorting networks; see \cite{angel2007sorting} and many later papers.

As in most of the literature, we are interested in the simple model where we fix a shape $\la^0$ (or rather a sequence of growing shapes) and consider a uniform random tableau $T$ of shape $\la^0$. In \cite{pittel2007young}, Pittel and Romik derived a limiting surface result for uniform
random Young tableaux of {rectangular} shapes, based on the hook formula \cite{HookLengthFormula1954} and counting arguments.
An earlier result of Biane in asymptotic representation theory \cite{Biane2003}
implies in fact, the existence of such limiting surfaces for any underlying shape. However, unlike in the Pittel--Romik paper, explicit computations are usually intractable: they involve the Markov--Krein correspondence and the free compression of probability measures, the latter being rarely an explicit computation.
More recently, entropy optimization methods have been applied to prove the existence of limiting surfaces,
extending the result to skew shapes \cite{sun2018dimer} (see also \cite{gordenko2020limit}). {These techniques lead to some natural gradient variational problems in $\R^2$ whose solutions are explicitly parameterized by $\kappa$-harmonic functions, as show in \cite{KenyonPrause2022variational}.}

Recently, in \cite{GR19}, a determinantal point process structure was discovered
for a Poissonized version of random Young tableaux.
This determinantal structure was used for a specific problem motivated by the aforementioned sorting networks, namely describing
the local limit of uniform tableaux of staircase shape around their outer diagonal \cite{GR19,gorin2022random}.

The goal of the current paper is to exploit this determinantal point process structure in order to get limiting results
for a large family of shapes.
Namely, we consider shapes obtained as dilatations of any given
Young diagram $\lambda^0$.
Here is an informal description of our results.
\begin{itemize}
  \item We obtain a new description of the limiting surface corresponding to the shape $\lambda^0$, based on solving a complex-valued polynomial equation (\cref{thm:limiting_surface_formula} and \cref{eq:equation_surface,eq:lim_conv}).
  This new description is more explicit compared to the one obtained through the existence approaches.
  
  \item These results led us to discover a surprising discontinuity phenomenon for the limiting surface corresponding to $\lambda^0$.
    More precisely, we establish a simple criterion -- expressed in terms of some equations involving the so-called \emph{interlacing coordinates} of $\lambda^0$ -- to determine if the limiting surface is continuous (\cref{thm:continuity}). This shows that such limiting surfaces are typically discontinuous.
  
  \item We obtain a local limit result in the bulk of a random Young tableau (\cref{thm:cv_bead_process}).
    The limit is a new model of infinite random tableaux (\cref{defn:isyt}), 
    constructed from the so-called \emph{random infinite bead process} (\cref{defn:bead_process}) introduced by the second author in \cite{boutillier2009bead}.
    Interestingly (but somehow expectedly by analogy with results
    on lozenge tilings \cite{petrov2014tilings,aggarwal2019universality}),
    this local limit depends on the underlying shape $\la$
    and on the chosen position in $\la$ only through a single parameter $\beta\in(-1,1)$.
\end{itemize}

\begin{remark}
In parallel to this work, explicit formulas for the limiting surfaces
of random Young tableaux have been obtained by Prause \cite{Prause2023tableaux}
through a different method (solving a variational
problem obtained by the tangent plane method of Kenyon and Prause
\cite{KenyonPrause2022variational}).
\end{remark}

\subsection{Young tableaux,  Poissonized Young tableaux and bead configurations}\label{sect:yt_pyt_bc}
Let us start by fixing terminology and notation.
An \emph{integer partition} of $n$, or \emph{partition} of $n$ for short, is a non-increasing list 
$\la=(\la_1,\la_2, \dots ,\la_l)$ of positive integers with 
$n = \sum_{i=1}^l\la_i$. We write $|\lambda|=n$ for the \emph{size} of the partition and $\ell(\la)=l$ for the {\em length} of the partition,
and use the convention $\la_i=0$ when $i>\ell(\la)$.
We will represent partitions graphically with the \emph{Russian convention},
i.e.\ for each $i \le \ell(\la)$ and $j \le \la_i$ we have a square box
whose sides are parallel to the diagonals $x=y$ and $x=-y$
and whose center has coordinates $(j-i,i+j-1)$;
see the left-hand side of \cref{fig:interlacing1}. We call this graphical representation the \emph{Young diagram} of shape $\lambda$. Note that with the Russian convention, the area of a cell of a Young diagram equals $2$ (and not $1$).
\begin{figure}[t]
	\YRussian
	\Yboxdim{30pt}
	\[ \begin{array}{c}
	\newcommand\enta{(0,1)}
	\newcommand\entb{(1,2)}
	\newcommand\entc{(-1,2)}
	\newcommand\entd{(2,3)}
	\newcommand\ente{(0,3)}
	\newcommand\entf{(3,4)}
	\newcommand\entg{(-2,3)}
	\newcommand\enth{(1,4)}
	\newcommand\enti{(-3,4)}
	\newcommand\entj{(2,5)}
	\newcommand\entk{(-1,4)}
	\young(\enta\entb\entd\entf,\entc\ente\enth\entj,\entg\entk,\enti)
 \end{array} \qquad \begin{array}{c}
\includegraphics[height=5.5cm]{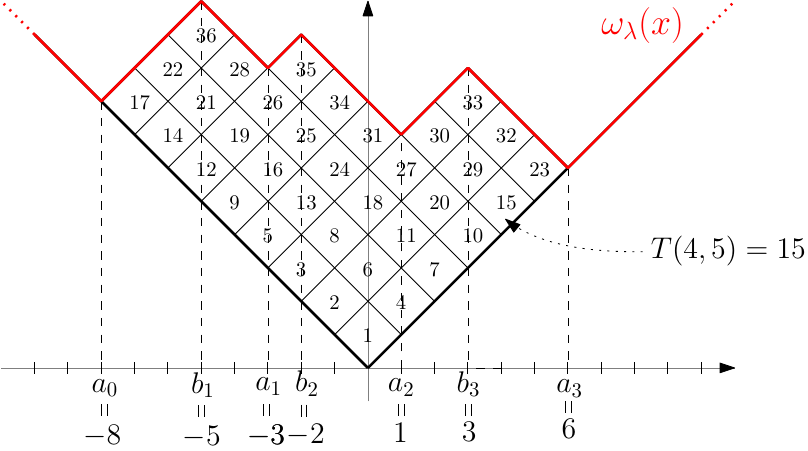}
\end{array}\]
	\captionsetup{width=\linewidth}
	\caption{\textbf{Left:} The Young diagram of the partition $(4,4,3,2)$ drawn in Russian convention,
	with the coordinates of each box inside it.
	\textbf{Right:} A Young tableau $T: \la \to [n]$ of shape $\la$ corresponding to the partition $(6,6,6,4,4,4,3,3)$ drawn with Russian convention; all the boxes are squares with area $2$. We indicate the interlacing coordinates $a_0 < b_1 < a_1 < b_2 < \dots < b_m < a_m$ below the $x$-axis.}
	\label{fig:interlacing1}
\end{figure}

When looking at a Young diagram $\lambda$, its upper boundary is the graph of a $1$-Lipschitz function, denoted by $\omega_{\lambda}:\R\to\R$,
and the diagram $\lambda$ can be encoded using the local minima and maxima of the function $\omega_{\lambda}$.
Following Kerov \cite{kerov2000anisotropic}, we denote them by
\begin{equation}\label{eq:int_coord}
	a_0 < b_1 < a_1<  b_2 < \dots < b_m < a_m, \qquad a_i,b_i\in\Z,
\end{equation}
and we call them \emph{interlacing coordinates}. See the right-hand side of \cref{fig:interlacing1} for an example.
Note that $a_0=-\ell(\la)$ and $a_m=\la_1$.
Furthermore, interlacing coordinates satisfy 
\begin{equation}\label{eq:int_rel}
	\sum_{i=0}^m   a_i = \sum_{i=1}^m  b_i,
\end{equation}
see, e.g., \cite[Proposition 2.4]{IvanovOlshanski2002}.

\medskip

A \emph{Young tableau} of shape $\la$ with $|\la|=n$ is a filling of the boxes
of $\la$ with the numbers $1, 2, \dots , n$ such that the numbers along every row and column are increasing. We encode a Young tableau as a function $T: \la \to [n]=\{1, 2, \dots , n\}$, where  the Young diagram $\la$ 
is identified with the set $\{(j-i,i+j-1), i \le \ell(\la), j \le \la_i\}$; see again the right-hand side of \cref{fig:interlacing1} for an example.
The function $T: \la \to [n]$ can be thought of as the graph of a
(non-continuous) surface above that set.

Note that, given a Young diagram $\la$, there are finitely many Young tableaux of shape $\la$, so it makes sense to consider a \emph{uniform random Young tableau} of shape $\la$.
In this paper, following \cite{GR19}, we also consider \emph{Poissonized Young tableaux} of shape $\la$,
which are functions $T: \la \to [0,1]$ with distinct real values that are increasing along rows and columns. See the left-hand side of \cref{fig:MT} for an example.
Note that for any fixed $\la$, the admissible functions $T: \la \to [0,1]$ form a subset of $[0,1]^\la$ of positive (and finite) Lebesgue measure, so it makes sense to consider a \emph{uniform random Poissonized Young tableau} of shape $\la$.
\medskip

A \emph{bead configuration} is a collection of points (called \emph{beads}) positioned on parallel vertical threads which is locally finite and satisfy an interlacing relation on the vertical positions of the beads: for every pair of consecutive beads on a thread, on each of its neighboring threads there is exactly one bead whose vertical position is between them. See the right-hand side of \cref{fig:MT} for an example. In the present paper, we will consider both  finite and infinite bead configurations, i.e.\ configurations containing finitely or infinitely many beads. In the finite case, a bead configuration
is a finite subset of $A\times[0,1]$ where $A$ is a finite sub-interval of $\mathbb{Z}$, while in the infinite case a bead configuration is a locally finite subset of $\mathbb{Z} \times \R$. In particular, each bead has a position $(x, h)$, where $x$ denotes the \emph{thread number} where the bead is positioned, and $h$ represents the \emph{height} of the bead on this thread.
\medskip

Given a Poissonized Young tableau $T: \la \to [0,1]$ of shape $\la$, we associate a finite bead configuration $M_{\la,T}$ in $([-\ell(\la),\la_1]\cap\Z) \times [0,1]=([a_0,a_m]\cap\Z)\times [0,1]$ defined by
\begin{equation}\label{eq:corr_bead_proc}
	M_{\la,T}= \Set{ (i,T(i,j)) | (i,j) \in \la }.
\end{equation}
An example is given in \cref{fig:MT}.
Note that the monotonicity condition of Poissonized Young tableaux implies that
points on neighboring vertical lines are interlacing, i.e.~satisfy the bead configuration constraints.
We also introduce the \emph{height function} $H_{\la,T}: ([a_0,a_m]\cap\Z) \times [0,1]\to \Z_{\geq 0}$, defined by
\begin{equation}\label{eq:defn_H}
	H_{\la,T}(x,t)= \#\left( M_{\la,T} \cap (\{x\} \times [0,t])\right), \quad \text{for all } (x,t)\in ([a_0,a_m]\cap\Z) \times [0,1],
\end{equation}
i.e.\ $H_{\la,T}(x,t)$ is the number of beads on the thread $x$ below height $t$.
We note that $H_{\la,T}$ is non-decreasing in $t$ and has the following boundary values:
\begin{equation}\label{eq:boundary_conditions}
	\begin{cases}
		H_{\la,T}(a_0,t)=H_{\la,T}(x,0)=H_{\la,T}(a_m,t)=0,\quad&\text{for all }t\in[0,1]\text{ and }x\in [a_0,a_m]\cap\Z,\\
		H_{\la,T}(x,1)=\tfrac12\, (\omega_\la(x)-|x|), \quad&\text{for all }x\in [a_0,a_m]\cap\Z,
	\end{cases}
\end{equation}
where we recall that $\omega_{\lambda}$ is defined above \cref{eq:int_coord}.
Clearly, the bead configuration $M_{\la,T}$ is entirely determined by the height function $H_{\la,T}$. Moreover, we have that
\begin{equation}\label{eq:syt_to_height}
	T(x,y) < t\quad\text{if and only if}\quad H(x,t) > \tfrac{1}{2}(y-|x|).
\end{equation}
Fixing a Young diagram $\la$ and taking a uniform random (Poissonized) Young tableau $T$ of shape $\la$ 
gives a random bead configuration $M_{\la,T}$
and a random height function $H_{\la,T}$, often simply denoted by $M_\la$ and $H_\la$.
\begin{figure}[ht]
\includegraphics[height=6cm]{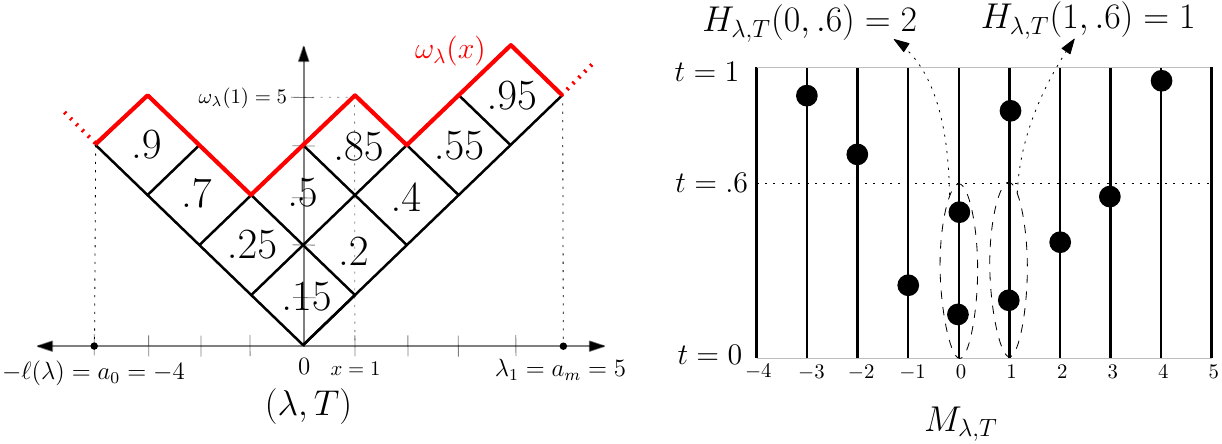}
\captionsetup{width=\linewidth}
\caption{\textbf{Left:} A Poissonized Young tableau $T$ of shape $\la=(5,3,1,1)$. \textbf{Right:} The corresponding bead configuration $M_{\la, T}$. 
To illustrate the definition of the height function,
we have indicated the values $H_{\la, T}(0,.6)$ and $H_{\la, T}(1,.6)$
and circled the corresponding beads.}
\label{fig:MT}
\end{figure}

\begin{remark}
We note that the notion of Poissonized Young tableaux had appeared in disguise
in earlier work than that of Gorin and Rahman.
Indeed, given a finite partially ordered set (or poset) $P$, 
it is standard to consider its \emph{order polytope}, i.e.~the subset of $[0,1]^P$
satisfying order constraints given by the poset ($i\le_P j \Rightarrow x_i \le x_j$).
Then the volume of this order polytope is known to be proportional to
the number of linear extensions \cite{stanley1986polytopes}.

Now, a Young diagram $\la$ can be seen as a partially ordered set,
where the elements are the cells and the order is given by coordinate-wise comparison.
Then the linear extensions are standard Young tableaux, and Poissonized Young tableaux
are points of the order polytopes.
This has been used for counting (skew) Young tableaux in \cite{elkies2003,baryshnikov_romik2010} and for the analysis of random tableaux
in \cite{marchal2016Young,banderier2020tableaux}, where the name ``continuous Young diagram'' is used.
\end{remark}

\subsection{The limiting height function and the limiting surface}
\label{sect:scaling_limit}

We fix a Young diagram $\la^0$ which will determine the shape of our growing sequence of diagrams. For an integer $n > 0$, 
we define $N=N(n,\la^0)=n^2 |\la^0|$ and consider the $(n \times n)$-dilated diagram $\la_N$
obtained by replacing each box of $\la^0$ by a square of $n \times n$ boxes.
Note that $\la_N$ has size $N$ and has the same (dilated) shape as $\la^0$.
We set $\eta=1/\sqrt{ |\la^0|}$ and consider the interval $[-\eta\, \ell(\la^0),\eta\, \la^0_1]\stackrel{\eqref{eq:int_coord}}{=}[\eta\, a_0, \eta\, a_m]$.
Informally, the interval $[\eta\, a_0, \eta\, a_m]$ is the projection on the $x$-axis of the Russian representation of $\la^0$,
scaled in both directions by a factor $\eta$ in order to have total area $2$.
In particular, given a Young tableau of shape $\la_N$, the corresponding bead configuration  $M_{\la_N}$ is supported on the set
$([{n}\, a_0, {n}\, a_m]\cap \Z) \times [0,1]$.

\medskip

The following convergence result for the height function of $M_{\la_N}$ is proved
in \cite[Theorem 7.15]{sun2018dimer} in the case of uniform random Poissonized Young tableaux. 
It also follows indirectly from concentration results on random Young diagrams
by Biane \cite{Biane2001}, as made explicit recently by \'Sniady and Ma\'slanka \cite[Proposition 10.1]{sniady2022evacuation} 
(the latter considers standard tableaux and not Poissonized tableaux,
but it is simple to see that this has no influence on the next statement).
\begin{theorem}[{\cite[Theorem 7.15]{sun2018dimer} and \cite[Proposition 10.1]{sniady2022evacuation}}]\label{thm:limiting_surface_existence}
Let $\la^0$ be a fixed Young diagram and let $T_N$
be a uniform random (Poissonized) Young tableau
of shape $\lambda_N$.
With the above notation, there exists a deterministic height function 
$H^\infty: [\eta\, a_0, \eta\, a_m] \times [0,1]\to \R$ such that the following convergence in probability holds:
\begin{equation}\label{eq:lim_height}
  \frac{1}{\sqrt N} \,H_{\la_N,T_N}(\lfloor x\sqrt N \rfloor,t) \xrightarrow[N \to +\infty]{} H^\infty(x,t),
\end{equation}
uniformly for all $(x,t)$ in $[\eta\, a_0, \eta\, a_m] \times [0,1]$.
\end{theorem}

In \cite{sun2018dimer}, the limiting function $H^\infty$ is implicitly found as the unique maximizer
of a certain entropy functional subject to some boundary conditions depending on the diagram $\la^0$.
Using the approach of \cite{Biane2001,sniady2022evacuation}, for each $t\in[0,1]$,
the section $H^\infty(\cdot,t)$ is described via the free cumulants of an associated
measure.  Both descriptions are hard to manipulate.
Our first result gives an alternative and more explicit description of $H^\infty$
via the solution of a polynomial equation, called the \emph{critical equation}.

\subsubsection{Critical equations, liquid regions and limiting height functions for bead processes}
\label{subsec:def_liquid} 
Let $a_0 < b_1 < a_1<  b_2 < \dots < a_m$ be the interlacing coordinates of $\la^0$, introduced in \eqref{eq:int_coord}.
For $(x,t)$ in $[\eta\, a_0, \eta\, a_m]  \times [0,1]$, we consider the following polynomial equation, referred to throughout the paper as the \emph{critical equation}:\footnote{This terminology is justified by the results at the beginning of \cref{ssec:localization-critical-points}.}
\begin{equation}
\label{eq:critical_intro}
	U \, \prod_{i=1}^m (x-\eta\, b_i+U) = (1-t) \, \prod_{i=0}^m (x-\eta\, a_i+U).
\end{equation}
This is a polynomial equation in the complex variable $U$ of degree $m+1$.
Using the fact that the $a_i$'s and $b_i$'s are alternating,
one can easily prove that \eqref{eq:critical_intro} has at least $m-1$ real solutions; see \cref{lem:critical_points_general} below. Hence it has either $0$ or $2$ non-real solutions.

\begin{definition}[Liquid region]\label{defn:liquid_region}
	We let $L$ be the set of pairs $(x,t)$ in $[\eta\, a_0, \eta\, a_m] \times [0,1]$ such that \eqref{eq:critical_intro} has two non-real solutions and we call it the {\em liquid region}.
	The complement of the liquid region in  $[\eta\, a_0, \eta\, a_m] \times [0,1]$
	will be referred to as the \emph{frozen region}.\footnote{The terminology \emph{liquid and frozen region} is standard in the dimer literature. See for instance \cref{thm:cv_bead_process} for a justification.}
\end{definition}

{Several equivalent descriptions of the liquid region are given in \cref{prop:description_liquid_region}. For instance, we will show that $L$ is an open subset of $[\eta\, a_0, \eta\, a_m] \times [0,1]$ and give an explicit parametrization of its boundary, i.e.\ the so-called \emph{frozen boundary curve}.}

For $(x,t) \in L$, we denote by $U_c=U_c(x,t)$ the unique solution with a positive imaginary part of the critical equation~\eqref{eq:critical_intro} 
and we define
\begin{equation}\label{eq:parm_liq}
	\a(x,t):=\frac{\Im U_c}{(1-t)}\qquad\text{and}\qquad\be(x,t):=\frac{\Re U_c}{|U_c|},
\end{equation}
where $\Im$ and $\Re$ denote the imaginary and real parts of a complex number.
For $(x,t) \notin L$, we set $\a(x,t):=0$, and we leave $\be(x,t)$ undefined. In particular, note that this defines a function $\alpha(x,t)$ for all $(x,t)\in [\eta\, a_0, \eta\, a_m]\times [0,1]$.
It turns out that the limiting height function $H^\infty$ 
is directly related to the function $\alpha$ via the following simple formula.

\begin{theorem}\label{thm:limiting_surface_formula}
With the above notation,
the limiting height function $H^\infty$ from  \cref{thm:limiting_surface_existence}
takes the form
\begin{equation*}
	H^\infty(x,t) = \frac{1}{\pi}\int_0^t \alpha(x,s)\dd{s}, \quad \text{for all }(x,t)\in [\eta\, a_0, \eta\, a_m] \times[0,1].
\end{equation*}
\end{theorem}
Informally, the asymptotic density of beads at position $(\lfloor x \sqrt N \rfloor,t)$ in $M_{\la_N}$
is $\sqrt N \, \alpha(x,t)$.
In particular, the liquid region coincides with the region
where there is a macroscopic quantity of beads.

\subsubsection{Limiting surfaces for Young tableaux and discontinuity phenomena} 
\label{ssec:limiting-surface}
It is natural to try to translate the limiting result for the bead process
to a limit result for the tableau itself, seen as a discrete surface.
Namely, we set 
\begin{equation}\label{eq:domain_tableaux}
	D_{\lambda^0}:=\left\{(x,y) \in \R^2 : |x| < y < \omega_{\eta \la^0}(x)\right\},
\end{equation}
which is the shape (seen as an open domain of the plane) of the diagram $\lambda^0$ in Russian notation, normalized to have area $2$.
For $(x,y)$ in $D_{\lambda^0}$, letting $T_N$ be a uniform Poissonized tableau
of shape $\la_N$, we consider 
\begin{equation}\label{eq:rescaling_T}
	\widetilde T_N(x,y):=  T_N(\lfloor x \sqrt N \rfloor,\lfloor y \sqrt N \rfloor +\delta),
\end{equation}
where $\delta \in \{0,1\}$ is chosen so that the arguments of $T_N$ have distinct parities.
We want to find a scaling limit for the function $\widetilde T_N(x,y)$. 
(Again, it is simple to see that considering uniform Poissonized or (classical) uniform Young tableaux of shape $\la$ is irrelevant here.)
To this end, we set for all $(x,y)\in D_{\lambda^0}$,
\begin{align}\label{eq:Tplusminus}
		&T^\infty_+=T^\infty_+(x,y):= \sup \left\{t\in[0,1] \,:\, H^\infty(x,t) \le {\tfrac 1 2 (y-|x|)} \right\},\\
		&T^\infty_-=T^\infty_-(x,y):= \inf \left\{t\in[0,1] \,:\, H^\infty(x,t) \ge {\tfrac 1 2 (y-|x|)} \right\}.\notag
\end{align}
Since $t \mapsto H^\infty(x,t)$ is non-decreasing, for $x$ fixed, the maps $y \mapsto T^\infty_-(x,y)$ and $y \mapsto T^\infty_+(x,y)$ are non-decreasing, and we also have $T^\infty_-(x,y) \leq T^\infty_+(x,y)$ for any $(x,y) \in D_{\lambda^0}$.

\begin{proposition}\label{prop:reg_points}
For all $\eps>0$, the following limit holds uniformly for all $(x,y) \in D_{\lambda^0}$:
\[\lim_{N \to +\infty} \mathbb P\big(\widetilde T_N(x,y) < T^\infty_- -\eps \big) 
=\lim_{N \to +\infty} \mathbb P\big(\widetilde T_N(x,y) > T^\infty_+ + \eps\big) =0.\]
\end{proposition}

\begin{proof}
Recalling the relation~\eqref{eq:syt_to_height} and the rescaling~\eqref{eq:rescaling_T}, we note that
\[ \widetilde T_N(x,y) <  T^\infty_- -\eps \quad\text{if and only if}\quad H_{\la_N}\big(\lfloor x \sqrt N \rfloor,T^\infty_- -\eps\big) > \frac 1 2 (\lfloor y \sqrt N \rfloor +\delta-|\lfloor x \sqrt N \rfloor |).\]
We claim that the latter event happens with probability tending to $0$. Indeed, \cref{thm:limiting_surface_existence} guarantees 
the following convergence in probability uniformly for all $(x,y) \in D_{\lambda^0}$:
\[ \lim_{N \to +\infty} \frac{1}{\sqrt N}H_{\la_N}\big(\lfloor x \sqrt N \rfloor,T^\infty_- -\eps \big) = H^\infty(x,T^\infty_- -\eps) < \tfrac 1 2 (y-|x|),\]
where the last inequality follows by definition of $T^\infty_-$. 
The statement for $T^\infty_+$ is proved similarly.
\end{proof}

We let $D^{\mathrm{reg}}_{\lambda^0}$ be the set of coordinates $(x,y)\in D_{\lambda^0}$ such that 
$T^\infty_-(x,y)=T^\infty_+(x,y)$. For such points, we simply write $T^\infty(x,y)$
for this common value.  By definition, on $D^{\mathrm{reg}}_{\lambda^0}$, one has
	\begin{equation}\label{eq:equation_surface}
		H^\infty(x,T^\infty(x,y))=\tfrac 1 2 (y-|x|).
	\end{equation}
Moreover, \cref{prop:reg_points} implies
the following convergence in probability for  $(x,y)\in D^{\mathrm{reg}}_{\lambda^0}$:
	\begin{equation}\label{eq:lim_conv}
		\lim_{N \to +\infty} \widetilde T_N(x,y) = T^\infty(x,y),
	\end{equation}
Note that this convergence holds uniformly on compact subsets of $D^{\mathrm{reg}}_{\lambda^0}$.

\begin{remark}\label{rem:cont_points}
	For any $x$ fixed, since $t\mapsto H^\infty(x,t)$ is non-decreasing, the points $(x,y)$ are in
	$D^{\mathrm{reg}}_{\lambda^0}$ for all but countably many $y$. As a consequence, $D_{\lambda^0} \setminus D^{\mathrm{reg}}_{\lambda^0}$ has zero (two-dimensional) Lebesgue measure. Moreover, if $(x,y)\in D_{\lambda^0} \setminus D^{\mathrm{reg}}_{\lambda^0}$ and $(x,y\pm \eps) \in D^{\mathrm{reg}}_{\lambda^0}$ for $\eps>0$ small enough, then
	$T^\infty(x,y-\eps)$ and $T^\infty(x,y+\eps)$ converge respectively
	to $T^\infty_-(x,y)$ and $T^\infty_+(x,y)$ as $\eps$ tends to $0$ with $\eps>0$.
	Indeed, since the map $y \mapsto T^\infty_-(x,y)$ is non-decreasing, we have
	$$
	\limsup_{\eps \to 0} \left(T^\infty(x,y-\eps)\right) = \limsup_{\eps \to 0} \left(T^\infty_-(x,y-\eps)\right) \leq T^\infty_-(x,y) .
	$$
	Conversely, $H^\infty(x,T^\infty(x,y-\eps)) = \frac{1}{2}(y - |x|-\eps)$, so $t = \liminf_{\eps \to 0} \left(T^\infty(x,y-\eps)\right)$ satisfies $H^\infty(x,t) \geq \frac{1}{2}(y-|x|-\eps)$ for any $\eps>0$. Therefore, $H^\infty(x,t) \geq \frac{1}{2}(y-|x|)$, and $t \geq T^\infty_-(x,y)$. We conclude that $\lim_{\eps \to 0} T^\infty(x,y-\eps) = T^\infty_-(x,y)$, and similarly, $\lim_{\eps \to 0} T^\infty(x,y+\eps) = T^\infty_+(x,y)$.
	Thus, the limiting surface $T_\infty$ is discontinuous at points $(x,y)$ in $D_{\lambda^0} \setminus D^{\mathrm{reg}}_{\lambda^0}$
	and we do not know whether $\widetilde T_N(x,y)$ converges or not.
	This discontinuity phenomenon was overlooked in \cite[Theorem 9.1]{sun2018dimer},
	where it is claimed that the convergence holds for all $(x,y)$ in $D_{\lambda^0}$.
\end{remark}

A natural question is whether such discontinuity points $(x,y)$ exist at all in $D_{\lambda^0}$.
The following result shows that such points indeed exist unless $\la^0$
is a rectangle, or unless its interlacing coordinates satisfy 
some specific equations.
\begin{theorem}\label{thm:continuity}
	For a Young diagram $\la^0$, the following assertions are equivalent:
	\begin{enumerate}
		\item The limiting surface $T^\infty$ is continuous in the whole domain $D_{\lambda^0}$;
		\item The interlacing coordinates defined in \eqref{eq:int_coord} satisfy the system of equations:
		\begin{equation}\label{eq:cont_cond}
			\sum_{\substack{i=0\\
				i\neq i_0}}^m \frac{1}{a_{i_0}-a_i}=\sum_{i=1}^m \frac{1}{a_{i_0}-b_i}, \qquad \text{for all }i_0=1,\dots,m-1.
		\end{equation}
	\end{enumerate}
\end{theorem}

Note that when $m=1$, i.e.\ when $\la^0$ has a rectangular shape, there are no equations in the second item. Indeed, the limiting surface $T^\infty$ is always continuous in this case.

\subsection{Applications for simple limit shapes}\label{sect:appl_intro}
In this section, we illustrate our results in the cases $m=1$ (rectangular shapes)
and $m=2$ ($L$-shapes). 
Before starting, let us note that our model and all results are invariant
when multiplying all interlacing coordinates of $\la^0$ by the same positive integers.
We will therefore allow ourselves to work with diagrams $\la^0$
with rational (non-necessarily integer) interlacing coordinates.
The statements of this section are proved in \cref{sect:applications}.

\subsubsection{An explicit formula for the rectangular case}
In this section, we consider a rectangular diagram $\la^0$.
Without loss of generality, we assume $a_0=-1$ and write $r=a_1$.
Necessarily, $b_1=r-1$. Solving explicitly the second degree critical equation \eqref{eq:critical_intro} yields:

\begin{proposition}\label{prop:limit_shape}
	The limiting height function corresponding to a $1 \times r$ rectangular shape $\la^0$
    is given by
	\begin{equation}
      \label{eq:limit_H_rectangular}
		H^\infty_r(x,t)=\frac{1}{\pi}\int_0^t\frac{\sqrt{s (4 r - (1 + r)^2 s) + 2 (r-1) \sqrt{r} s x - r x^2}}{2 \sqrt{r} (1 - s) s} \dd{s}
	\end{equation} 
	with the convention that $\sqrt{x}=0$ if $x\leq 0$.
    Moreover, the limiting surface $T^\infty_r$ is continuous on $D_{\lambda^0}$
    and is therefore implicitly determined by the equation
	\begin{equation}\label{eq:eq_cont_case}
		H^\infty_r(x,T^\infty_r(x,y))=\tfrac 1 2 (y-|x|).
	\end{equation}
\end{proposition}

\begin{remark}\label{rem:square_case}
	In the case $r=1$ (square Young tableaux), we get
	\begin{equation*}
		H^\infty_1(x,t)=\frac{1}{\pi} \int_{0}^{t} \frac{\sqrt{4 s - 4 s^2 - x^2}}{2 s - 2 s^2}\dd{s}.
	\end{equation*}
        The graph of the function $\frac{\sqrt{4 s - 4 s^2 - x^2}}{2 s - 2 s^2}$ is plotted on the left-hand side of \cref{fig:density_plot}, while the corresponding limiting surface $T^\infty_1$ is on the right.
	The above integral can be explicitly computed, 
    recovering the formula found by Pittel and Romik from \cite{pittel2007young}.
    Pittel and Romik also give formulas for the general rectangular case, which
    should coincide with \eqref{eq:limit_H_rectangular}, though we could not verify it.\end{remark}

Using precisely the same proof of \cref{prop:limit_shape}, one can also obtain explicit formulas for $L$-shaped diagrams; since the latter expressions are pretty involved, we decided not to display them.

\begin{figure}[ht]
	\includegraphics[height=4.6cm]{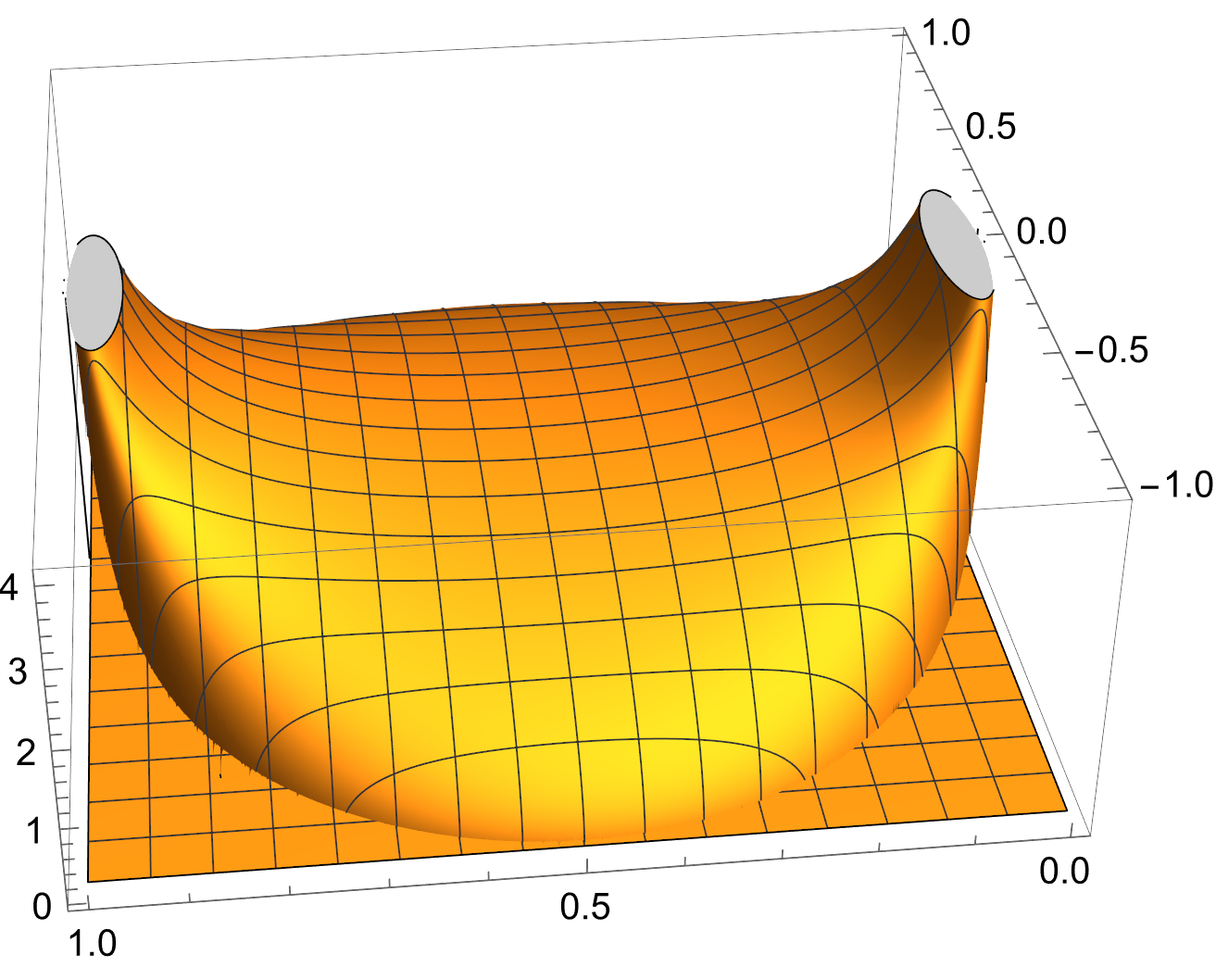}
	\qquad\qquad
	\includegraphics[height=4.6cm]{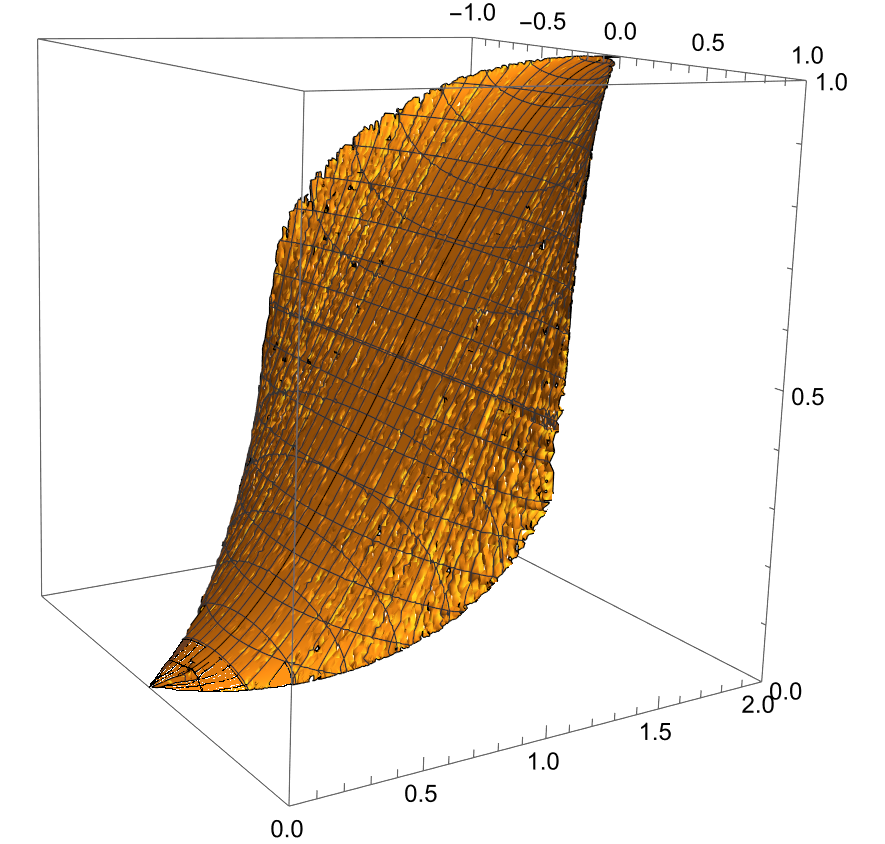}
	\captionsetup{width=\linewidth}
	\caption{\textbf{Left:} The graphs of the function $\frac{\sqrt{4 s - 4 s^2 - x^2}}{2 s - 2 s^2}$ from \cref{rem:square_case}. \textbf{Right:} The corresponding limiting surface $T^\infty_1$ for squared diagrams.
	Note that we are using two different axes' orientations to improve the visualization's quality.
		\label{fig:density_plot}}
\end{figure}

\subsubsection{Two concrete examples of $L$-shape diagrams}
\label{ssec:examples}
We now consider two specific diagrams $\la^0$ and $\widetilde \la^0$
which are both $L$-shaped (i.e.~$m=2$).
Because of the shape of the corresponding liquid regions (see pictures in \cref{fig:Example1,fig:Example2}),
the first one is called the \emph{heart example} and the second one the \emph{pipe example}. 

\medskip

In the heart example (c.f.\ \cref{fig:Example1}), the Young diagram $\la^0$ has
interlacing coordinates 
\begin{equation}\label{eq:example_1}
	a_0=-5\;<\;b_1=-4\;<\;a_1=-1\;<\;b_2=3\;<\;a_2=5.
\end{equation} 
In this case we have $|\la^0|=13$, so that $\eta=1/\sqrt{ |\la^0|}=1/\sqrt{13}$ and 
$[\eta\, a_0, \eta\, a_m]=[-5/\sqrt{13},5/\sqrt{13}]\approx[-1.39,1.39]$.

In the pipe example (c.f.\ \cref{fig:Example2}), the Young diagram $\widetilde \la^0$
has interlacing coordinates
\begin{equation}\label{eq:example_2}
	\widetilde a_0=-200\;<\;\widetilde b_1=-197\;<\; \widetilde a_1=-90\;<\;\widetilde b_2=10\;<\;\widetilde a_2=103.
\end{equation} 
In this case, we have $|\widetilde \la^0|=9900$, so that $\widetilde\eta=\frac{1}{30 \sqrt{11}}$
and
$[\widetilde \eta\, \widetilde a_0, \widetilde \eta\, \widetilde a_m]=[-\frac{200}{30 \sqrt{11}},\frac{103}{30 \sqrt{11}}]\approx[-2.01,1.04]$. 

\medskip

Various pictures of these two examples, discussed in the sequel,
are presented in \cref{fig:Example1,fig:Example2}.
\begin{figure}[p]
		\vspace{1mm}
		\includegraphics[width=0.32\linewidth]{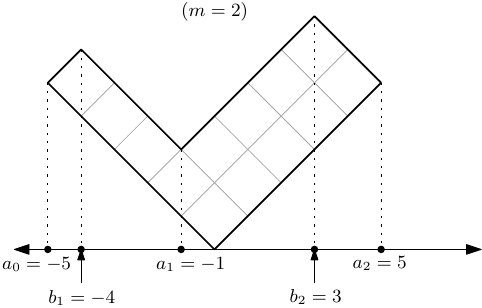}
		\includegraphics[width=0.32\linewidth]{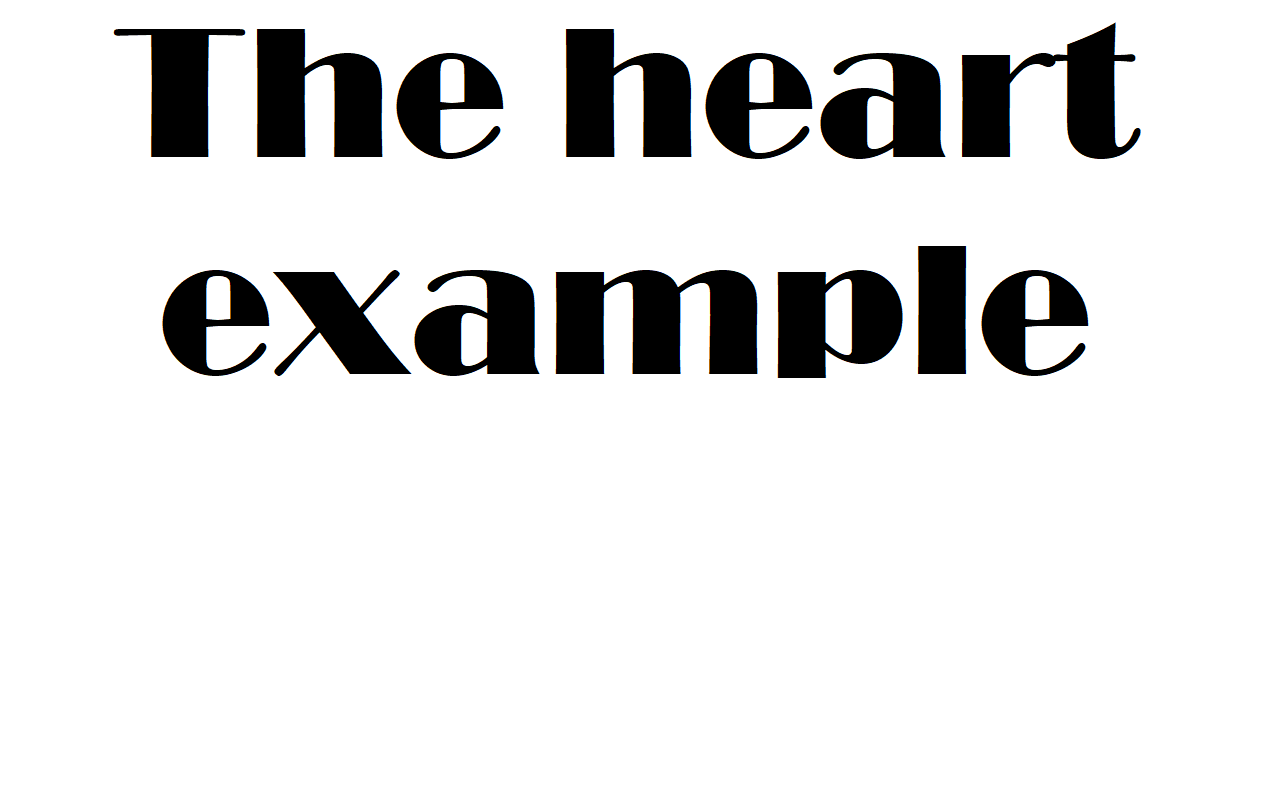}
		\includegraphics[width=0.32\linewidth]{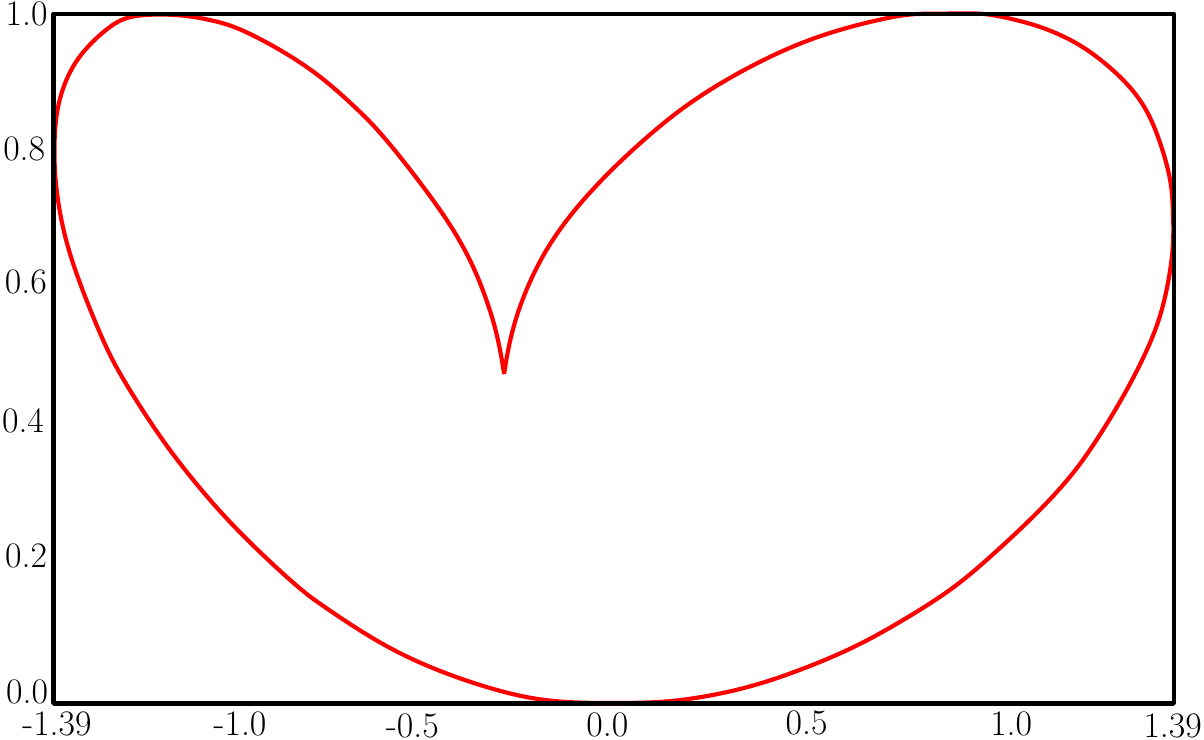}
		 \includegraphics[width=0.32\linewidth]{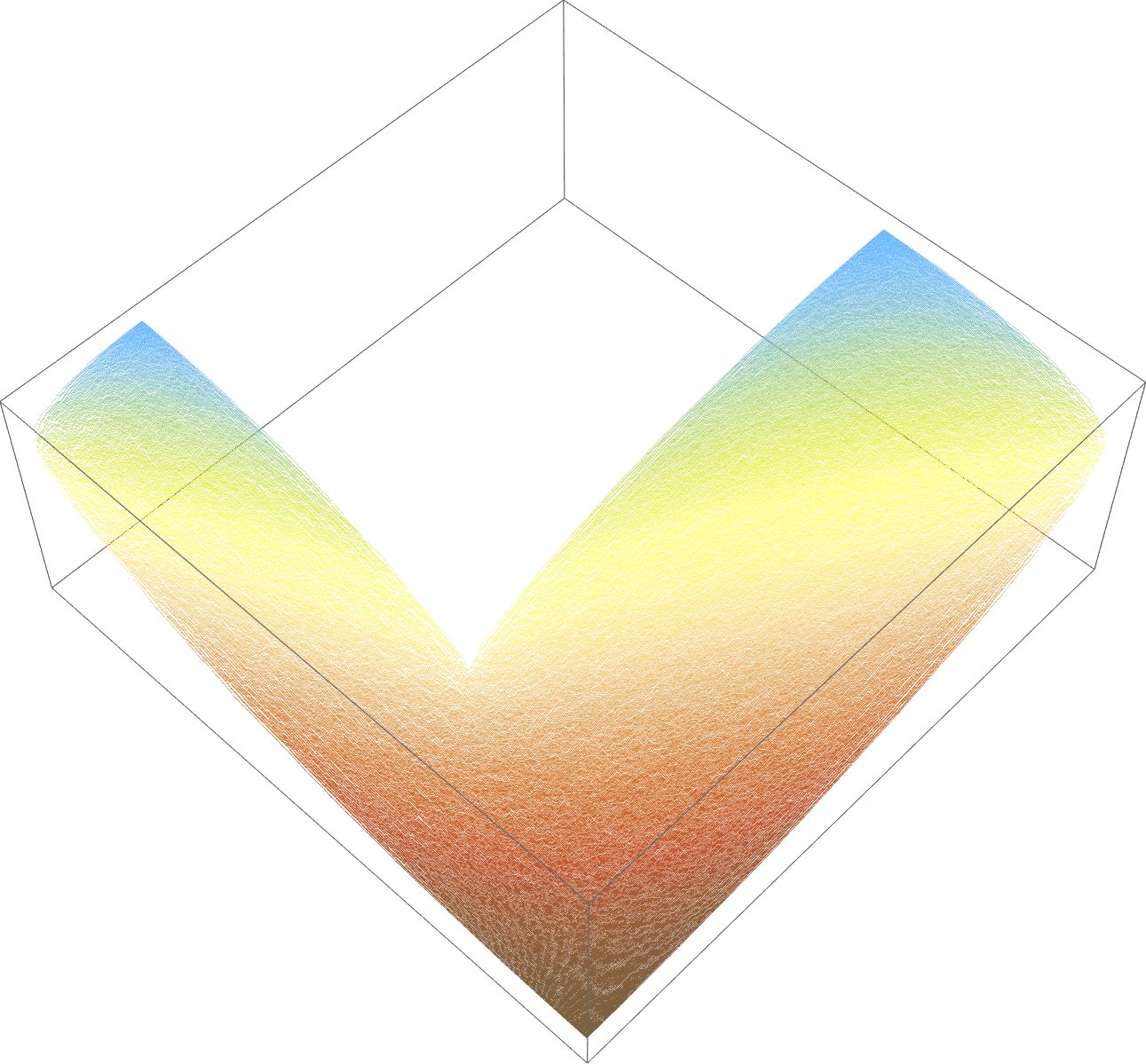}
		\includegraphics[width=0.32\linewidth]{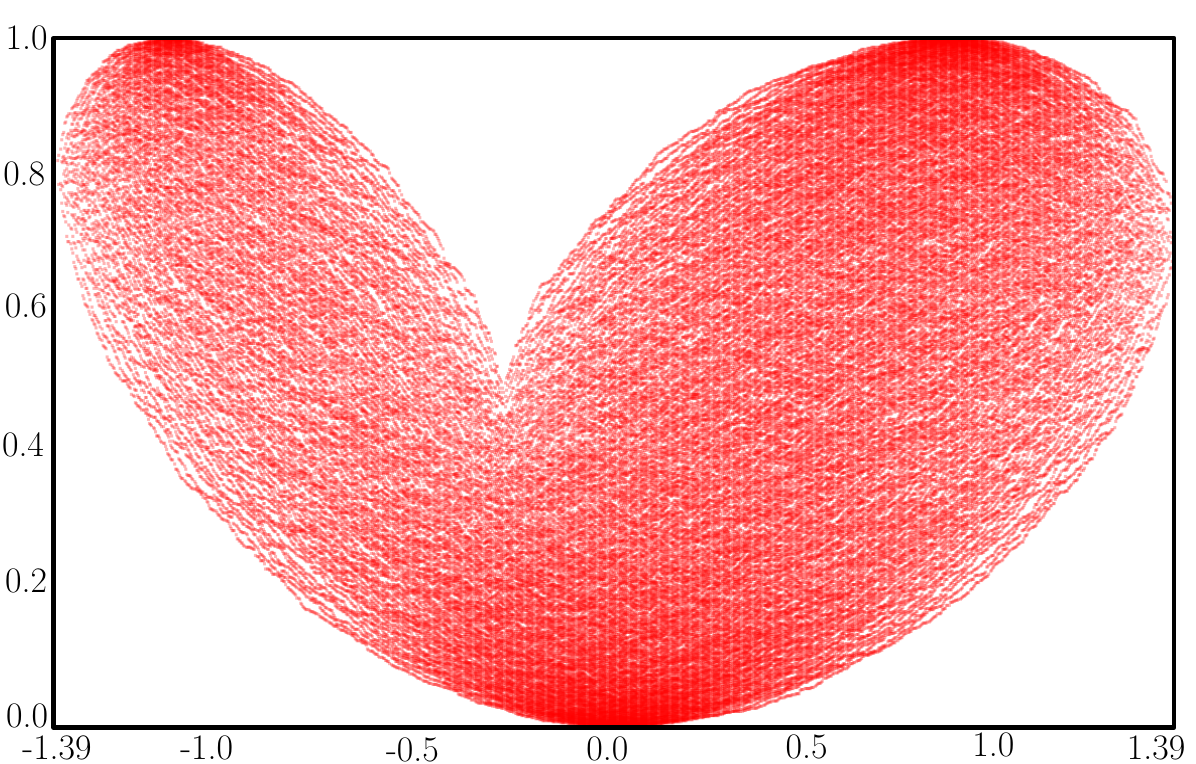}
		\includegraphics[width=0.32\linewidth]{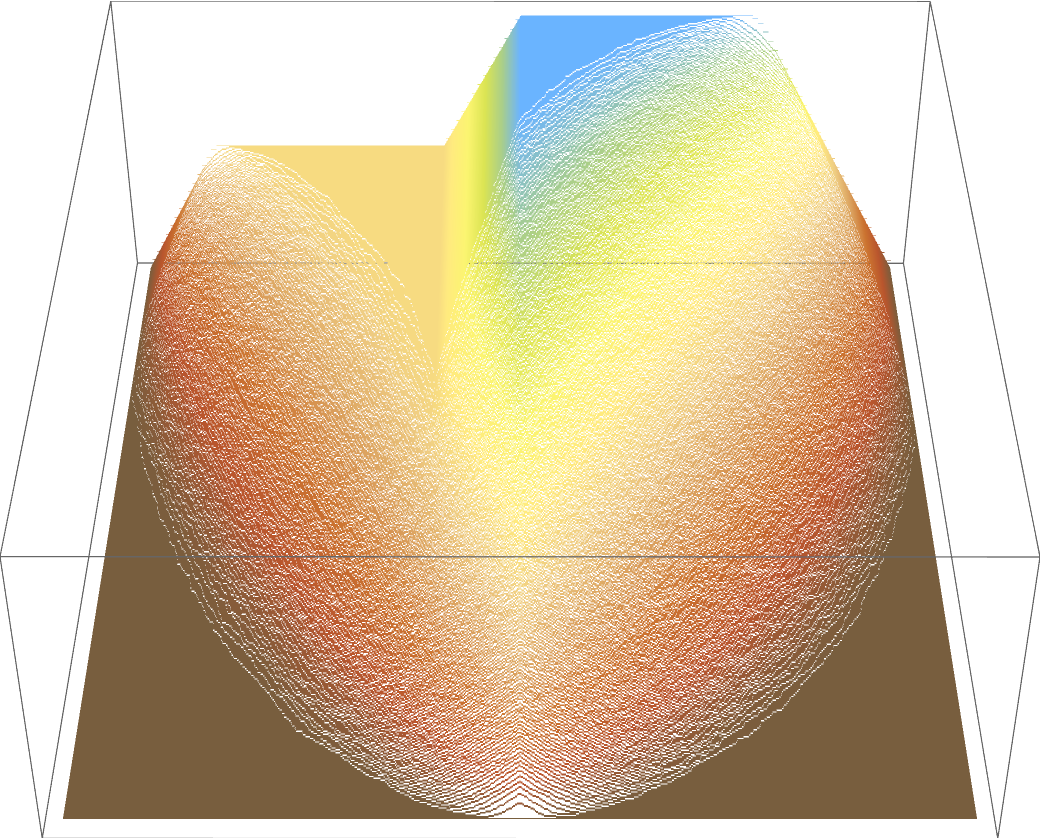}
		\vspace{1mm}
		\captionsetup{width=\linewidth}
		\caption{Figures for the heart example. \textbf{Top-left:} The Young diagram $\la^0$ considered in the heart example, see \eqref{eq:example_1}. \textbf{Top-right:} The frozen boundary curve of the corresponding liquid region.
			\textbf{Bottom (from left to right):} a uniform random tableau
			of shape $\la_N$ with $N=130000$ boxes ($n=100$) displayed as a discrete surface in a 3D space (brown is used for small values of the surface and blue for large ones); the corresponding bead processes $M_{\la_N}$; the corresponding height function $H_{\la_N}$.\label{fig:Example1}}
		\vspace{6mm}
		\includegraphics[width=0.32\linewidth]{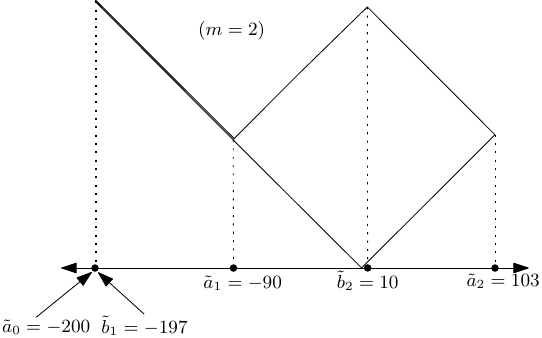}
		\includegraphics[width=0.32\linewidth]{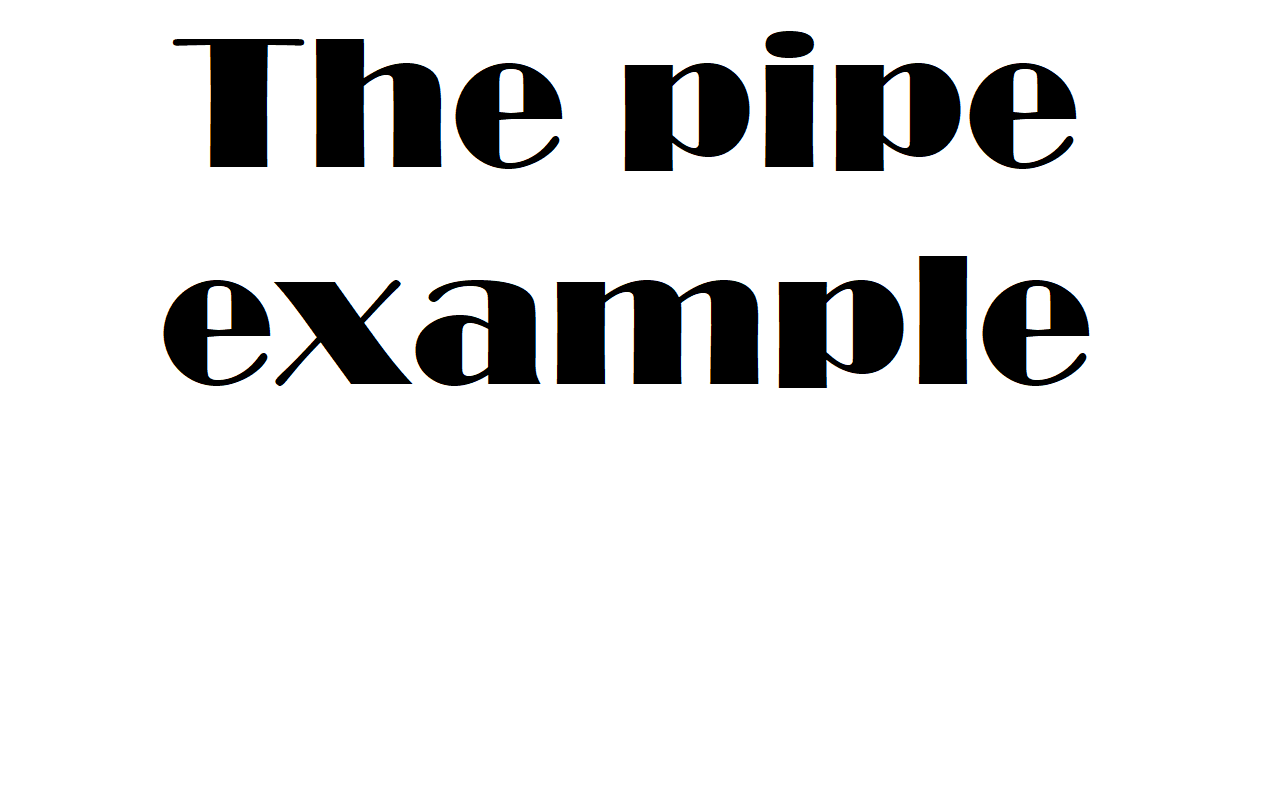} 
		\includegraphics[width=0.32\linewidth]{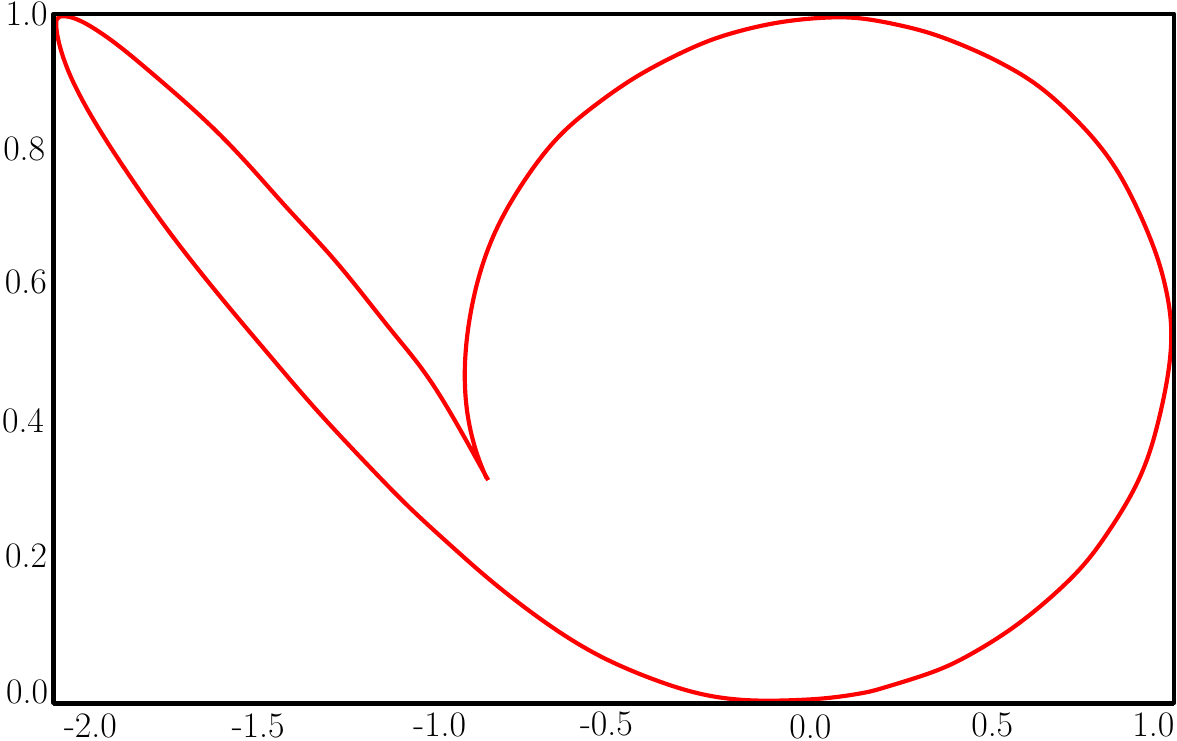}
		\includegraphics[width=0.32\linewidth]{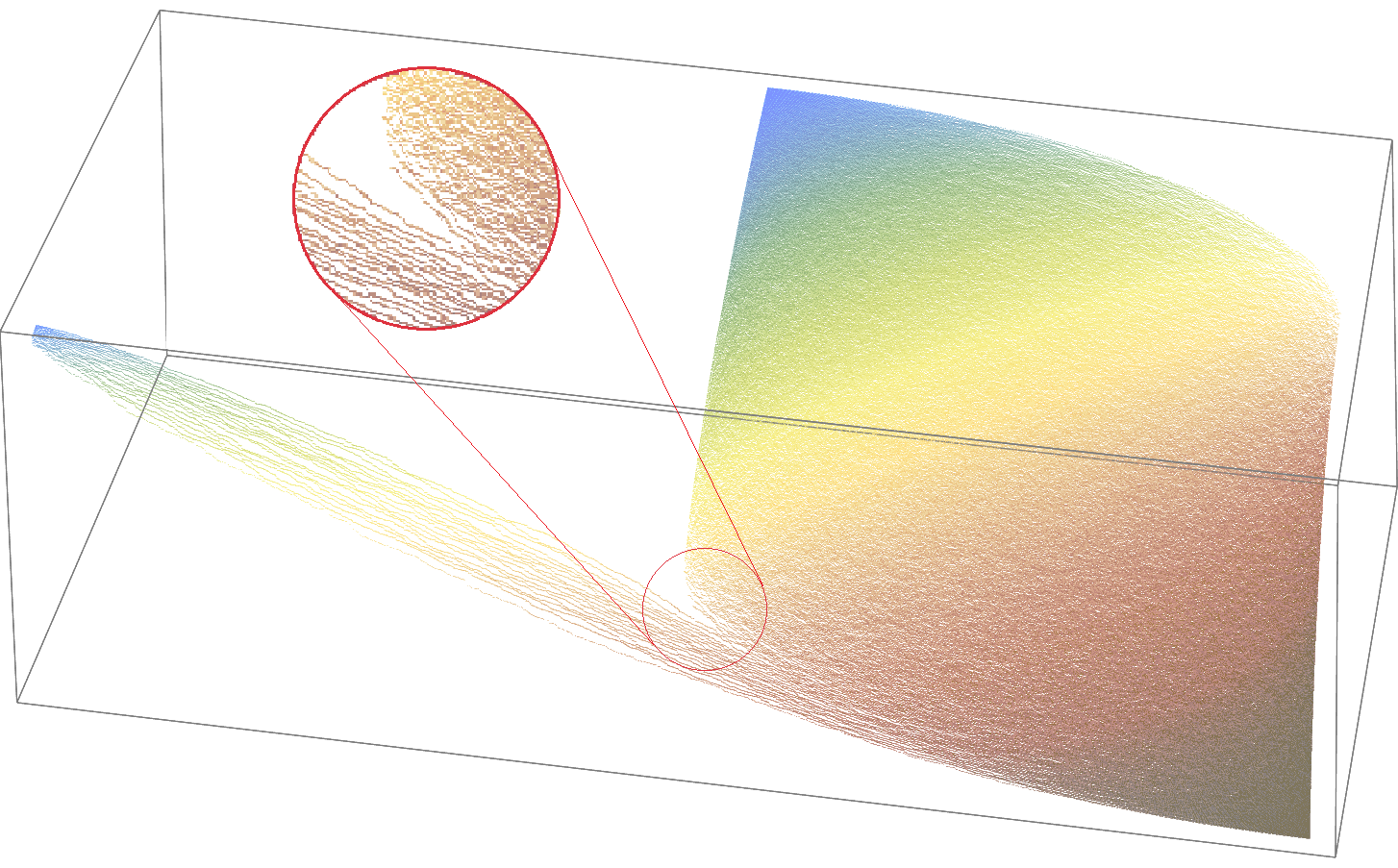}
		\includegraphics[width=0.32\linewidth]{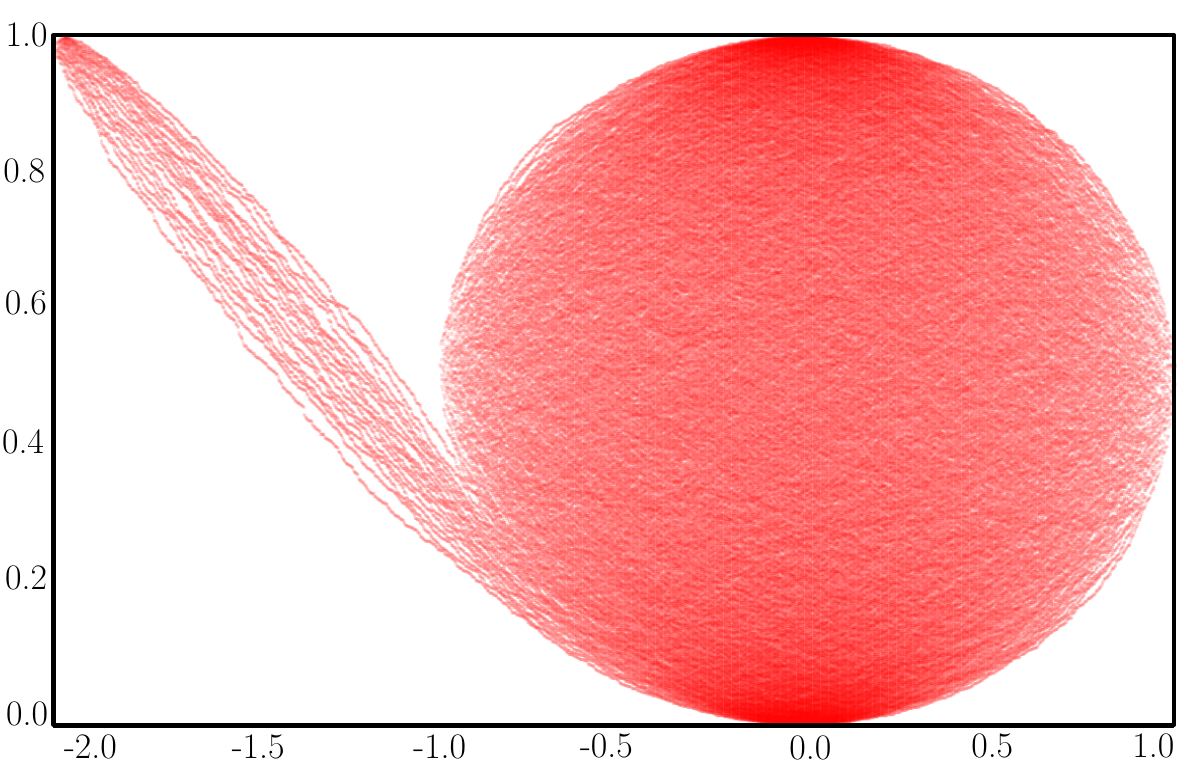}
		\includegraphics[width=0.32\linewidth]{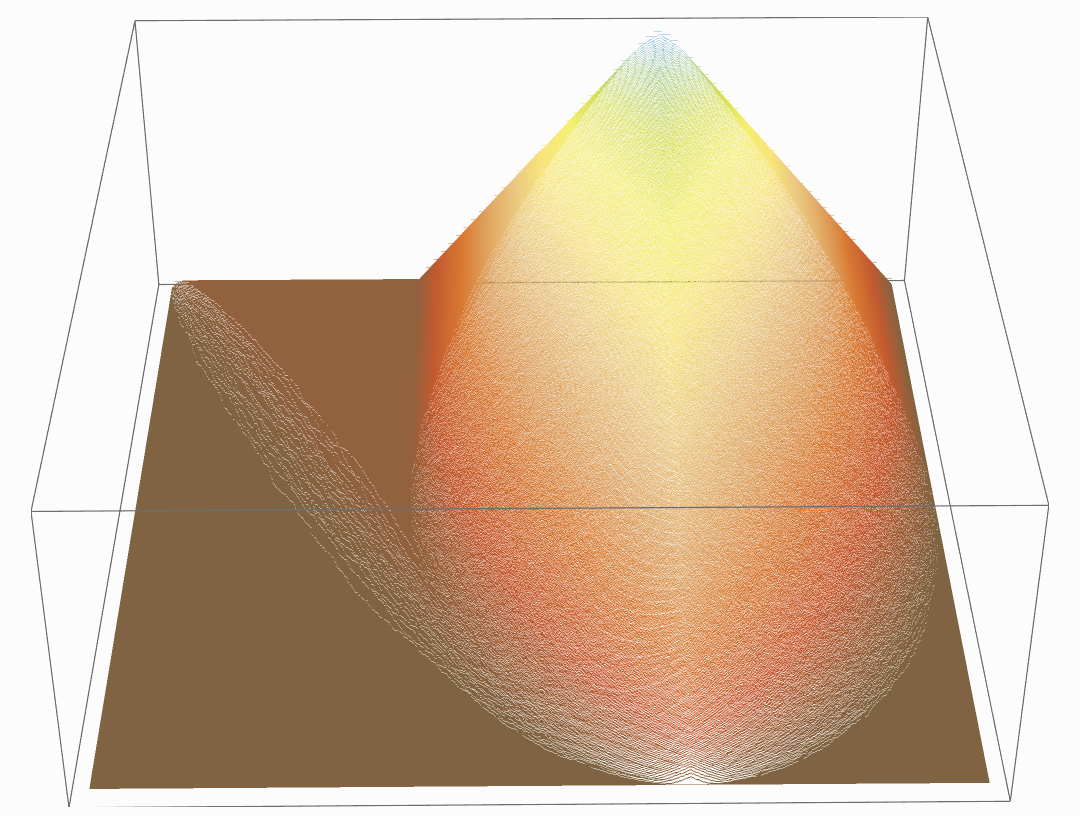}
		\vspace{1mm}
		\captionsetup{width=\linewidth}
		\caption{Figures for the pipe example. \textbf{Top-left:} The Young diagram $\widetilde \la^0$ considered in the pipe example, see \eqref{eq:example_2}. \textbf{Top-right:} The frozen boundary curve of the corresponding liquid region.
			\textbf{Bottom (from left to right):} a uniform random tableau
			of shape $\la_N$ with $N=59400$ boxes ($n=6$) displayed as a discrete surface in a 3D space (brown is used for small values of the surface and blue for large ones); the corresponding bead process $M_{\la_N}$; the corresponding height function $H_{\la_N}$. The red circle in the picture in the left-hand side highlights the discontinuous location of the limiting surface.  \label{fig:Example2}}
	\end{figure}
For both examples, we have computed (using the parametrization from \cref{prop:description_liquid_region} below) the boundary of the liquid region
defined in \cref{defn:liquid_region}. Independently, we have also generated uniform random tableaux of shape $\lambda_N$ for large $N$ (using the Greene--Nijenhuis--Wilf
hook walk algorithm \cite{greene1982probabilistic}). 
The bead processes corresponding to these uniform random tableaux
are also plotted,
and we see that, in both cases, the position of the beads coincides with the liquid region,
as predicted by \cref{thm:limiting_surface_formula}. Finally, we plotted the height functions associated with the bead processes.
\medskip

An essential difference between the two examples is that in the heart example, the interlacing coordinates satisfy Condition~\eqref{eq:cont_cond},
while this is not the case in the pipe example.
From \cref{thm:continuity}, we, therefore, expect to see a continuous limiting surface
in the heart example and not in the pipe example. This is indeed the case,
as we now explain.

In the heart example, the intersection of the liquid region with any vertical line
is connected; in other terms, for every $x\in[\eta\, a_0,\eta\, a_m]$, the function $t \mapsto H^\infty(x,t)$ is first constant equal to $0$, then strictly increasing, and then constant equal to its maximal value.
Therefore, with the notation of \eqref{eq:Tplusminus}, 
we have $T^\infty_-(x,y)=T^\infty_+(x,y)$ for all $(x,y)$ in $D_{\la^0}$ and the limiting surface $T^{\infty}$ is defined and continuous on the whole set $D_{\la^0}$.
Looking at the random tableau drawn as a discrete surface, it is indeed plausible
that it converges to a continuous surface.

In the pipe example, however, we can find some $x_0$ (on the right of $\eta\, a_1=-\frac{3}{\sqrt{11}}\approx-0.9$) so that
the liquid region intersects the line $x_0 \times [0,1]$ in two disjoint intervals.
The function $t \mapsto H^\infty(x,t)$ is then constant between these two intervals,
and the limiting surface $T^\infty$ is discontinuous.
This discontinuity can be observed on the simulation of a uniform random tableau (see the zoom inside the red circle on the left-hand side).

\subsubsection{Characterizing continuity for $L$-shapes}\label{sect:disc_phen}
In numerical simulations, we need to consider extremely unbalanced diagrams $\la^0$ 
(as the one in the pipe example above) to observe a discontinuity in the limit shape.
We will see, however, in this section that such discontinuity occurs for generic $L$-shape diagrams 
$\la^0$.

To this end, let us parametrize $L$-shape Young diagrams as follows.
For $(p,q,r)\in\diamond:=\left\{(p,q,r)\in\Q^3\,|\, r > 0, p\in(-1,r), |p| < q \leq \min\{p + 2, 2r-p\}\right\}$,
we let $\la^0_{p,q,r}$ have interlacing coordinates
\begin{equation}\label{eq:param_int}
	a_0=-1\quad<\quad b_1=\frac{p + q - 2}{2}\quad<\quad a_1=p\quad<\quad b_2=\frac{p - q + 2r}{2}\quad <\quad a_2=r.
\end{equation} 
Then the inner corner of $\la^0_{p,q,r}$ has coordinates $(p,q)$
and $\eta=\frac{2}{\sqrt{p^2 - 2 p (-1 + r) + q (2 - q + 2 r)}}$;
see \cref{fig:phase_diagram-param} for an illustration. 
\begin{figure}[ht]
	\includegraphics[height=9cm]{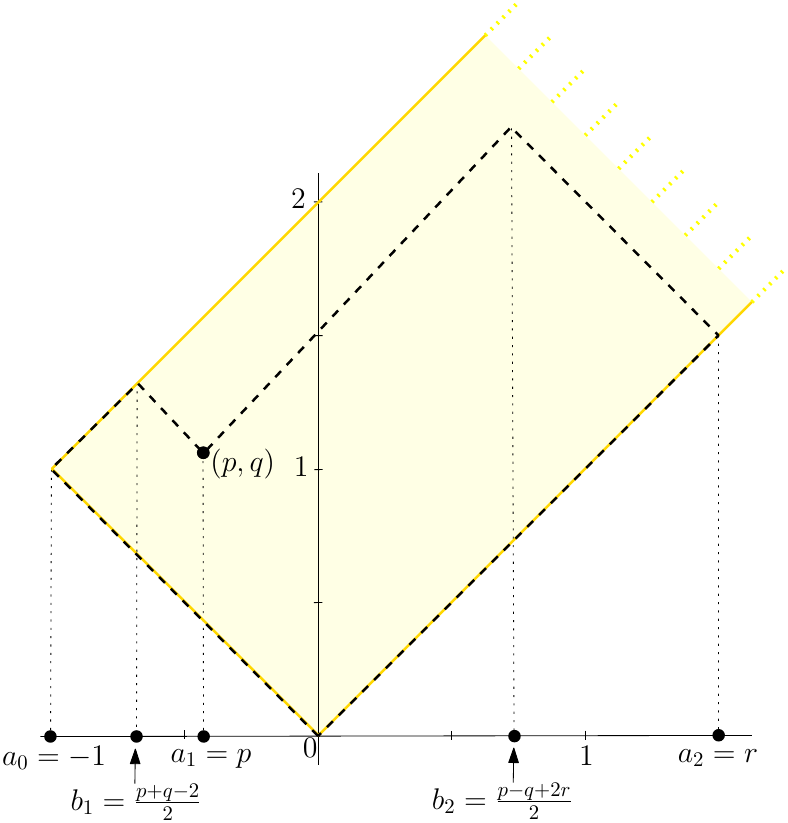}
	\captionsetup{width=\linewidth}
	\caption{The Young diagram $\la^0_{p,q,r}$. In yellow, the region $\diamond$. \label{fig:phase_diagram-param}}
\end{figure}

Given the diagram $\la^0_{p,q,r}$, we denote by $D_{p,q,r}=D_{\la^0_{p,q,r}}$ and $T^\infty_{p,q,r}$
the corresponding domain and limiting surface, as defined in \cref{ssec:limiting-surface}.
The following results characterizes the triplets $(p,q,r)$, for which
$T^\infty_{p,q,r}$ is defined and continuous on the whole domain $D_{p,q,r}$.
\begin{proposition}\label{prop:phase_diagram}
	The following results hold:
	\begin{itemize} 
		\item If $r=1$ then the surface $T^\infty_{p,q,1}$ is defined and 
          continuous on $D_{p,q,1}$
          if and only if \begin{equation*}
			 p=0\qquad \text{or}\qquad q=2-\sqrt{2-p^2}=:Q(p).
		\end{equation*}
		\item If $r\neq 1$, then the surface $T^\infty_{p,q,r}$ is 
          defined and continuous on $D_{p,q,r}$ if and only if 
		\begin{equation*}
			p\leq 0 \text{ and } q=Q^+_r(p)  \qquad \text{or}\qquad  p\geq r-1 \text{ and } q=Q^-_r(p),
		\end{equation*}
		where
		\begin{align*}
			Q^\pm_r(p)=1 + r \pm \frac{ \sqrt{(1 + p - r) (1 + 2 p - r) ( 
					p r + (1 + r)^2-p - 2 p^2)}}{1 + 2 p - r}.
		\end{align*}
	\end{itemize}
\end{proposition}

\begin{remark}
	In the case $r=1$, we note that there is a dense set (dense in the continuous curve $\set{(x,y)\in\R^2 | y=Q(x),x\in(0,1)}$) of rational solutions $(p,q)$ to the equation $q=Q(p)$. They are parametrized by
	\begin{equation*}
		\Set{(p,q)\in\Q^2 | q=Q(p)}=\Set{\left(\frac{s (2 r + s) - r^2}{s^2 + r^2},1 + \frac{2 s (s - r)}{s^2 + r^2}\right) |  (s,r)\in\Q^2,sr\neq 0}.
	\end{equation*}
	These results can be obtained in the same way as one obtains the parametrization for the rational points of the unit circle, noting that $x^2+y^2=1$ if and only if
	$(x-y)^2+(x+y)^2=2$. A similar remark holds also in the case $r>1$.
\end{remark}

\begin{figure}[ht]
	\includegraphics[height=6.5cm]{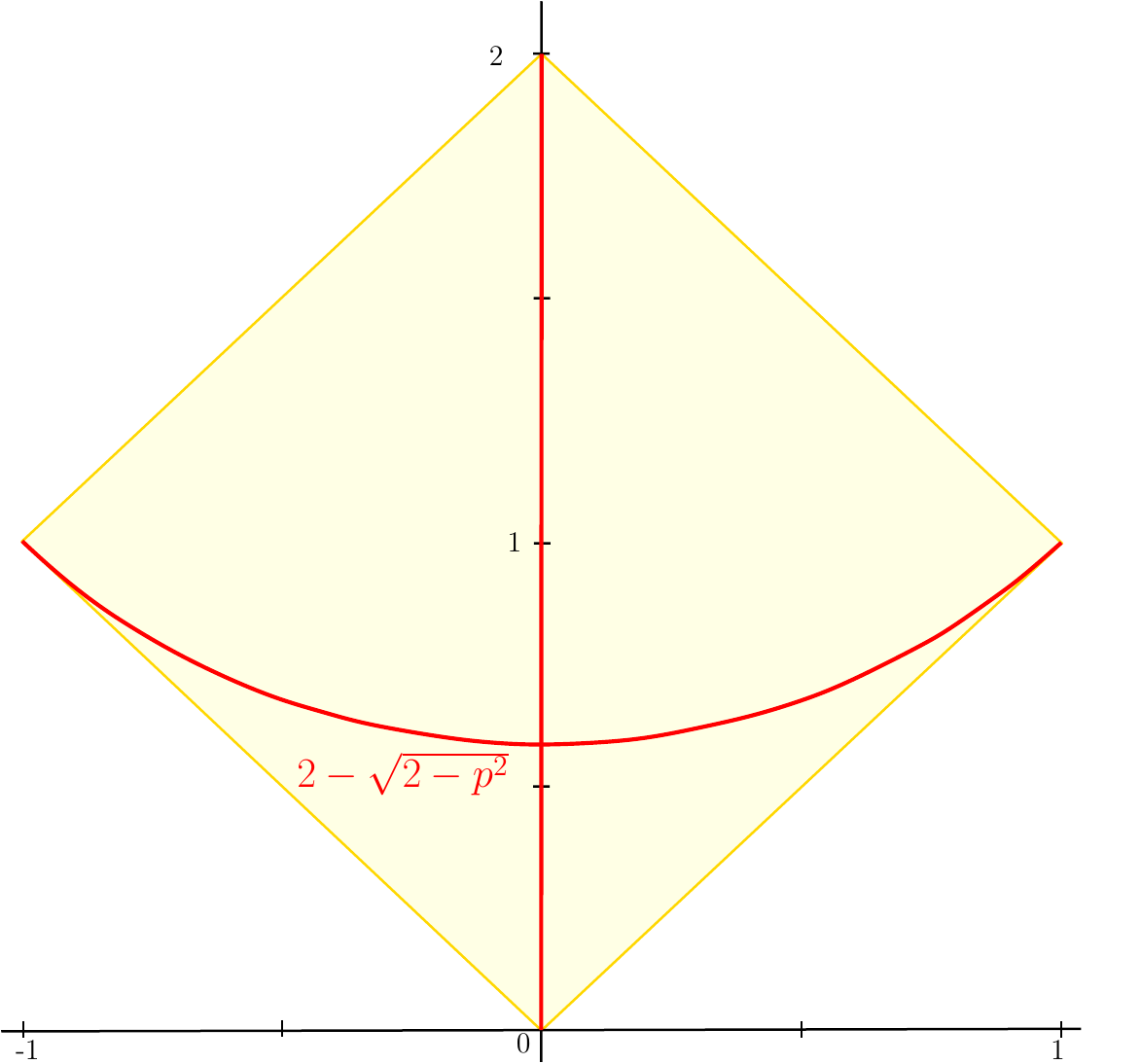}
	\qquad\qquad
	\includegraphics[height=6.5cm]{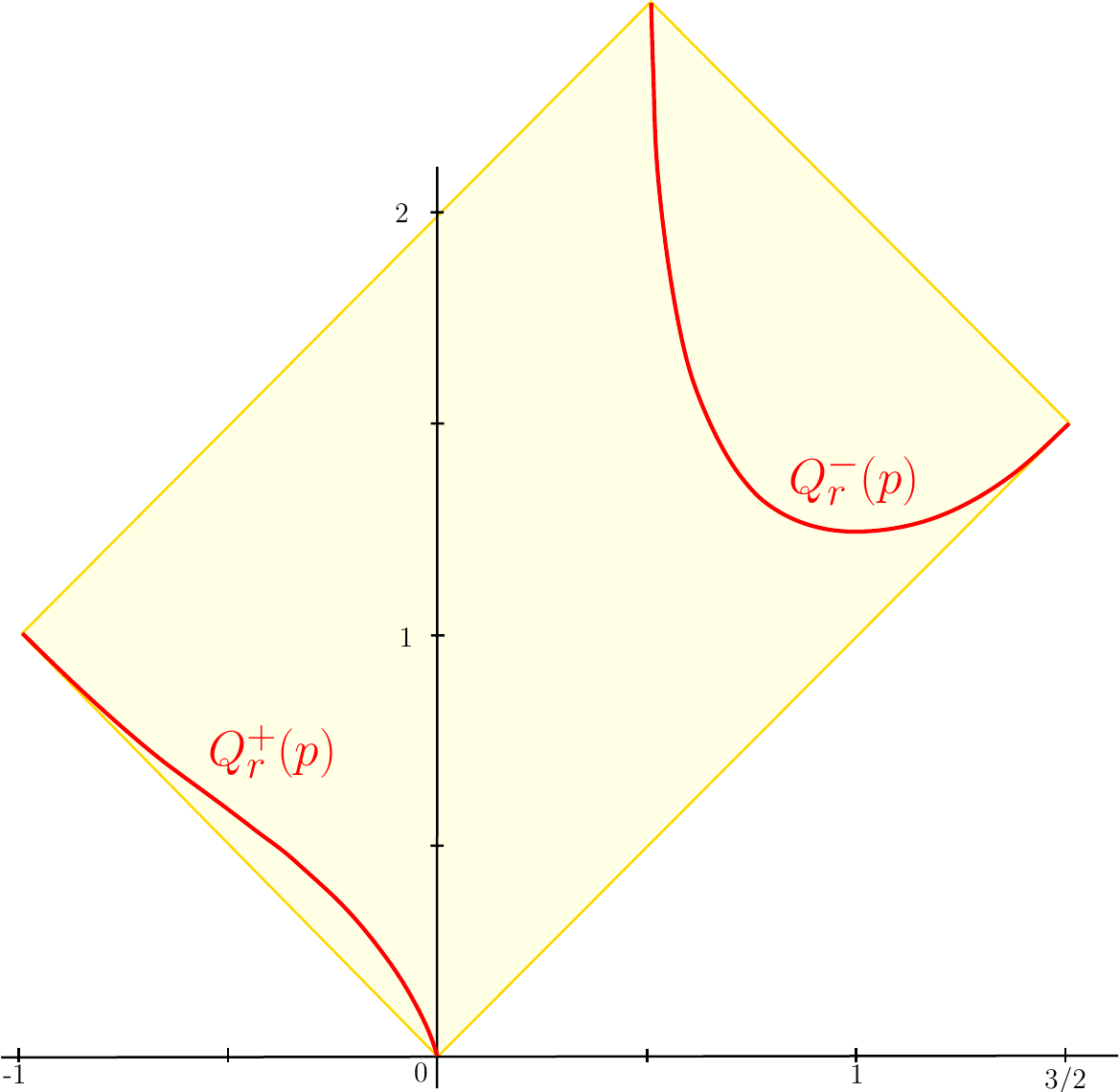}
	\captionsetup{width=\linewidth}
	\caption{\textbf{Left:} Here we fixed the parameter $r=1$. The yellow region is the region $\left\{(p,q)\in\R^2\,|\, p\in(-1,1), |p| < q \leq 2-|p|\right\}$. \cref{prop:phase_diagram} states that the limiting surface $T^\infty_{p,q,1}$ is continuous if and only if $(p,q)$ lies on one of the two red curves. \textbf{Right (case $r>1$):} Here $r=3/2$. The yellow region is the region $\left\{(p,q)\in\R^2\,|\, p\in(-1,3/2), |p| < q \leq \min\{p + 2, 2r-p\}\right\}$. \cref{prop:phase_diagram} states that the limiting surface $T^\infty_{p,q,r}$ is continuous if and only if $(p,q)$ lies on one of the two red curves.
	\label{fig:phase_diagram}}
\end{figure}

\subsection{Local limit results}

We now present our local limit results. We first introduce the random infinite bead process mentioned before.

\subsubsection{The random infinite bead process}
\label{sec:bead-cedric}
The second author \cite{boutillier2009bead} constructed a two-parameter family of ergodic Gibbs measures on the set of infinite bead configurations (recall the terminology introduced in \cref{sect:yt_pyt_bc}). 
These measures are constructed as limits of some dimer model measures on some bipartite graphs when certain weights degenerate. These Gibbs measures are shown to be determinantal point processes; we refer the reader to \cite{hough2009determinantal} for background on determinantal point processes. 
In particular, the following result is a slight reformulation of \cite[Theorem 2]{boutillier2009bead}, see \cref{rem:kernel_Jalphabeta} below.

\begin{theorem}[\cite{boutillier2009bead}]\label{thm:bead_kernel_cedric}
	Let $(\a,\be)$ be in $\mathbb R_+ \times (-1,1)$.
	There exists a determinantal point process $M_{\a,\be}$ on $\mathbb Z \times \mathbb R$
	with correlation kernel
	\begin{equation}\label{eq:BeadKernel}
		J_{\a,\be}\big((x_1,t_1),(x_2,t_2)\big) =\begin{cases}
			\frac{\a}{2\pi} \int_{[-1,1]} \E^{\I (t_1-t_2) \a u}
			\left( \be + \I u \sqrt{1-\be^2} \right)^{x_2-x_1} \dd{u} &\text{ if}\quad x_2 \ge x_1;\\
			-  \frac{\a}{2\pi} \int_{\mathbb R \setminus [-1,1]} \E^{\I (t_1-t_2) \a u}
			\left( \be + \I u \sqrt{1-\be^2} \right)^{x_2-x_1} \dd{u} &\text{ if}\quad x_2 < x_1.
		\end{cases}
	\end{equation}
\end{theorem}

\begin{definition}[Random infinite bead process]\label{defn:bead_process}
	The point process $M_{\a,\be}$ is called the {\em random infinite bead process} of intensity $\a$ and skewness $\be$.
\end{definition}
These bead processes have nice properties. First, they are translation invariant
in both directions and induce the uniform measure on bead configurations
on every finite domain (see \cite{boutillier2009bead} for a precise statement).
Moreover, they appear as limits of dimer tilings \cite{boutillier2009bead,fleming2012finitization} and of the eigenvalues of imbricated GUE matrices (called GUE corners process), see \cite{adler2014minor}.
We also note that:
\begin{itemize}
 	\item the expected number of beads in a portion of a thread of length 1 is $\frac{\alpha}{\pi}$;
 	\item the expected ratio
between the {vertical} distance of a bead $b$ from its neighbor on the left and
{below}\footnote{The interpretation of the parameter $\beta$ in the text of
  \cite{boutillier2009bead}, where \emph{above} is used instead of \emph{below},
  is wrong because of a missing minus sign inside the
$\arccos$ in Eq.\ (29) there.}
it, and
the distance between $b$ and the successive bead on the same thread, is $\frac{\arccos(\beta)}{\pi}$; see \cite[Eq.\ (29)]{boutillier2009bead}.
 \end{itemize}

\begin{remark}\label{rem:kernel_Jalphabeta}
	We note that \cref{thm:bead_kernel_cedric} is a simple reformulation of \cite[Theorem 2]{boutillier2009bead}. Indeed, for $\a=1$, the above kernel $J_{1,\beta}$ correspond to $J_\gamma$ in \cite{boutillier2009bead} for $\gamma=\beta$.\footnote{In the present paper, we use $\beta$ instead of $\gamma$ to avoid conflicts of notation with integration paths.}
	For general $\a$, the bead process $M_{\a,\be}$ is simply obtained by applying a dilatation with scaling factor $\frac{1}{\alpha}$ to $M_{1,\be}$. Then the kernel of $M_{\a,\be}$ is given by $J_{\a,\be}\big((x_1,t_1),(x_2,t_2)\big) = \alpha J_{1,\be}\big((x_1,\a t_1),(x_2,\a t_2)\big)$, which is consistent with the statement above.
\end{remark}

\subsubsection{Local limit result for the bead process}

Given a Young diagram $\la^0$, we now fix $(x_0,t_0)$ in $[\eta\, a_0, \eta\, a_m]\times [0,1]$.
In the rest of the paper, we assume that $x_0 \sqrt{|\la^0|}$ is an integer\footnote{This is not as restrictive as it might seem. Indeed, since we are going to look at the limit when $n\to \infty$, 
	as soon as $x_0 \sqrt{|\la^0|}$ is a rational number, it is possible to replace $\la^0$ with a rescaled version of $\la^0$ so that $x_0 \sqrt{|\la^0|}$ is an integer. It is also likely that,
	for a generic value of $x_0\in\mathbb{R}$, one can obtain the same result using a sequence $x_0^{(N)}$
	with $x_0^{(N)}=x_0+O(N^{-1/2})$ so that  $x_0^{(N)}\sqrt N$ is an integer.
	The proof would however need some adjustement;
	indeed, a modification of order $O(N^{-1/2})$ of $x_0$ would modify for example
	the function $h$ appearing in Eq.~\eqref{eq:equiv_integrand}.
	Therefore, since
	the proof is already rather technical as is, we have not pursued this direction.}
so that $x_0 \sqrt{N}$ is always an integer.
We look at the random bead process  $M_{\la_N}$ in a window of size $O(1) \times O(1/\sqrt{N})$
around $(x_0 \sqrt{N},t_0)$.
To this end, we define
\begin{equation}\label{eq:defn_of_Mla}
	\widetilde{M}_{\la_N}^{(x_0,t_0)} = \Set{ (x,t) \in \Z \times \R |  \big(x_0 \sqrt{N} + x,t_0 + \tfrac{t}{\sqrt{N}} \big) \in  M_{\la_N} }.
\end{equation}

In the next result, we consider the standard topology for point processes (see \cref{sec:bead} for further details). Recall also the definition of the liquid region from \cref{defn:liquid_region},
and the subsequent definitions of $\a(x_0,t_0)$ and $\beta(x_0,t_0)$.

\begin{theorem}[Local limit of the bead process associated with a Poissonized Young tableau]\label{thm:cv_bead_process}
Fix a Young diagram $\la^0$ and consider a sequence $T_N:\la_N\to [0,1]$ of uniform random Poissonized Young tableaux of shape $\la_N$. We choose $x_0\in [\eta\, a_0, \eta\, a_m]$ and $t_0 \in [0,1]$, and we let $n$ go to infinity (so $N$ goes to infinity).
\begin{itemize}
	\item If $(x_0,t_0)$ is in the liquid region, then the random bead process $\widetilde{M}_{\la_N}^{(x_0,t_0)}$ in \eqref{eq:defn_of_Mla} converges in distribution to the random infinite bead process of intensity $\alpha(x_0,t_0)$ and skewness $\beta(x_0,t_0)$.
	
	\item If $(x_0,t_0)$ is in the complement of the liquid region, then the random bead process $\widetilde{M}_{\la_N}^{(x_0,t_0)}$  converges in probability to the empty set.
\end{itemize}
\end{theorem}

Note that the second case contains the case where $(x_0,t_0)$ lies on the boundary of the liquid region
(recall from \cref{prop:description_liquid_region} that the liquid region is an open set).
See \cref{sect:future_work} for a discussion on this case.

\subsubsection{Local limit result for the tableau}
\label{ssec:local_tableau_intro}

Our next result describes the local limit of uniform random Poissonized Young tableaux. We need some preliminary definitions.

A \emph{marked standard Young tableau} is a triplet $(\la,T,(x,y))$ where $(\la,T)$ is a standard Young tableau of shape $\la$ and $(x,y)$ are the coordinates of a distinguished box in $\la$. A \emph{marked Poissonized Young tableau} is defined analogously.

We introduce a one-parameter family of random infinite standard Young tableaux directly constructed from the random infinite bead process. These will be our candidate local limits for random uniform Poissonized Young tableaux. 
Let $M_{\be}:=M_{1,\beta}$ be the random infinite bead process of intensity $1$ and skewness $\be\in (-1,1)$ introduced in \cref{defn:bead_process}. Recall that each bead in $M_{\be}$ has a position $(x, h)$, where $x$ denotes the thread number where the bead is positioned, and $h$ represents the height of the bead on this $x$-thread. There is a natural bijection between the set of beads in $M_{\be}$ and the set $\set{(x,y)\in\Z^2 | x+y \text{ is odd}}$: we label by $(0,1)$ the first bead in the zero-thread with positive height, and then we label all the other beads as shown in the left-hand side of \cref{fig:beads_to_SYT2}.

\begin{figure}[ht]
	\includegraphics[scale=0.8]{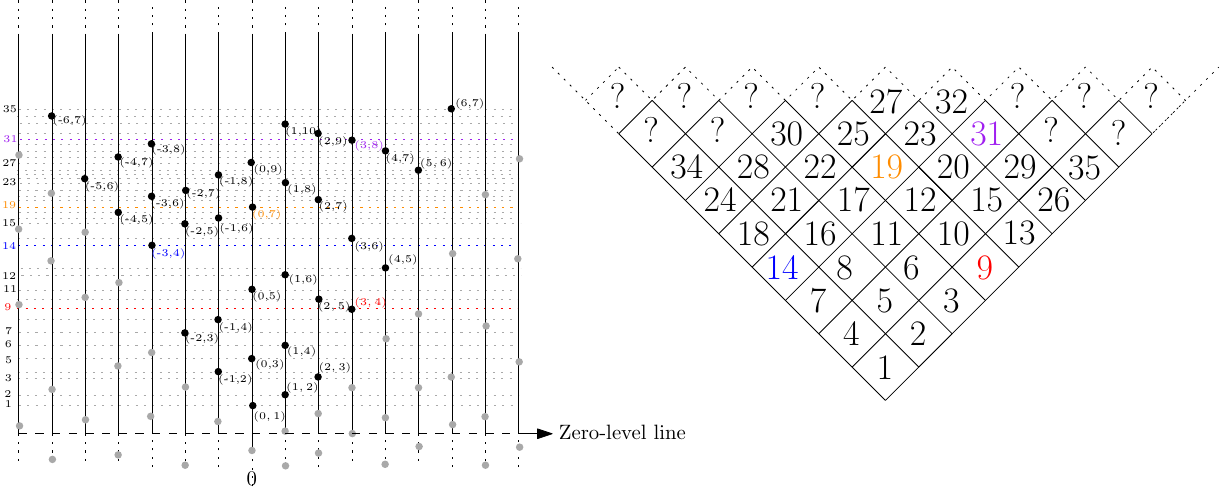}
	\captionsetup{width=\linewidth}
	\caption{\textbf{Left:} A section of a sampling of the infinite bead process $M_{\be}$. We only show the labels of the beads with a label $(x,y)\in \diamond_{\infty}$. These beads are highlighted in black. The ranking of some black beads is shown on the left-hand side of the diagram.  \textbf{Right:} The corresponding random infinite standard Young tableau $(\diamond_{\infty},R_{\be})$. On both sides of the picture, we colored some elements in correspondence to help the reader to compare the two pictures.}
	\label{fig:beads_to_SYT2}
\end{figure}

Given a bead in $M_{\be}$ labeled by $(x,y)$ and with position $(x,h)$, we denote by $H_{\be}(x,y)=h$ the height of the bead labeled by $(x,y)$. 
The collection of heights $\left\{H_{\be}(x,y)\right\}$ with $(x,y)\in\set{(x,y)\in\Z^2 | x+y \text{ is odd and } y> {|x|}}$ can be ranked increasingly, starting from the smallest one to which we assign rank 1 (see the left-hand side of the bead diagram in \cref{fig:beads_to_SYT2}; we prove in \cref{prop:inifinite_tableau_well_defined} below
that the heights $\left\{H_{\be}(x,y)\right\}$ are all distinct and have no accumulation points a.s.). Given a bead $(x,y)$, we denote its ranking by $R_{\be}(x,y)$.

\begin{definition}[Random infinite standard Young tableau]\label{defn:isyt}
	The \emph{random infinite standard Young tableau} of skewness $\be\in (-1,1)$ is the random infinite standard Young tableau $(\diamond_{\infty},R_{\be})$, where $\diamond_{\infty}$ is the infinite Young diagram formed by all the boxes at positions $(x,y)\in \Z^2$ such that $x+y$ is odd and $y> |x|$. 
\end{definition}

See the right-hand side of \cref{fig:beads_to_SYT2} for an example of random infinite standard Young tableau $(\diamond_{\infty},R_{\be})$.

\begin{remark}
It is, of course, possible to do the same construction starting from a bead process
of general intensity $\alpha$ instead of $1$. However, a bead process of intensity $\alpha$
is obtained from one of intensity $1$ by rescaling the vertical axis. Since the
construction of the tableau only involves comparing the heights of various beads and not their actual heights, such a rescaling does not modify the law of the resulting
tableau. Hence the skewness $\beta$ is indeed the only relevant parameter.
\end{remark}

In the following result, we endow the space of marked standard Young tableaux with the topology induced by the \emph{local convergence}, formally introduced in \cref{defn:local_conv}. Roughly speaking, this topology says that a (deterministic) sequence of marked Young tableaux $(\la_n,T_n,(x_n,y_n))$ converges to an infinite Young tableau $(\diamond_{\infty},R)$, if the values of the boxes in $(\la_n,T_n)$ contained in any finite neighborhood above the marked box at $(x_n,y_n)$ are eventually {\em in the same relative order} as the values of the boxes of $(\diamond_{\infty},R)$ contained in the same finite neighborhood above the box at $(0,0)$.

\begin{corollary}[Local limit for the uniform Poissonized Young tableau; corollary of \cref{thm:cv_bead_process}]\label{corol:local_lim_standard Young tableau}
	Fix a Young diagram $\la^0$ and $(x_0, t_0)$ in the corresponding liquid region $L$.
	Consider a sequence $T_N$ of uniform random Poissonized Young tableaux of shape $\la_N$.
	We denote by $\oblong_N$ the box corresponding to
	the first bead of $M_{\la_N,T_N}$ above $t_0$ in the $x_0 \sqrt{N}$-th thread.
	
	Then, as $N\to \infty$, the sequence of marked standard Young tableaux $(\la_N,T_N,\oblong_N)_N$ locally converges in distribution
	to the random infinite standard Young tableau $(\diamond_{\infty},R_{\be})$ of 
	 skewness $\beta=\beta(x_0,t_0)$.
\end{corollary}

\subsection{The complex Burger's equation and a PDE for the limiting height function}
\label{sec:Burger}
A standard fact when describing limit shapes and local limits via determinantal point
processes is that the solution $U_c(x,t)$ of the critical equation \eqref{eq:critical_intro},
sometimes called {\em complex slope}, satisfy some partial differential equation (PDE).
This PDE is normally related to the so-called {\em complex Burger's equation},
see \cite{KenyonOkounkov2007Burgers,petrov2014tilings}.
In our model, we get the following PDE .

\begin{proposition}\label{prop:Burgers_eq} In the sequel, the indices ${\_}_x$ and ${\_}_t$ denote partial derivatives with respect to the variables $x$ and $t$.
	The solution $U_c=U_c(x,t)$ of the critical equation \eqref{eq:critical_intro}
    satisfies the following PDE in the liquid region:
	\begin{align}
      \label{eq:burgers}
		(U_c)_t+ U_c\,\frac{(U_c)_x+1}{1-t}&=0.
    \end{align}
\end{proposition}
The proposition is proved in \cref{proof:burger}.
This PDE yields a PDE for the complex-valued function 
$\mathcal H^\infty(x,t):=\frac1{\pi} \int_0^t \frac{U_c(x,s)}{1-s} \dd{s}$.
Namely,
\begin{equation}\label{eq:complex_diff_height}
		 \mathcal{H}^\infty_{tt}+\pi\,\mathcal{H}^\infty_t \,\mathcal{H}^\infty_{xt}=0.
	\end{equation}
We recall from \cref{thm:limiting_surface_formula} that the limiting height function $H^\infty(x,t)$ is simply obtained as 
$H^\infty(x,t) = \Im \mathcal{H}^\infty(x,t)$.
Taking the imaginary part in \cref{eq:complex_diff_height}, we get
\begin{equation}
\label{eq:mixed-PDE-H}
 {H}^\infty_{tt} +\pi\, \widetilde H^\infty_t {H}^\infty_{xt} + \pi {H}^\infty_t \widetilde H^\infty_{xt}=0,
 \end{equation}
where $\widetilde H^\infty(x,t) = \Re \mathcal{H}^\infty(x,t)$.
To get a PDE involving only the height function $H^\infty(x,t)$, we need an extra ingredient.
At least on an informal level, this extra ingredient is provided
by a combination of the local convergence to the infinite bead process in \cref{thm:cv_bead_process} and the scaling limit result in \cref{thm:limiting_surface_formula}, as we explain now.

Recall from \cref{sec:bead-cedric} that for the infinite bead process $M_{\alpha,\beta}$, the expected ratio
between the {vertical distance} of a bead $b$ from its neighbor on the left and {below} it, and
the distance between $b$ and the successive bead on the same thread, is $\frac{\arccos(\beta)}{\pi}$. But successive beads on the same thread are at expected distance $\pi/\alpha$ so that the expected {vertical distance} between a bead $b$ and its neighbor on the left and {below} it is $\arccos(\beta)/\alpha$. Using this fact for $M_{\alpha,\beta}$, and approximating the bead process $\widetilde{M}_{\la_N}^{(x,t)}$ with the infinite bead process $M_{\alpha,\beta}$ with parameters $\alpha=\alpha(x,t)$ and $\beta=\beta(x,t)$ (\cref{thm:cv_bead_process}), we get that for large $N$ the
height function $ H_{\lambda_N}$ around $(x,t)$ with $x>0$ satisfies
\begin{equation}\label{eq:approxslope}
	H_{\lambda_N}(x\sqrt N,t) \approx H_{\lambda_N}\left(x\sqrt N -1, t{-} \frac{\arccos(\beta)}{\alpha \sqrt N} \right),
\end{equation}
where we used the fact that when $x>0$, the first bead on the $(x\sqrt N)$-thread is above the first bead on the $(x\sqrt N-1)$-thread.
Iterating this argument $\eps \sqrt N$ times, and multiplying both sides by $1/\sqrt{N}$, we get
\[ \frac{1}{\sqrt{N}}\, H_{\lambda_N}\left(x\sqrt N,t \right) \approx \frac{1}{\sqrt{N}} H_{\lambda_N}\left((x-\eps)\sqrt N , t-\eps\, \frac{\arccos(\beta)}{\alpha} \right).\]
Finally, using \cref{thm:limiting_surface_existence} and sending $N$ to infinity, we deduce that
\[ H^\infty(x,t) = H^\infty \left(x-\eps, t - \eps\,\frac{\arccos(\beta)}{\alpha} \right).\]
This implies $H^\infty_x + \frac{\arccos(\beta)}{\alpha} H^\infty_t=0$ for $x>0$. When $x\leq 0$, we need to use the fact that the first bead on the $(x\sqrt N)$-thread is now \emph{below} the first bead on the $(x\sqrt N-1)$-thread.
As a consequence, \eqref{eq:approxslope} is replaced by 
\begin{equation*}
	H_{\lambda_N}(x\sqrt N,t) \approx H_{\lambda_N}\left(x\sqrt N -1, t+ \frac{\pi-\arccos(\beta)}{\alpha \sqrt N} \right),
\end{equation*}
where we also used the fact that the expected {vertical distance} between a bead $b$ and its neighbor on the left 
and {\emph{above}} it is $(\pi-\arccos(\beta))/\alpha$.
Hence, for $x \le 0$ one gets $H^\infty_x - \frac{\pi - \arccos(\beta)}{\alpha} H^\infty_t=0$.
The equations obtained for the cases $x>0$ and $x \le 0$ can be grouped in a single formula as follows:
\begin{equation*}
H^\infty_x + \left(\frac{\arccos(\beta)}{\alpha}-\frac{\pi \delta_{x\leq0}}{\alpha}\right) H^\infty_t=0.
\end{equation*}
Since $H^\infty_t= \alpha/\pi$ by \cref{thm:limiting_surface_formula}, the previous equation simplifies to
$\beta=\cos( \pi \delta_{x\leq0} -\pi H^\infty_x)$. Recalling from \eqref{eq:parm_liq} that $\be=\frac{\Re U_c}{|U_c|}$, we deduce that 
$\Arg(U_c)=\pi \delta_{x\leq0} -\pi \, H^\infty_x$.
Recalling also from \eqref{eq:parm_liq} that $\alpha=\frac{\Im U_c}{1-t}$,  we conclude that  
\begin{equation*}
	\widetilde H^\infty_t= \frac{\Re U_c}{\pi(1-t)} = \frac{\alpha}{\pi \tan(\Arg(U_c))} =\frac{H^\infty_t}{\tan(\pi\delta_{x\leq0} -\pi \, H^\infty_x)}=-\,\frac{H^\infty_t}{\tan(\pi \, H^\infty_x)},
\end{equation*}
and
\begin{equation*}
	\widetilde H^\infty_{xt}=-\frac{ H^\infty_{xt}}{\tan(\pi H^\infty_{x})}+\pi  H^\infty_{t} H^\infty_{xx}\left(1+\frac{1}{\tan^2(\pi H^\infty_{x})}\right).
\end{equation*}
Combining the latter two equations with \eqref{eq:mixed-PDE-H}, we get the desired PDE for the height function:
\begin{equation}
\label{eqPDE-H}
 {H}^\infty_{tt} -2\pi\, \frac{H^\infty_t {H}^\infty_{xt}}{\tan(\pi\, H^\infty_x)}  +\pi^2 (H^\infty_t)^2 H^\infty_{xx}
 \left(1 + \frac1{\tan^2(\pi\, H^\infty_x)}\right)=0.
 \end{equation}
Such a PDE for the limiting surface was obtained by Sun as a consequence of its variational principle \cite[Eq.~(28)]{sun2018dimer}. Note that our equation is slightly different from \cite[Eq.~(28)]{sun2018dimer} since we do not use the same convention for the height function ${H}^\infty$.
 
We do not know if this heuristic argument can be transformed into an actual proof of \eqref{eqPDE-H}; we did not pursue in this direction because of the earlier appearance of the equation in the work of Sun.

\subsection{Methods}
As mentioned in the background section,
our starting point is the determinantal structure of
the bead process associated to a
uniform Poissonized tableau of given shape \cite{GR19}.
The local limit theorem for bead processes around $(x_0 \sqrt N,t_0)$ (\cref{thm:cv_bead_process}) is then obtained
by an asymptotic analysis of the kernel in the appropriate regime.
Since the determinantal kernel is expressed in terms of a double contour integral,
such asymptotic analysis is performed via saddle point analysis,
following the method used, e.g., in the papers
\cite{okounkov2003SchurProcess,borodin2008unitary,petrov2014tilings}.

The saddle points are precisely the two non-real roots of the critical equation~\eqref{eq:critical_intro}
if they exist (liquid region), or some of its real roots when all roots are real (frozen region).
In the case of the frozen region, the construction of the contours
depend on the relative position of $(x_0,t_0)$ with respect to the liquid region.
The case where we have an alternation of liquid and frozen phases along a line $\{x_0\} \times [0,1]$ (leading to a discontinuity of the surface) is particularly interesting
since one has to split one of the integration contours into two new distinguished integration contours
(\cref{ssec:frozen-intermediate}).
Such alternation of liquid and frozen phases has appeared in certain models of random tilings,
  such as 
  plane partitions with 2-periodic weights~\cite[Figure 3]{mkrtchyan2014}
  or
  pyramid partitions~\cite[Figure~1.1]{li_vuletic2022},
  and in a different context
(asymptotic representation theory of the unitary group) in \cite[Theorem 4.6]{borodin2008unitary}.

From the local limit theorem, 
we know that the parameter $\alpha(x_0,t_0)$ controls
 the local density of beads around $(x_0,t_0)$.
 It is, therefore, not surprising that the limiting height function is obtained as an integral
 of this local density. Formally, we use a special case of the convergence
 of the determinantal kernel established for the local limit theorem 
 to prove our formula for the limiting height function.
 The relation between local limit results via determinantal point processes,
 and limit shape results were remarked in other contexts,
 see e.g.~\cite[Remark 1.7]{BorodinOkounkovOlshanski2000} 
 or~\cite[Section 3.1.10]{okounkov2003SchurProcess}, but not formally used
 as we do in the present paper.

\subsection{Future work}\label{sect:future_work}
Here are a few directions that could be interesting to investigate.
\begin{enumerate}
\item We have limited ourselves here to a sequence of diagrams $\la_N$,
obtained as the dilatation of a base diagram $\la^0$. It would be also
interesting to consider more general sequences of diagrams with a 
continuous limit shape $\omega(x)$ in the sense of \cite{Biane1998,IvanovOlshanski2002}.
In this case, the critical equation is not polynomial anymore, but it should rewrite as
$$1-t = U\,G_{\mu_\omega}(U+x),$$
where $G_{\mu_\omega}(z)=\int_\R \frac{1}{z-s}\,\mu(\dd{s})$ is the Cauchy transform of the transition measure of the limit shape (see \cite{Biane2003}). 

\item We believe that the discontinuities observed in this paper
are created by the non-smoothness of the limit shape $\omega(x)$.
We conjecture that if $\omega$ is smooth (probably $\mathcal{C}^1$ is enough),
then the limiting surface is continuous. An interesting question is also
to find a general criterium for the limiting surface to be continuous, extending \cref{thm:continuity}.
The latter suggests that there should be some identity to check at each 
singular point of the limit shape $\omega$; but at the moment we do not know what such
criterium should be.

\item In \cref{corol:local_lim_standard Young tableau}, we consider the local
limit of a Young diagram around a box $\Box_N$, whose coordinates are random
(they depend on the tableau $T_N$). Considering the limit around a fixed box in the Young diagram would be more satisfactory. This corresponds, however, to a random
position in the bead process.
One potential strategy to attack this problem would be to refine our analysis of the kernel in \cref{sec:analysis_kernel} replacing the fixed point $((x_1,t_1),(x_2,t_2))$ with a point which might depend on $N$. We also believe that one can extend the results in \cref{corol:local_lim_standard Young tableau} to uniform Young tableaux (instead of uniform Poissonized Young tableaux).
\item On the boundary of the liquid region, \cref{thm:cv_bead_process} states that there are
with high probability no beads in a window of size $O(1) \times O(1/\sqrt{N})$.
We expect, however, that with a different rescaling in $t$, the bead process
will converge either to the Airy process (for typical points of the boundary) or to the Cusp-Airy/Pearcey process (for cusps points of the boundary). For instance, this type of behavior for typical points of the boundary was proved for lozenge tilings in \cite{petrov2014tilings}; while the aforementioned behavior for cusps points of the boundary was for instance investigated in \cite{OkounkovReshetikhin2006birth,OkounkovReshetikhin2007Pearcey,DuseJohanssonMetcalfe2016cuspAiry} in the setting of random tiling models. The appearance of the Airy process could be used to 
establish Tracy-Widom fluctuations of a generic entry in the first row of random Young tableaux.
Such a result was proved with a different and specific method
for uniform random tableaux of square shape by Marchal \cite{marchal2016Young}.

\item We believe that the family of infinite tableaux $(\diamond_{\infty},R_{\be})$ constructed in this paper
deserves more attention. Unlike Plancherel infinite tableaux \cite{KerovVershik1986},
the tableaux $(\diamond_{\infty},R_{\be})$ do not come from a Markovian growth process. But,
as the local limit of finite random Young tableaux, they should have an interesting re-rooting invariance property.
\end{enumerate}

\subsection{Outline of the paper}
The rest of the paper is organized as follows.
In \cref{eq:prelim}, we first recall the determinantal structure of
the bead process associated to a
uniform Poissonized tableau of given shape \cite{GR19}, and then we rewrite the corresponding kernel in a more suitable way for our analysis. Moreover, in \cref{sec:bead}, we recall some basic topological facts for determinantal point processes.
\cref{sec:analysis_kernel} is devoted to the analysis of the kernel of the random bead process $\widetilde{M}_{\la_N}^{(x_0,t_0)} $ in \eqref{eq:defn_of_Mla}. The main goal of this section will be to identify the \emph{critical points of the action} (\cref{ssec:localization-critical-points}).
The next two sections are devoted to the proof of the local limit of the bead process stated in \cref{thm:cv_bead_process}, which is the building block for all the other results in the paper: in \cref{sect:liquid_region} we look at the asymptotics of the kernel inside the liquid region, while in \cref{sect:frozen_region}  we look at the asymptotics of the kernel in the frozen region.
Finally, in the last three sections of this paper, we deduce Theorems \ref{thm:limiting_surface_formula} and \ref{thm:continuity} (\cref{sec:LimitShape}), we give the proofs of our applications discussed in \cref{sect:appl_intro} (\cref{sect:applications}) 
and we conclude with the proof of the local limit for Young tableaux stated in \cref{corol:local_lim_standard Young tableau} (\cref{sect:local_syt}).

\section{Preliminaries}\label{eq:prelim}

\subsection{The Gorin--Rahman determinantal formula}

\subsubsection{Young diagrams and associated meromorphic functions}

Recall the interlacing coordinates introduced in \cref{eq:int_coord}. With a Young diagram $\lambda$, we associate a function of a complex variable $u$. 
Throughout the paper, $\Gamma$ stands for the usual gamma function. 
We recall that the gamma function is a meromorphic function with simple poles at the non-positive integers $\left\{0,-1,-2,\dots \right\}$. Moreover, it has no zeros, so the reciprocal gamma function $\frac{1}{\Gamma}$ is an entire function.
We set
\begin{equation}\label{eq:f_lambda}
F_\la(u) := \Gamma(u+1)\,\prod_{i=1}^\infty \frac{u+i}{u-\la_i+i} =
\frac{\prod_{i=0}^m \Gamma(u-a_i+1)}{\prod_{i=1}^m \Gamma(u-b_i+1)},
\end{equation}
where the equivalence between the two formulas
is proved  in\footnote{The quantity $\Phi(z;\la)$
in \cite{IvanovOlshanski2002} is defined right before Proposition 1.2. It is related to $F_\la$ by the equation $F_\la(u) = \Gamma(u+1) \,\Phi(u+\frac{1}{2};\la)$. Beware that notations for interlacing coordinates are different here and in \cite{IvanovOlshanski2002}.} \cite[Eq.~(2.7)]{IvanovOlshanski2002}.
This meromorphic function $F_\la$ has an infinite countable set of simple poles,
namely $\Set{\la_i-i |\, i \ge 1}$.

From now on, $(a_i)_{0\le i \le m}$ and $(b_i)_{1\le i \le m}$ 
are the interlacing coordinates of our fixed base diagram $\la^0$.
The interlacing coordinates of the dilatation $\la_N$
 are simply $(n\, a_0, n \, b_1, \dots, n\, b_m, n\, a_m)$,
 where $n$ is the dilatation factor 
 (each box in $\la^0$ corresponds to an $n \times n$ square in $\la_N$).
 Recalling the relation $N=n^2 |\la^0|=n^2 \eta^{-2}$, we have
 \begin{equation}
\label{eq:Flambda_dilatation}
F_{\la_N}(u) = \frac{\prod_{i=0}^m \Gamma(u- \eta\, a_i \, \sqrt N +1)}{\prod_{i=1}^m \Gamma(u-\eta\, b_i \, \sqrt N+1)}. 
\end{equation}

\subsubsection{The Gorin--Rahman correlation kernel}
From \cite[Theorem 1.5]{GR19}, for any fixed diagram $\la$, taking a uniform random Poissonized Young tableau $T$ of shape $\la$, the bead process $M_\la=M_{\la,T}$ introduced in \eqref{eq:corr_bead_proc} is a determinantal point process on $\Z \times [0,1]$ with correlation kernel
\begin{align}\label{eq:corr_kern_1}
	K_\la((x_1,t_1),(x_2,t_2)) &= \One_{x_1>x_2,\,t_1<t_2}\,\frac{(t_1-t_2)^{x_1-x_2-1}}{(x_1-x_2-1)!} \\
	&\quad+ \frac{1}{(2\I \pi)^2}\oint_{\widetilde\gamma_z}\!\oint_{\widetilde\gamma_w} \frac{F_\la(x_2+z)}{F_\la(x_1-1-w)}\,\frac{\Gamma(-w)}{\Gamma(z+1)}\,\frac{(1-t_2)^z\,(1-t_1)^{w}}{z+w+x_2-x_1+1}\dd{w}\dd{z},\notag
\end{align}
 where the double contour integral runs over counterclockwise paths $\widetilde \gamma_w$ and $\widetilde \gamma_z$ such that
 \begin{itemize} 
 \item $\widetilde\gamma_w$ is inside $\widetilde\gamma_z$ ;
 \item $\widetilde \gamma_w$ and $\widetilde \gamma_z$ 
contain all the integers in $[ 0,\ell(\la)-1+x_1]$ and in $[0,\la_1-1-x_2 ]$
respectively;
\item the ratio $\frac{1}{z+w+x_2-x_1+1}$ remains uniformly bounded.
 \end{itemize}
Recalling the expression for $F_\la$ in \eqref{eq:f_lambda}, we get that 
\begin{itemize}
	\item the simple poles of $\frac{F_\la(x_2+z)}{\Gamma(z+1)}$ are $\Set{\la_i-i-x_2 | i \ge 1}\cap\Z_{\geq 0}$;
	\item the simple poles of $\frac{\Gamma(-w)}{F_\la(x_1-1-w)}$ are $\Z_{\geq 0}\setminus \Set{x_1-1-\la_i+i | i \ge 1}$.
\end{itemize}
\bigskip

For later convenience, we perform the following change of variables:
\begin{equation*}
	\begin{cases}
		z'=x_2+z,\\
		w'=x_1-1-w.
	\end{cases}
\end{equation*}
If  $\widetilde\gamma_z$ was sufficiently large, after this change of variable,
the new $z$-contour still encloses the new $w$-contour.
Therefore, we
obtain the following expression for the correlation kernel:
\begin{align}\label{eq:corr_kern_2}
	K_\la((x_1,t_1),(x_2,t_2)) &= \One_{x_1>x_2,\,t_1<t_2}\,\frac{(t_1-t_2)^{x_1-x_2-1}}{(x_1-x_2-1)!} \\
	&\quad- \frac{1}{(2\I \pi)^2}\oint_{\gamma_z}\!\oint_{\gamma_w} \frac{F_\la(z)}{F_\la(w)}\,\frac{\Gamma(w-x_1+1)}{\Gamma(z-x_2+1)}\,\frac{(1-t_2)^{z-x_2}\,(1-t_1)^{-w+x_1-1}}{z-w}\dd{w}\dd{z},\notag
\end{align}
 where the double contour integral runs over counterclockwise paths $\gamma_w$ and $\gamma_z$ such that
\begin{itemize} 
	\item $\gamma_w$ is inside $\gamma_z$;
	\item $\gamma_w$ and $\gamma_z$ 
	contain all the integers in $[-\ell(\la),x_1-1 ]$ and in $[ x_2 ,\la_1-1 ]$ respectively;
	\item the ratio $\frac{1}{z-w}$ remains uniformly bounded.
\end{itemize}
We note that in the new integrand in \eqref{eq:corr_kern_2},
\begin{itemize}
\item the set of $w$-poles (which are all simple) is $\Z_{\leq x_1-1} \setminus \Set{\la_i-i | i\geq 1}$, which is included in $[-\ell(\la),x_1-1]$;
\item  the set of $z$-poles (which are all simple) is 
$\Set{\la_i-i | i\geq 1} \cap \Z_{\geq x_2}$, which is included in $[x_2,\la_1-1]$; 
\item there is an additional simple pole for $z=w$.
\end{itemize}
Hence the second property of $\gamma_w$ and $\gamma_z$ given above ensures that the integration contours contain all poles.
We refer to \cref{fig:poles} for a visual representation of these poles and integration contours.

\begin{figure}[t]
\[\begin{array}{c}
\includegraphics[height=4cm]{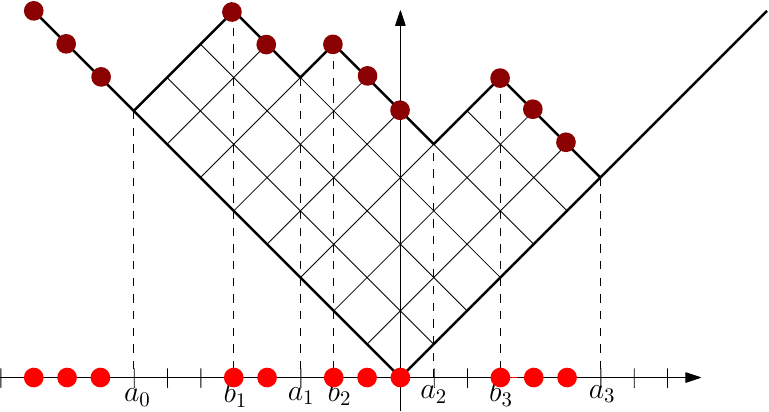}
\end{array} \quad 
\begin{array}{c}
\includegraphics[height=6cm]{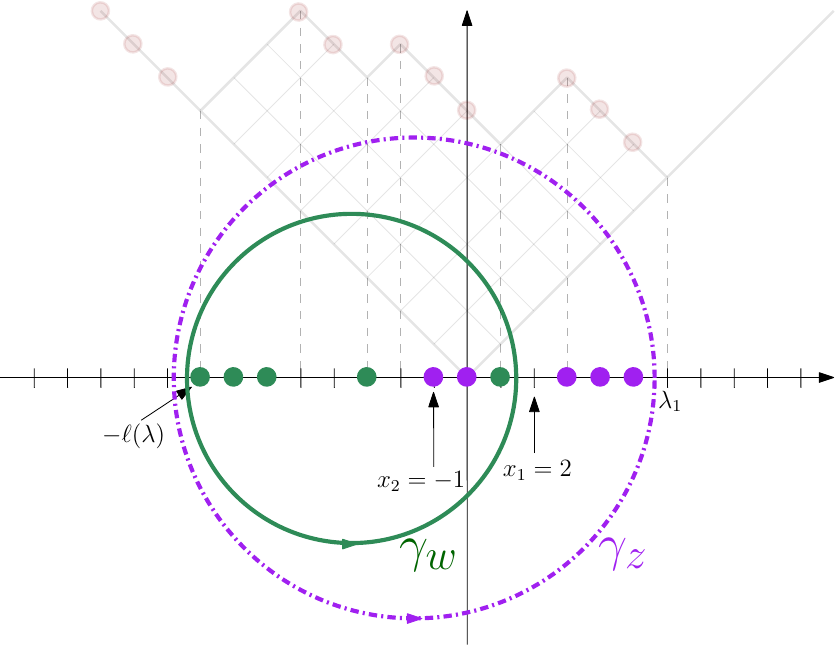}
\end{array}\]
	\captionsetup{width=\linewidth}
	\caption{We consider here the Young diagram $\la=(6,6,6,4,4,4,3,3)$ drawn with Russian convention. \textbf{Left:} We indicate the poles of $F_\lambda$
	in red  (and in brown the corresponding points on the boundary of the diagram).
	\textbf{Right:} We show the poles of the integrand, the $w$-poles in green, and the
    $z$-poles in purple
     in the case when $x_1=2$ and $x_2=-1$ (recall that there is an additional simple pole for $z=w$).
    The integration contours in \cref{eq:corr_kern_2} are also represented.
	}
	\label{fig:poles}
\end{figure}

\subsubsection{Removing the indicator function from the correlation kernel.} Let $\gamma'_z$ and $\gamma'_w$ be the two contours satisfying  the same conditions as $\gamma_z$ and $\gamma_w$, 
except that $\gamma'_z$ lies inside $\gamma'_w$.
A simple residue computation, done below, yields the following result.

\begin{lemma} For all $(x_1,t_1),(x_2,t_2)\in\Z \times [0,1]$, it holds that
	\begin{multline}
		\One_{x_1>x_2} \frac{(t_1-t_2)^{x_1-x_2-1}}{(x_1-x_2-1)!} 
		- \frac{1}{(2\I \pi)^2}\oint_{\gamma_z}\!\oint_{\gamma_w} \frac{F_\la(z)}{F_\la(w)}\,\frac{\Gamma(w-x_1+1)}{\Gamma(z-x_2+1)}\,\frac{(1-t_2)^{z-x_2}\,(1-t_1)^{-w+x_1-1}}{z-w}\dd{w}\dd{z}\\
		= -
		\frac{1}{(2\I \pi)^2}\oint_{\gamma'_z}\!\oint_{\gamma'_w} \frac{F_\la(z)}{F_\la(w)}\,\frac{\Gamma(w-x_1+1)}{\Gamma(z-x_2+1)}\,\frac{(1-t_2)^{z-x_2}\,(1-t_1)^{-w+x_1-1}}{z-w}\dd{w}\dd{z}.
	\end{multline}
\end{lemma}

\begin{proof}
	To simplify, we can assume that $\gamma'_w=\gamma_w$.
	Fixing some $w \in \gamma_w$,
	we can deform $\gamma_z$ into $\gamma'_z$, crossing uniquely the pole $z=w$.
	 Hence, for $w \in \gamma_w$, by residue theorem we have that
	\begin{multline}
		\frac{1}{(2\I \pi)^2}\oint_{\gamma_z}\! \frac{F_\la(z)}{F_\la(w)}\,\frac{\Gamma(w-x_1+1)}{\Gamma(z-x_2+1)}\,\frac{(1-t_2)^{z-x_2}\,(1-t_1)^{-w+x_1-1}}{z-w}\dd{w}\dd{z}\\
		-
		\frac{1}{(2\I \pi)^2}\oint_{\gamma'_z} \frac{F_\la(z)}{F_\la(w)}\,\frac{\Gamma(w-x_1+1)}{\Gamma(z-x_2+1)}\,\frac{(1-t_2)^{z-x_2}\,(1-t_1)^{-w+x_1-1}}{z-w}\dd{w}\dd{z}\\
		=
		\frac{1}{2\I \pi}  \frac{\Gamma(w-x_1+1)}{\Gamma(w-x_2+1)}\,(1-t_2)^{w-x_2}\,(1-t_1)^{-w+x_1-1}
		\label{eq:Tech2}
	\end{multline}
	We now integrate the right-hand side over the closed contour $\gamma_w$.
In the case $x_1\le x_2$, we have 
\[\frac{\Gamma(w-x_1+1)}{\Gamma(w-x_2+1)}=(w-x_1)(w-x_1-1) \cdots (w-x_2+1), \]
and the right-hand side of \eqref{eq:Tech2} is an entire function. Hence integrating over $\gamma_w$ gives $0$.
On the contrary when $x_1>x_2$, we have  
\[\frac{\Gamma(w-x_1+1)}{\Gamma(w-x_2+1)}=\frac{1}{(w-x_2)(w-x_2-1)\cdots(w-x_1+1)}.\]
In this case, the RHS of \eqref{eq:Tech2} is a meromorphic function of $w$
with simple poles in $x_1-1-i$, for $i \in \{0,1,\dots,x_1-x_2-1\}$.
Bringing both cases together, we have
	\begin{align*}
		\frac{1}{2\I \pi}\oint_{\gamma_w} \frac{\Gamma(w-x_1+1)}{\Gamma(w-x_2+1)}&\,(1-t_2)^{w-x_2}\,(1-t_1)^{-w+x_1-1}\dd{w}\\
		&=
		\One_{x_1>x_2} \sum_{i=0}^{x_1-x_2-1}\frac{(1-t_2)^{(x_1-1-i)-x_2}\,(1-t_1)^{-(x_1-1-i)+x_1-1}}{((x_1-1-i)-x_2) \cdots 1 \cdot (-1) \cdots (-i)}\\
		&=
                \frac{\One_{x_1>x_2}}{(x_1-x_2-1)!}
              \sum_{i=0}^{x_1-x_2-1}\binom{x_1-x_2-1}{i}(1-t_2)^{x_1-x_2-1-i}(t_1-1)^{i}\\
		&=
		\One_{x_1>x_2}
                \frac{(t_1-t_2)^{x_1-x_2-1}}{(x_1-x_2-1)!}.\qedhere
	\end{align*}
\end{proof}
We, therefore, have the following simplified (and final) representation of the correlation kernel:
\begin{multline}
  \label{eq:kernel_without_indicator}
	K_\la((x_1,t_1),(x_2,t_2)) =
        -\frac{1}{(2\I \pi)^2}\oint_{\gamma_z}\!\oint_{\gamma_w} \frac{F_\la(z)}{F_\la(w)}\,\frac{\Gamma(w-x_1+1)}{\Gamma(z-x_2+1)}\,\frac{(1-t_2)^{z-x_2}\,(1-t_1)^{-w+x_1-1}}{z-w}\dd{w}\dd{z}, 
\end{multline}
where the double contour integral runs over counterclockwise paths $\gamma_w$ and $\gamma_z$ such that
\begin{itemize} 
  \item $\gamma_w$ is inside (resp.\ outside) $\gamma_z$ if $t_1\geq t_2$ (resp.\ $t_1 < t_2$);
	\item $\gamma_w$ and $\gamma_z$
	contain all the integers in $[ -\ell(\la) ,x_1-1 ]=[a_0,x_1-1]$ and in $[ x_2 ,\la_1-1 ]=[x_2 ,a_m-1]$ respectively;
	\item the ratio $\frac{1}{z-w}$ remains uniformly bounded.
\end{itemize}
Moreover, the integrand has simple poles for $w$ at $\Z_{\leq x_1-1}\setminus\Set{\la_i-i | i\geq 1}\subseteq[-\ell(\la),x_1-1]$,  simple poles for $z$ at 
$\Set{\la_i-i | i\geq 1 }\cap\Z_{\geq x_2}\subseteq[x_2,\la_1-1]$, and a simple pole for $z=w$.

\subsection{A topology for bead processes}
\label{sec:bead}

We discuss in this section a (standard) notion of convergence for bead processes, or more generally for determinantal point processes. The expert reader may consider skipping this section.

A bead process can be naturally interpreted as a random locally finite counting measure on the complete and separable metric space $\mathcal X:=\Z\times \R$, i.e.\ as a random variable in the space of locally finite counting measures on $\mathcal X$, denoted by ${\mathcal M}^{\#}_{\mathcal X}$. We endow ${\mathcal M}^{\#}_{\mathcal X}$ with the $\sigma$-algebra $\mathcal F_{\mathcal X}$ generated by the cylinder sets
\begin{equation}
	C^n_B=\Set{\mu\in{\mathcal M}^{\#}_{\mathcal X} | \mu(B)=n},
\end{equation}
defined for all $n\in\mathbb N$ and all Borel subsets $B$ of $\mathcal X$.
\medskip

The elements of this $\sigma$-algebra are the Borel subsets for the weak topology on the space $({\mathcal M}^{\#}_{\mathcal X},\mathcal F_{\mathcal X})$ of locally finite counting measures on ${\mathcal X}$. The convergent sequences for the weak topology are defined as follows: if $(\mu_n)_n$ is a sequence in ${\mathcal M}^{\#}_{\mathcal X}$ and $\mu \in {\mathcal M}^{\#}_{\mathcal X}$, then
\begin{equation*}
	 \mu_n\xrightarrow{w}\mu \qquad\text{if}\qquad \int_{{\mathcal X}} f(x) \dd{\mu_n}(x)\to\int_{{\mathcal X}} f(x) \dd{\mu}(x),
\end{equation*}
for all bounded continuous functions $f$ on ${\mathcal X}$ with bounded support. 
This topology makes ${\mathcal M}^{\#}_{\mathcal X}$ a Polish space, which eases the manipulation of sequences of random measures in ${\mathcal M}^{\#}_{\mathcal X}$. In particular, by \cite[Theorem 11.1.VII]{daley2008introduction}, if $(\nu_n)_n$ is a sequence of random measures in ${\mathcal M}^{\#}_{\mathcal X}$ and $\nu$ is another random measure in ${\mathcal M}^{\#}_{\mathcal X}$, then $\nu_n\xrightarrow{d}\nu$ w.r.t.\ the weak topology if for all $k\geq 1$,
\begin{equation}\label{eq:conv_random_meas}
	(\nu_n(B_i))_{1\leq i \leq k}\xrightarrow{d} (\nu(B_i))_{1\leq i \leq k}, 
\end{equation}
for all collections $(B_i)_{1\leq i \leq k}$ of  bounded Borel $\nu$-continuity subsets\footnote{Recall that a bounded Borel subset of ${\mathcal X}$ is a \emph{continuity subset} for a random measure $\nu$ in ${\mathcal M}^{\#}_{\mathcal X}$ if $\nu(\partial B)=0$ almost surely.} of ${\mathcal X}$. Since both $\nu_n$ and $\nu$ are random counting measures, the condition in \eqref{eq:conv_random_meas} is equivalent to requiring that for all $k\geq 1$,
\begin{equation}\label{eq:conv_random_meas2}
	\proba(\nu_n(B_i)=m_i, \forall 1\leq i \leq k)
	\to 
	\proba(\nu(B_i)=m_i, \forall 1\leq i \leq k), 
\end{equation}
for all collections $(B_i)_{1\leq i \leq k}$ of  bounded Borel $\nu$-continuity subsets of ${\mathcal X}$ and all integer-valued vectors $(m_i)_{1\leq i \leq k}\in(\Z_{\geq 0})^k$. In particular, we highlight for later scopes the following result.

\begin{proposition}\label{prop:cons_conv_bead}
	 Let $(\nu_n)_n$ be a sequence of random measures in ${\mathcal M}^{\#}_{\mathcal X}$ and $\nu$ is another random measure in ${\mathcal M}^{\#}_{\mathcal X}$. If $\nu_n\xrightarrow{d}\nu$ w.r.t.\ the weak topology, then 
	 \begin{equation*}
	 	\proba(\nu_n \in A)
	 	\to 
	 	\proba(\nu \in A)
	 \end{equation*}
 	for all sets $A\in\mathcal F_{\mathcal X}$ which can be written as a countable union/intersection/complementation of cylinder sets $C^n_B$ such that $B$ is a bounded Borel $\nu$-continuity subset of $\mathcal X$.
\end{proposition}

\medskip

Recalling that a bead process is a determinantal point process on ${\mathcal X}=\Z\times \R$, we now look at the specific case where $\nu$ and $\nu_n$ are determinantal point process on ${\mathcal X}$. 

Let $J=J\big((x_1,t_1),(x_2,t_2)\big)$ and $J_n=J_n\big((x_1,t_1),(x_2,t_2)\big)$ be the kernels of $\nu$ and $\nu_n$ respectively. We also introduce the following generating function for the determinantal point process $\nu$ (and analogously for the determinantal point processes $\nu_n$): for $z=(z_1,\dots,z_k)$ and $B=(B_1,\dots, B_k)$, we set
\begin{equation*}
	Q^{\nu}_B(z)
	:=\sum_{\ell\in\N^k}\frac{(-z)^\ell}{\ell !} \int_{B^\ell} \det\left(\left[J\big((x_i,t_i),(x_j,t_j)\big)\right]_{1\leq i,j \leq |\ell|}\right)\dd{\la}^{\otimes |\ell|}\big((x_1,t_1),\dots,(x_{|\ell|},t_{|\ell|})\big),
\end{equation*}
where $\la$ denotes the product measure of the counting measure on $\Z$ and the Lebesgue measure on $\R$, and where we also used the following multi-index notation
\begin{equation*}
	|\ell| = \sum_{i=1}^k \ell_i,\qquad \ell!=\prod_{i=1}^k \ell_i!,\qquad B^\ell=B_1^{\ell_1}\times\dots\times B_k^{\ell_k},\qquad z^{\ell}=z_1^{\ell_1}\cdots z_k^{\ell_k}.
\end{equation*}
With some simple manipulations (see for instance \cite[Equation 3.4]{boutillier2005modeles}), one can show that for $m=(m_1,\dots,m_k)\in(\Z_{\geq 0})^k$, it holds that
\begin{equation*}
	\proba\big(\nu(B_i)=m_i, \forall 1\leq i \leq k\big)=\frac{(-1)^{|m|}}{m!} \left.\frac{\partial^{m}}{\partial z^{m}}\,Q^{\nu}_B(z)\right|_{z=(1,\dots,1)},
\end{equation*}
where $\frac{\partial^{m}}{\partial z^{m}}$ stands for the differential operator $\frac{\partial^{m_1}}{(\partial z_1)^{m_1}}\cdots\frac{\partial^{m_k}}{(\partial z_k)^{|m_k|}}$ and $|m|=\sum_{i=1}^k m_i$.
Therefore, thanks to the latter expression, the condition in \eqref{eq:conv_random_meas2} is equivalent to require that for all $k\geq 1$,
\begin{equation}\label{eq:conv_random_meas3}
	\left.\frac{\partial^{m}}{\partial z^{m}}\,Q^{\nu_n}_B(z)\right|_{z=(1,\dots,1)}
	\xrightarrow{n\to \infty} 
	\left.\frac{\partial^{m}}{\partial z^{m}}\,Q^{\nu}_B(z)\right|_{z=(1,\dots,1)}
\end{equation}
for all collections $(B_i)_{1\leq i \leq k}$ of  bounded Borel $\nu$-continuity subsets of ${\mathcal X}$ and all integer-valued vectors $(m_i)_{1\leq i \leq k}\in(\Z_{\geq 0})^k$. We have the following useful result.

\begin{proposition}\label{prop:dpp_conv}
	If the kernel $J=J\big((x_1,t_1),(x_2,t_2)\big)$ is locally bounded and $J_n$ converges locally uniformly to $J$, then $\nu_n\xrightarrow{d}\nu$ w.r.t.\ the weak topology.
\end{proposition}

\begin{proof}
	Thanks to our previous discussions it is enough to show that the condition in \eqref{eq:conv_random_meas3} holds.
	Fix a collection $B=(B_i)_{1\leq i \leq k}$ of bounded Borel $\nu$-continuity subsets of ${\mathcal X}$, and $\ell\in\N^k$. Since $B^\ell$ is relatively compact, by using the Hadamard inequality and the locally uniform convergence of $J_n$ towards the locally bounded kernel $J$, we see that for all $ (x_i,t_i)_{1\leq i\leq |\ell|} \in B^\ell$,
\begin{equation*}
    \left|\det\left(\left[J_n\big((x_i,t_i),(x_j,t_j)\big)\right]_{1\leq i,j \leq |\ell|}\right)\right|\leq \prod_{j=1}^{|\ell|} \left(\sum_{i=1}^{|\ell|} J_n\big((x_i,t_i),(x_j,t_j)\big)^2\right)^{\!\frac{1}{2}}\leq  \left(C\sqrt{|\ell|}\right)^{|\ell|}
  \end{equation*}
  for some positive constant $C$. Since $J_n$ converges pointwise to $J$, for every fixed $(x_i,t_i)_{1\leq i\leq |\ell|}\in B^\ell$,
  \begin{equation*}
    \det\left(\left[J_n\big((x_i,t_i),(x_j,t_j)\big)\right]_{1\leq i,j \leq |\ell|}\right)
    \to \det\left(\left[J\big((x_i,t_i),(x_j,t_j)\big)\right]_{1\leq i,j \leq |\ell|}\right),
  \end{equation*}
   and by dominated convergence, we get
	\begin{multline}\label{eq:conv_goal}
		\int_{B^\ell} \det\left(\left[J_n\big((x_i,t_i),(x_j,t_j)\big)\right]_{1\leq i,j \leq |\ell|}\right)
		\dd{\la}^{\otimes |\ell|}\big((x_1,t_1),\dots,(x_{|\ell|},t_{|\ell|})\big)\\
		\to
		\int_{B^\ell} \det\left(\left[J\big((x_i,t_i),(x_j,t_j)\big)\right]_{1\leq i,j \leq |\ell|}\right)
		\dd{\la}^{\otimes |\ell|}\big((x_1,t_1),\dots,(x_{|\ell|},t_{|\ell|})\big).
	\end{multline}
	It remains to prove that \eqref{eq:conv_goal} implies the condition in \eqref{eq:conv_random_meas3}. Thanks to the estimate given above, the term with label $\ell$ in the series $Q^{\nu_n}_B(z)$ or $Q^{\nu}_B(z)$ is bounded from above by 
  $$\frac{|z|^\ell \,|\la(B)|^\ell }{\ell!}(C\sqrt{|\ell|})^{|\ell|} \leq \left(\frac{\|z\|_\infty\, \|\la(B)\|_\infty\, C\,k\E }{\sqrt{|\ell|}}\right)^{|\ell|} .$$
  Here we used the fact that for any $k$-tuple $\ell = (\ell_1,\ldots,\ell_k)$, $\ell ! \geq (\Gamma(\frac{|\ell|}{k}+1))^k \geq (\frac{|\ell|}{k\E})^{|\ell|}$; see also the discussion below \cite[Lemma 3.1]{boutillier2005modeles}. Therefore, $Q^{\nu_n}_B(z)$ and $Q^{\nu}_B(z)$ 
  are entire functions of $z$, and by dominated convergence of the terms of the series, $Q^{\nu_n}_B(z)$ converges locally uniformly on $\mathbb{C}^k$ towards $Q^\nu_B(z)$. By Cauchy's differentiation formula, this implies the (locally uniform) convergence of all the partial derivatives, whence the general condition \eqref{eq:conv_random_meas3}.
\end{proof}

We finish this section by recalling a standard observation in the theory of determinantal point processes.
Assume that $K_1$ and $K_2$ are two kernels on a set $X$ differing by a {\em conjugation factor}, i.e.~such that
there exists a function $g$ on $X$ satisfying
\[K_1(x,y) =\frac{g(x)}{g(y)} K_2(x,y),\quad \text{for all } x,y\in X.\]
Then, for all $x_1,\dots,x_k$ in $X$, we have
\[\det\left( K_1(x_i,x_j) \right)_{1 \le i,j \le k} = \frac{\prod_{i=1}^k g(x_i)}{\prod_{j=1}^k g(x_j)}
\det\left( K_2(x_i,x_j) \right)_{1 \le i,j \le k} = \det\left( K_2(x_i,x_j) \right)_{1 \le i,j \le k}.\]
Therefore the two kernels induce the same correlation functions (of any order)
and therefore the associated point processes are equal in distribution.

\section{Identifying the critical points of the action}
\label{sec:analysis_kernel}

\subsection{The kernel of the renormalized bead process} 
\label{sec:KernelRenormalized}

We start our analysis by looking at the random bead process $M_{\la_N}$ introduced in \eqref{eq:corr_bead_proc} in a window of size $O(1) \times O(1/\sqrt{N})$
around a fixed point $(x_0 \sqrt{N},t_0)$ with $(x_0,t_0)\in[\eta\, a_0, \eta\, a_m]\times [0,1]$. In particular, in \eqref{eq:defn_of_Mla} we introduced the renormalized bead process
\begin{equation*}
	\widetilde{M}_{\la_N}^{(x_0,t_0)} = \Set{ (x,t) \in \Z \times \R |  \big(x_0 \sqrt{N} + x,t_0 + \tfrac{t}{\sqrt{N}} \big) \in  M_{\la_N} }.
\end{equation*}
Using simple properties of determinantal point processes,
we know that  $\widetilde{M}_{\la_N}^{(x_0,t_0)}$ is a determinantal point process with correlation kernel
\begin{equation}\label{eq:renorm_kern}
\widetilde{K}_{\la_N}^{(x_0,t_0)}((x_1,t_1),(x_2,t_2)) := \frac1{\sqrt{N}}\, K_{\la_N}\left(\left( x_0 \sqrt{N} + x_1,t_0 +\frac{t_1}{\sqrt{N}}\right),\left( x_0 \sqrt{N} + x_2,t_0 +\frac{t_2}{\sqrt{N}}\right)\right).
\end{equation}
We will use the double integral expression for
$K_{\la_N}$ given in \eqref{eq:kernel_without_indicator}.
The integration contours are different in the case $t_1 \ge t_2$
and $t_1<t_2$.
Unless stated otherwise, expressions below are valid for $t_1 \ge t_2$,
and we indicate only when necessary the needed modification for $t_1 < t_2$.

\medskip

Performing the change of variables
$(w,z) = (\sqrt{N}(W+x_0),\,\sqrt{N}(Z+x_0))$,
we have
 \begin{multline}\label{eq:double_int_integr} \widetilde{K}_{\la_N}^{(x_0,t_0)}((x_1,t_1),(x_2,t_2)) 
   = - \frac{1}{(2\I \pi)^2}\oint_{\gz}\!\oint_{\gw} \frac{F_{\la_N}(\sqrt{N}(Z+x_0))}{F_{\la_N}(\sqrt{N}(W+x_0))}\,\frac{\Gamma(\sqrt{N}W- x_1+1)}{\Gamma(\sqrt{N}Z - x_2+1)}
 \\ \cdot \frac{(1-\widetilde{t}_2)^{\sqrt{N}Z - x_2}\,(1-\widetilde{t}_1)^{-\sqrt{N}W + x_1-1}}{(Z-W)}\dd{W}\dd{Z},
 \end{multline}
where $\widetilde{t}_1=t_0 +\frac{t_1}{\sqrt{N}}$ and $\widetilde{t}_2=t_0 +\frac{t_2}{\sqrt{N}}$. 
With the new variables $W$ and $Z$, recalling the comments below \eqref{eq:kernel_without_indicator}, we get that the poles of the integrand in \eqref{eq:double_int_integr} occur
\begin{itemize}
	\item for $W=Z$;
	\item for some values of $W$ in an interval of the form
	$I_W:=[\frac{n\,a_0}{\sqrt N}-x_0,\frac{x_1-1}{\sqrt{N}}]=[\eta \,a_0-x_0,o(1)]$ (referred to as $W$-poles below);
	\item for some values of $Z$ in an interval $I_Z:=[\frac{x_2}{\sqrt N}, \frac{n\, a_m}{\sqrt N}-x_0]=[o(1),\eta\, a_m- x_0]$ (referred to as $Z$-poles below).
\end{itemize}
The integration contours $\gw$ and $\gz$ above should be such that $\gw$ is inside $\gz$
and contain $I_W$ and $I_Z$ respectively. Setting $L=\eta \max(|a_0|,a_m)$ and recalling that $x_0$ is in $[\eta\, a_0, \eta\, a_m]$, we have
\begin{equation}\label{eq:bound_L}
	(\eta\, a_0-x_0),(\eta\, a_m-x_0)\in[-2L,2L].
\end{equation}
Hence, for $N$ large enough, we can take, for instance, $\gw=\partial D(0,3L)$ and $\gz=\partial D(0,4L)$ (both followed in counterclockwise order), where $D(z,r)$ denotes the open disk centered at $z\in\mathbb{C}$ and of radius $r\in\mathbb{R}_{\geq 0}$ and $\partial D(z,r)$ denotes its boundary, i.e. the circle with center $z$ and radius $r$. For convenience, we will often write $\Int_N(W,Z)$ for the integrand in \eqref{eq:double_int_integr}.

\subsection{Asymptotic analysis of the integrand}

Our next goal is to write the double integral in \eqref{eq:double_int_integr} under an exponential form
using asymptotic approximations.
In the following, for two sequences $A=(A_N)$ and $B=(B_N)$, we write $A \simeq B$ when $A=B(1+O(N^{-1/2}))$.
Moreover, unless stated otherwise, when $A$ and $B$ depend on $Z$ and $W$,
the estimates are uniform for $Z$ and $W$ lying in the integration contours 
 $\gw=\partial D(0,3L)$ and $\gz=\partial D(0,4L)$
 (we recall that $x_0$, $x_1$, $x_2$, and $t_0$, $t_1$, $t_2$ are fixed).
First, since $\widetilde{t}_1=t_0 +\frac{t_1}{\sqrt{N}}$ and $\widetilde{t}_2=t_0 +\frac{t_2}{\sqrt{N}}$, we have that
\begin{align}
&(1-\widetilde{t}_2)^{\sqrt{N}Z - x_2} = (1-t_0)^{\sqrt{N}Z - x_2} \left( 1- \frac{t_2}{(1-t_0)\sqrt{N}}\right)^{\sqrt{N}Z - x_2}
 \simeq (1-t_0)^{-x_2} \E^{ \sqrt{N} Z\log(1-t_0)} \E^{-\frac{Zt_2}{1-t_0}};\notag\\
&(1-\widetilde{t}_1)^{-\sqrt{N}W + x_1-1} \simeq (1-t_0)^{x_1-1} \E^{ -\sqrt{N} W\log(1-t_0)}\,\E^{\frac{Wt_1}{1-t_0}}.\label{eq:asympt1}
\end{align}
We now estimate the quotients involving $F_{\la_N}$ and $\Gamma$ in \eqref{eq:double_int_integr}. In what follows we use the principal branch of the logarithm, defined and continuous on $\mathbb C \setminus \mathbb R_-$. 
We also use the convention that, for negative real $x$, we have $\log(x)=\log(-x)+\I\pi$. In particular $\Re \log(x)$ is continuous on $\mathbb C \setminus \{0\}$.

The following function $S_1:\mathbb C \to \mathbb C$ will play a crucial role in our analysis (see \cref{lem:Asymp_Fla} and \cref{eq:equiv_integrand} below):
\begin{equation}\label{eq:s1}
  S_1(U):=-g(U)+ \sum_{i=0}^m g(U+x_0-\eta\, a_i) -\sum_{i=1}^m g(U+x_0-\eta\, b_i),\quad\text{ where }\quad g(U)=U\log(U).
\end{equation}
Let us also introduce a specific subset $\DS$ of the complex plane (we invite the reader to compare the next explanation with \cref{fig:diagram-for-DS}).
An interval of the form $[\eta\, a_{i-1}-x_0, \eta b_i-x_0]$ or $[\eta\, b_i-x_0, \eta a_i-x_0]$
for some $1 \le i\le m$ will be called {\em negative} (resp.~{\em positive})
if it is included in $(-\infty,0]$ (resp.~$[0,+\infty)$).
Then we define $\DS$ as the complex plane $\mathbb C$
from which we remove the following closed real intervals:
\begin{itemize}
  \item the negative intervals of the form $[\eta\, a_{i-1}-x_0, \eta b_i-x_0]$;
  \item the positive intervals of the form $[\eta\, b_i-x_0, \eta a_i-x_0]$;
  \item either $[\eta\, a_{i_0-1}-x_0,0]$ if $\eta\, a_{i_0-1}-x_0 < 0 \le \eta\, b_{i_0}-x_0$
     for some $i_0$, or $[0,\eta\, a_{i_0}-x_0]$ if $\eta\, b_{i_0}-x_0 \le 0 < \eta\, a_{i_0}-x_0$ for some $i_0$.
\end{itemize}

\begin{figure}[ht]
	\includegraphics[height=8cm]{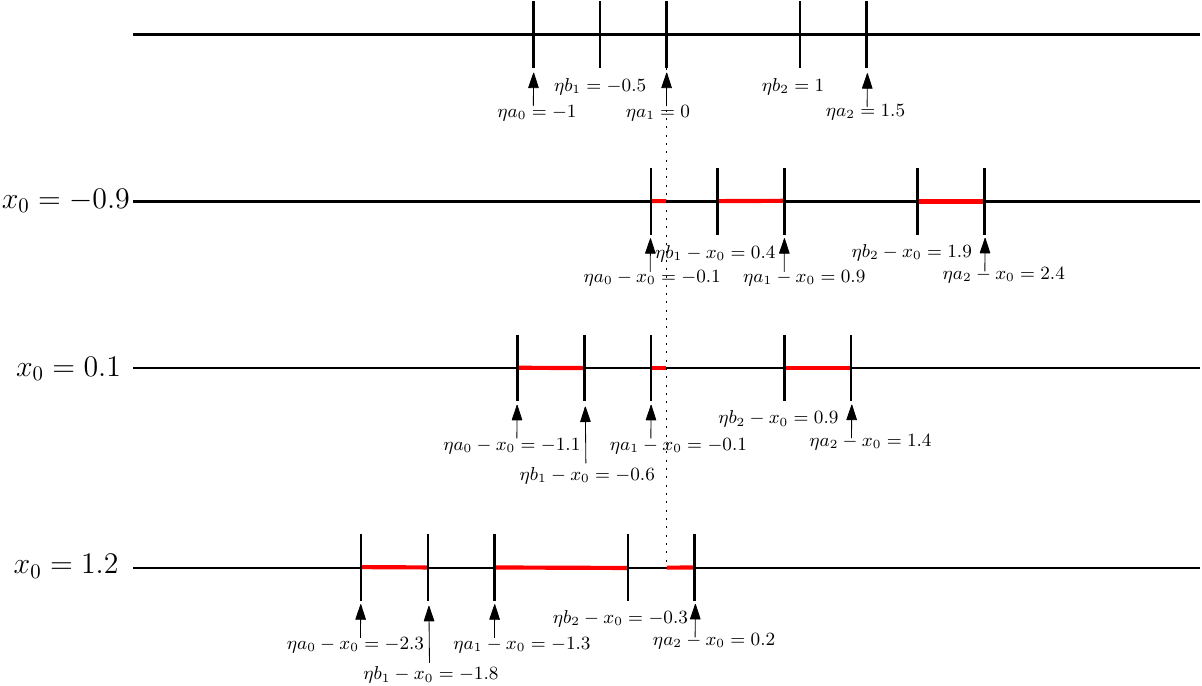}
	\caption{A schema for the  subset $\DS$ of the complex plane introduced below \cref{eq:s1}. Here the interlacing coordinates are $a_0=-2<b_1=-1<a_1=0<b_2=2<a_2=3$. In particular, according to our notation $m=2$ and $\eta=1/\sqrt{ |\la^0|}=1/2$, so that $[\eta\, a_0, \eta\, a_m]=[-1,3/2]$.  We sketch the subset $\DS$ in three different cases, when $x_0=-0.9$, $x_0=0.1$, and $x_0=1.2$: $\DS$ is equal to the complex plane $\mathbb C$ from which we remove the red closed real intervals.}
	\label{fig:diagram-for-DS}
\end{figure}

Finally, following e.g.~\cite{Biane2003}, we introduce the Cauchy transform 
of the \emph{transition measure} of the renormalized diagram $\eta  \la^0$:
\[G(U):=\frac{\prod_{i=1}^m (U-\eta\, b_i)}{\prod_{i=0}^m (U-\eta\, a_i)}.\]

\begin{lemma}\label{lem:Asymp_Fla}
Fix an integer $x\in\Z$.
As $N  \to +\infty$, the following approximation holds, 
uniformly for $U$ in compact subsets of $\DS$:
\begin{equation}
\label{eq:Asymp_Fla}
\frac{F_{\la_N}(\sqrt{N}(U+x_0))}{\Gamma(\sqrt N U-x+1)}\simeq 
\E^{x_0 (g(\sqrt{N})-\sqrt{N})}\,
\E^{S_1(U) \sqrt{N}}  \left(U\sqrt N\right)^{\!x} \big(U \cdot G(U+x_0)\big)^{-\frac{1}{2}},
\end{equation}
Moreover, writing $\LHS(U)$ and $\RHS(U)$ for the left-hand and right-hand sides of the previous estimates, for any $\eps>0$, we have the uniform bound $\frac{|\LHS(U)|}{|\RHS(U)|}= O \left( \exp(\eps \pi \sqrt N) \right)$ for $U$ in compact subsets of 
\[\{V: \Re(V)<0 \text{ and } |\Im(V)| \le \eps\}
\setminus \{\eta\, a_0-x_0,\eta \,b_0-x_0,\dots,\eta\, b_m-x_0,\eta\, a_m-x_0\}.\]
Similarly, for any $\eps>0$, we have the uniform bound $\frac{|\RHS(U)|}{|\LHS(U)|}= O \left(\exp(\eps \pi \sqrt N)\right)$
for $U$ in compact subsets of 
\[\{V: \Re(V)>0 \text{ and } |\Im(V)| \le \eps\}
\setminus \{\eta\, a_0-x_0,\eta \,b_0-x_0,\dots,\eta\, b_m-x_0,\eta\, a_m-x_0\}.\]
\end{lemma}

Before proving the lemma, let us discuss some implication.
Recall that we write $\Int_N(W,Z)$ for the integrand in \eqref{eq:double_int_integr}.
Bringing together the estimates in \cref{eq:asympt1,eq:Asymp_Fla},
we get that uniformly for $Z$ and $W$ in compact subsets of $\DS$,
\begin{equation}
	\label{eq:equiv_integrand}
	\Int_N(W,Z) \simeq (\sqrt{N})^{x_2-x_1}
	\E^{\sqrt{N}(S(W) - S(Z))}\,\hh(W,Z),
\end{equation}
where 
\begin{align}\label{eq:action}
	&S(U):=-S_1(U)-U\log(1-t_0)\notag\\
	&\qquad\stackrel{\eqref{eq:s1}}{=} g(U) - U\log(1-t_0)- \sum_{i=0}^m g (x_0-\eta\, a_i+U) + \sum_{i=1}^m g(x_0-\eta\, b_i+U); \notag \\
	&\hh(W,Z):= \frac{\E^{\frac{Wt_1-Zt_2}{1-t_0}}\, (1-t_0)^{ x_1- x_2-1}}{W^{x_1}Z^{-x_2}(Z-W)}\,
	\left(\frac{W\, G(x_0+W)}{Z\, G(x_0+Z)} \right)^{\frac{1}{2}}.
\end{align}
We note for future reference, that as $|U|$ tends to infinity, we have that
\begin{equation}\label{eq:as_action}
	S(U) = |\log(1-t_0)|\, U +O(\log(|U|)).
\end{equation}
Indeed, the terms of order $U\log U$ cancel out.

\begin{proof}[Proof of \cref{lem:Asymp_Fla}]
From \eqref{eq:Flambda_dilatation} we have that
\[ F_{\la_N}\big(\sqrt{N}(U+x_0)\big)=\frac{\prod_{i=0}^m \Gamma\big((U+x_0-\eta\, a_i)\sqrt N+1\big)}{\prod_{i=1}^m \Gamma\big((U+x_0-\eta\, b_i)\sqrt N+1\big)}.
\]
We use Stirling's approximation for the $\Gamma$ function:
\begin{equation}\label{eq:stirl_form}
	 \Gamma(U+1) = \E^{U \log U - U}\,\sqrt{2\pi U}\,(1+O(U^{-1})); 
\end{equation}
this approximation is uniform for $\Arg(U)$ in a compact sub-interval of $(-\pi,\pi)$.
Then a straightforward computation
shows that \eqref{eq:Asymp_Fla} holds uniformly for $U$ in compact subsets of 
{$\mathbb C \setminus (-\infty, \eta\, a_m-x_0]$}.

We need to extend this estimate to compact subsets of $\DS$.
For this we will prove that \eqref{eq:Asymp_Fla} holds on compact subsets of $I +\I\mathbb R$,
where $I$ is one interval among: the interval $(-\infty, \eta\, a_0-x_0)$; the negative intervals
$(\eta\, b_i-x_0, \eta a_i-x_0)$; the positive intervals $(\eta\, a_{i-1}-x_0, \eta b_i-x_0)$;                                                                                                                                                                                
and either $(0,\eta\, b_{i_0}-x_0)$ or $(\eta\, b_{i_0}-x_0,0)$.
We will treat the case where $I$ is a negative interval $(\eta\, b_i-x_0, \eta a_i-x_0)$,
the other cases being similar.

The idea is to use Euler's reflection formula 
$\Gamma(U)\, \Gamma(1-U)=\frac{\pi}{\sin(\pi U)}$ to get rid of $\Gamma$ functions
applied to negative real numbers.
For $U \in I+\I \mathbb R$, we have
\begin{multline*}
\frac{F_{\la_N}(\sqrt{N}(U+x_0))}{\Gamma(\sqrt N U-x+1)}=
(-1)^{({+}x_0+\eta(b_{i+1}+\dots+b_m-a_{i}-\dots-a_m))\sqrt{N}{+}x}\,\Gamma\!\left(x-\sqrt N U\right)\\
\frac{\prod_{j=0}^{i-1} \Gamma\big((U+x_0-\eta\, a_j)\sqrt N+1\big)}{\prod_{j=1}^i \Gamma\big((U+x_0-\eta\, b_j)\sqrt N+1\big)}\,
\frac{\prod_{j=i+1}^m \Gamma\!\left(-(U+x_0-\eta\, b_j)\sqrt N\right)}{\prod_{j=i}^m \Gamma\!\left(-(U+x_0-\eta\, a_j)\sqrt N\right)};
\end{multline*}
the sign comes from the quotient of the $\sin(\pi U)$ factors in Euler's reflection factor,
noting that we have applied the reflection formulas as many times in the numerator as in the denominator,
and that the argument of the various $\Gamma$ functions differ by integer values.
For $U \in I+\I \mathbb R$, all arguments of $\Gamma$ functions in the above formula
have positive real part.
Thus we can apply Stirling's formula \eqref{eq:stirl_form} and 
we find that, uniformly on compact subsets of $I +\I\mathbb R$,
\begin{multline}
\label{eq:Asymp_Fla2}
\frac{F_{\la_N}(\sqrt{N}(U+x_0))}{\Gamma(\sqrt N U-x+1)}\simeq 
(-1)^{({+}x_0+\eta(b_{i+1}+\dots+b_m-a_{i}-\dots-a_m))\sqrt{N}{+}x} \\
 \,\E^{x_0 (g(\sqrt{N})-\sqrt N)}\,
\E^{S_2(U) \sqrt{N}}  \left(-U\sqrt N\right)^{x} \left(U G(U+x_0)\right)^{-\frac{1}{2}},
\end{multline}
where 
\begin{multline*}
S_2(U)=g(-U) + \sum_{j=0}^{i-1} g (U+x_0-\eta\, a_j) - \sum_{j=1}^i g(U+x_0-\eta\, b_i)\\
- \sum_{j=i}^{m} g (-U-x_0+\eta\, a_j) + \sum_{j=i+1}^{m} g (-U-x_0+\eta\, b_j) .
\end{multline*}
For $\Re(U)< 0$, we have $g(-U)=-g(U) + i\pi U$ if $\Im(U) \ge 0$ 
and $g(-U)=-g(U) - i\pi U$ if $\Im(U) <  0$.
Hence, comparing the latter displayed equation with \eqref{eq:s1}, we have for $U \in I+\I \mathbb R$ that (recall that $I$ is a negative interval):
\[S_2(U) =\begin{cases}
  S_1(U) + \I\pi ( -x_0 +\eta (a_{i}{+}\dots{+}a_m{-}b_{i+1}{-}\dots{-}b_m)) &\text{ if $\Im(U) \ge 0$},\\
  S_1(U) - \I\pi ( -x_0 +\eta (a_{i}{+}\dots{+}a_m{-}b_{i+1}{-}\dots{-}b_m))  &\text{ if $\Im(U) <  0$}.
\end{cases}\]
In particular, recalling that $x_0 \sqrt N$ and $\eta \sqrt{N}$ are integers,
 we obtain
 \[ \E^{S_2(U) \sqrt{N}}=(-1)^{( -x_0 +\eta (a_{i}{+}\dots{+}a_m{-}b_{i+1}{-}\dots{-}b_m)) \sqrt{N}}\, \E^{S_1(U) \sqrt{N}}.\]
 Consequently, signs cancel out in \eqref{eq:Asymp_Fla2} and we get that 
 \eqref{eq:Asymp_Fla} is also valid uniformly on compact subsets of $I +\I \mathbb R$.
Doing similar reasoning for the other intervals $I$ listed above,
we conclude that \eqref{eq:Asymp_Fla} is valid uniformly on compact subsets of $\DS$.

\medskip

We now prove the bound $\frac{|\LHS(U)|}{|\RHS(U)|}= O \left( \exp(\eps \pi \sqrt N) \right)$
on compact subsets of 
\[\{V: \Re(V)<0 \text{ and } |\Im(V)| < \eps\}
\setminus \{\eta\, a_0-x_0,\eta b_0-x_0,\dots,\eta b_m-x_0,\eta a_m-x_0\}.\]
Thanks to the first part of the lemma, it only remains to prove that this quantity
is bounded on compact subsets of $I + \I [-\eps,\eps]$, 
for any negative interval $I$ of the form $(\eta\, a_{i-1}-x_0, \eta b_i-x_0)$. 
Let us take $U$ in $I + \I [-\eps,\eps]$. Using Euler's reflection formula (once more in the denominator than
in the numerator), we have
\begin{multline*}
  |\LHS(U)|=\frac{|F_{\la_N}(\sqrt{N}(U+x_0))|}{|\Gamma(\sqrt N U-x+1)|}=\frac{\left|\sin({-\pi\sqrt{N}} (U+x_0-\eta b_i) )\right|}{\pi}
 \big| \Gamma\!\left(x-\sqrt N U\right)\!\big|\\
\cdot \frac{\prod_{j=0}^{i-1} \big|\Gamma\big((U+x_0-\eta\, a_j)\sqrt N+1\big)\big|}{\prod_{j=1}^{{i-1}} \big|\Gamma\big((U+x_0-\eta\, b_j)\sqrt N+1\big)\big|}\,
\frac{\prod_{j=i}^m \big|\Gamma\!\left(-(U+x_0-\eta\, b_j)\sqrt N\right)\!\big|}{\prod_{j=i}^m \big|\Gamma\!\left(-(U+x_0-\eta\, a_j)\sqrt N\right)\!\big|}. 
\end{multline*}
Using the trivial bound $|\sin(z)| \le \exp(|\Im(z)|)$,
we get that
$\left|\sin({-\pi\sqrt{N}} (U+x_0-\eta b_i) )\right| \le \exp(\eps \pi \sqrt N)$ for any $U$ with $|\Im(U)| \le \eps$.
Moreover, the gamma factors can be controlled as above using Stirling's formula (all their arguments
have positive real parts for $U$ in $I+ \I [-\eps,\eps]$).
This proves the second statement in the lemma.

\medskip

For the third statement, we need to work on rectangles $I + \I [-\eps,\eps]$,
where $I$ is a positive interval of the form $I=( \eta b_i-x_0,\eta a_i-x_0)$.
Using again the reflection formula to get rid of the $\Gamma$ functions applied to arguments with negative
real parts yields:
\begin{multline*}
  |\LHS(U)|=\frac{|F_{\la_N}(\sqrt{N}(U+x_0))|}{|\Gamma(\sqrt N U-x+1)|}=\frac{\pi}{\left|\sin({-\pi\sqrt{N}} (U+x_0-\eta a_i) )\right|}
  \frac{1}{|\Gamma(\sqrt N U-x+1)|} \\
\cdot \frac{\prod_{j=0}^{i-1} \big|\Gamma\big((U+x_0-\eta\, a_j)\sqrt N+1\big)\big|}{\prod_{j=1}^{i} \big|\Gamma\big((U+x_0-\eta\, b_j)\sqrt N+1\big)\big|}\,
\frac{\prod_{j={i+1}}^m \big|\Gamma\!\left(-(U+x_0-\eta\, b_j)\sqrt N\right) \!\big|}{\prod_{j=i}^m \big|\Gamma\!\left(-(U+x_0-\eta\, a_j)\sqrt N\right)\!\big|}. 
\end{multline*}
Note that, this time, we applied the reflection formula once more in the numerator than in the denominator,
yielding an extra sine term at the denominator.
The bound $\frac{|\RHS(U)|}{|\LHS(U)|}= O \left(\exp(\eps \pi \sqrt N)\right)$  is then obtained bounding once again the sine term via the inequality
$|\sin(z)| \le \exp(|\Im(z)|)$, and using Stirling's formula for the $\Gamma$ factors.
\end{proof}

\subsection{The critical points of the action} 
\label{ssec:localization-critical-points}
The function $S$ in the estimate in \eqref{eq:equiv_integrand} is called the {\em action} of the model.
In order to perform a saddle point analysis of the double integral,
we look for the critical points of the action, i.e., complex solutions
of the critical equation $\frac{\partial S}{\partial U}(U) = 0$.
This critical equation  writes as
\begin{equation}
	\log U - \log(1-t_0) - \sum_{i=0}^m \log(x_0-\eta\, a_i+U)
	+\sum_{i=1}^m \log(x_0-\eta\, b_i+U) =0,
	\label{eq:critical_eq}
\end{equation}
and taking the exponential, we recover the critical equation~\eqref{eq:critical_intro} from the introduction, that is
\begin{equation}\label{eq:critical_eq2}
	U \, \prod_{i=1}^m (x_0-\eta\, b_i+U) = (1-t_0) \, \prod_{i=0}^m (x_0-\eta\, a_i+U).
\end{equation} 
Here is an analogue of \cite[Proposition 7.6]{petrov2014tilings} in our setting. Recall that $x_0 \in [\eta \,a_0,\eta\, a_m]$ and  $t_0\in[0,1]$. In the next analysis we exclude the cases $x_0=\eta\, a_i$ for $i\in\{0,1,\dots,m\}$; these cases will be treated later in \cref{rem:pat_cases}.
\begin{lemma}
	\label{lem:critical_points_general}
	Let $i_0\geq 1$ be such that $x_0 \in  (\eta \,a_{i_0-1} , \eta\, a_{i_0})$.
	Then for $i$ in $\{1,\dots,i_0-1\}$,
	the critical equation \eqref{eq:critical_intro} (see \eqref{eq:critical_eq2} for a closer reference) has 
	\begin{itemize}
		\item no root in the interval ${[}\eta\, a_{i-1} -x_0, \eta\, b_i -x_0)$,
		\item at least one real root in the interval $[\eta\, b_{i} -x_0, \eta\, a_i -x_0)$,
	\end{itemize}
	while, for $i$ in $\{i_0+1,\dots,m\}$, it has 
	\begin{itemize}
		\item no root in the interval ${[}\eta\, b_{i} -x_0, \eta\, a_i -x_0)$,
		\item at least one real root in $(\eta \,a_{i-1} -x_0, \eta\, b_{i} -x_0{]}$.
	\end{itemize}
\end{lemma}
\begin{proof}
	{If $t_0=1$ then the lemma statement is obvious,
		hence we assume that $t_0\in[0,1)$ (so that the factor $(1-t_0)$ is non-zero).}
	The lemma is easily obtained by looking at the signs of the left and right-hand sides 
	of \eqref{eq:critical_intro}, that is
	\begin{equation}\label{eq:left_right_fct}
		L(U) := U\prod_{j=1}^m(x_0-\eta\,b_j + U)
		\qquad \text{and}\qquad 
		R(U) := (1-t_0)\,\prod_{j=0}^m(x_0-\eta\,a_j+U).
	\end{equation}
	For instance, for $1 \leq i \leq i_0-1$ and $U \in {[}\eta\, a_{i-1} -x_0, \eta\, b_i -x_0)$, the quantity $L(U)$ has sign $(-1)^{m-i+2}$ (note that $U<0$ on this interval since $i \le i_0-1$), while $R(U)$ has sign $(-1)^{m-i+1}$. Hence \eqref{eq:critical_intro} has no root on this interval, as claimed.
	
	On the opposite, if $1 \leq i \leq i_0-1$, we look at the signs
	of the boundary points of the interval $[\eta\, b_{i} -x_0, \eta\, a_i -x_0)$ in order to prove the existence of a root. Note that
	\begin{itemize}
		\item $L(\eta\,b_i-x_0)=0$ and $R(\eta\,b_i-x_0)$ {is non-zero and} has sign $(-1)^{m-i+1}$;
		\item $L(\eta\,a_i-x_0)$ {is non-zero and} has sign $(-1)^{m-i+1}$ {(the factor $U= \eta\,a_i-x_0$ is negative since $i \le i_0-1$)} and $R(\eta\,a_i-x_0)=0$.
	\end{itemize}   
	Thus, the difference $L(U)-R(U)$ has to vanish between these two boundary points.
	The case $i\in\{i_0+1,\dots,m\}$ is then treated in a similar manner.
\end{proof}

The lemma locates $m-1$ solutions out of the $m+1$ solutions of the polynomial critical equation~\eqref{eq:critical_intro} (see \eqref{eq:critical_eq2} for a closer reference).
The two extra solutions are either both real or non-real complex conjugate. {Note that if $t_0=0$ then the polynomial equation \eqref{eq:critical_intro} is of degree $m$ and so all the $m$ solutions are real, while if $t_0=1$ then the polynomial equation \eqref{eq:critical_intro} has clearly only real solutions; we exclude these trivial cases from the next discussion.}
We recall from \cref{defn:liquid_region} that the
liquid region $L$ is the set of pairs $(x_0,t_0)\in {[\eta \,a_0,\eta\, a_m]\times [0,1]}$
such that the critical equation~\eqref{eq:critical_intro} has exactly two non-real solutions.
The following proposition gives us some information
about the shape of the liquid region,
and the localization of the critical points (in addition to those identified in
\cref{lem:critical_points_general}).
\begin{proposition}
\label{prop:t_regions}
  Fix $x_0 \in [\eta \,a_0,\eta\, a_m]$. As above, let $i_0\geq 1$
  be such that $x_0 \in  (\eta \,a_{i_0-1}, \eta \,a_{i_0})$ (note that we do not know how $ \eta\, b_{i_0}$ compares with $x_0$). 
  Then the following assertions hold:
  \begin{itemize}
  \item There exists $t_-=t_-(x_0)\geq 0$ such that the critical equation~\eqref{eq:critical_intro} has two real solutions (counted with multiplicities) outside the interval 
  $(\eta\, a_0 -x_0,\eta\, a_m -x_0)$ if and only if $t_0 \le t_-$.
  Moreover, $t_-(x_0)=0$ if and only if $x_0=0$.
  For $x_0<0$, these solutions are in $(-\infty,\eta\, a_0 -x_0)$,
  while for $x_0>0$, they lie in $(\eta \,a_m -x_0,+\infty)$.
  \item {There exists $t_+=t_+(x_0)\leq 1$ such that the critical equation~\eqref{eq:critical_intro} 
   has two real solutions (counted with multiplicities) inside the interval $(\eta\, a_{i_0-1}-x_0,\eta\, a_{i_0}-x_0)$
   if and only if $t_0 \ge t_+$. If this is the case, these solutions are inside the sub-interval $(0\wedge(\eta\, b_{i_0}-x_0),0\vee(\eta\, b_{i_0}-x_0))$.}
  Moreover, $t_+(x_0)=1$ if and only if $x_0=\eta\, b_{i_0}$.
\item Finally, if $t_0\in(t_-(x_0),t_+(x_0))$ and the critical equation~\eqref{eq:critical_intro} has only real solutions, then the two extra real solutions\footnote{in addition to the ones determined by \cref{lem:critical_points_general}.} (counted with multiplicities) are either both inside a negative interval $(\eta\,b_{j_0} -x_0,\eta\,a_{j_0} -x_0)$ for some $j_0<i_0$,
  or both inside a positive interval $(\eta\,a_{j_0}-x_0,\eta\,b_{{j_0}+1} -x_0)$ for some $j_0\geq i_0$. In particular, this interval contains three real solutions.
    \end{itemize}
\end{proposition}
We call the regions $\{0\leq t_0 \le t_-(x_0)\}$, $\{t_+(x_0) \leq t_0 \leq 1\}$ and $\{t_-(x_0)<t_0<t_+(x_0)\}$
the \emph{small} $t$, \emph{large} $t$, and \emph{intermediate} $t$ regions, respectively.
By \cref{defn:liquid_region}, the liquid region is included in the intermediate $t$ region,
 but the converse is not true;
 i.e.\ it might happen that $t_0$ is in the intermediate $t$ region, but not in the liquid region,
 which corresponds to the discontinuity phenomenon discussed in the introduction in \cref{ssec:limiting-surface}.

An illustrative example of the various results obtained in this section is given in \cref{fig:liquid_region_and_landscape}.

\begin{figure}[ht]
	\includegraphics[height=4cm]{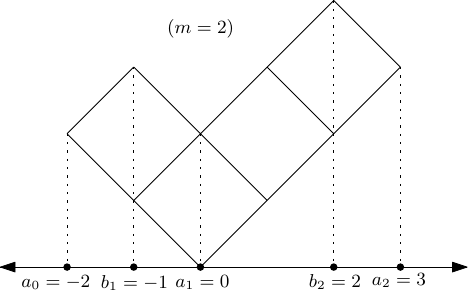}
	\hspace{1cm}
	\includegraphics[height=4cm]{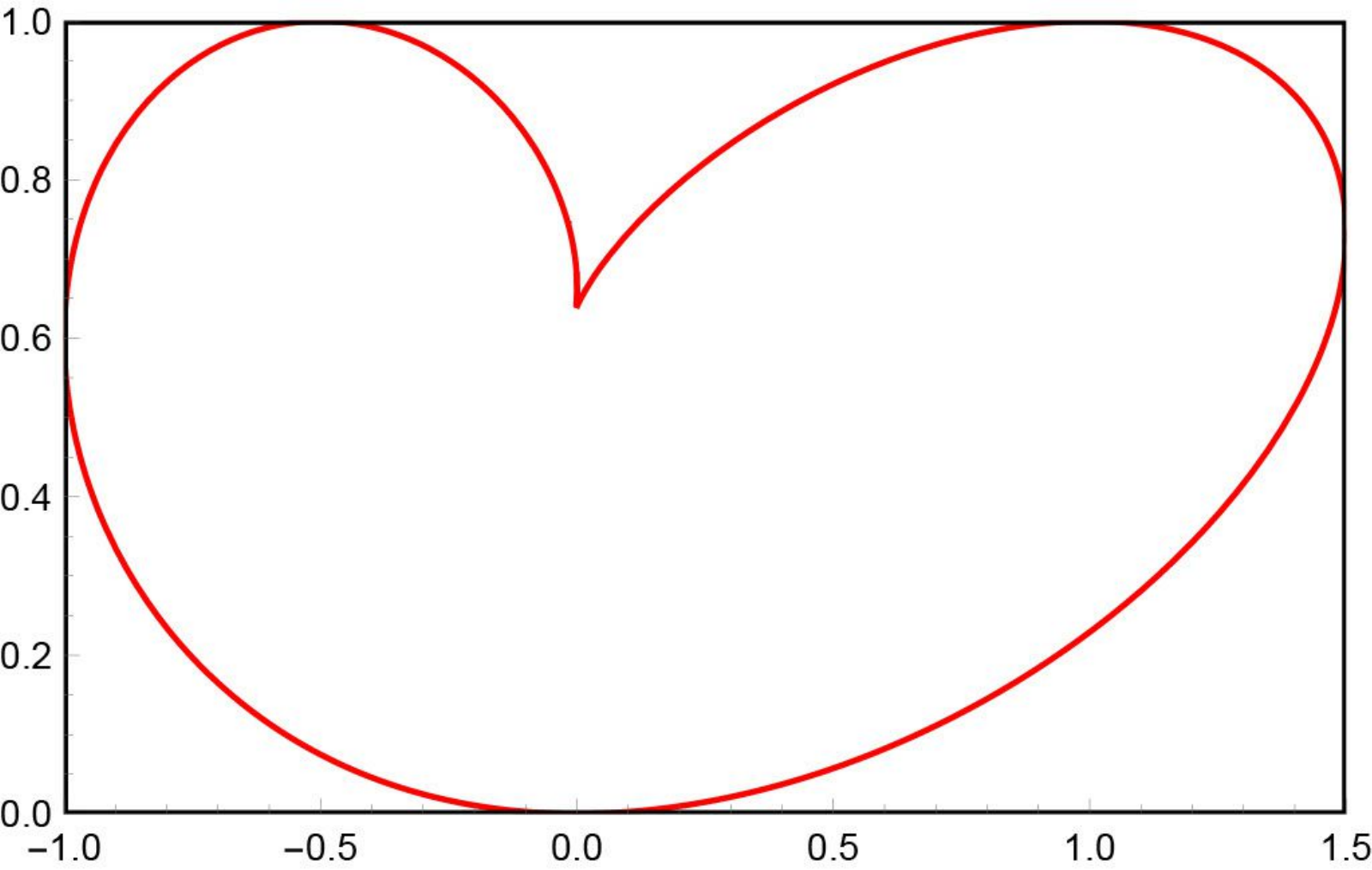}\vspace{1mm}
	\includegraphics[height=4cm]{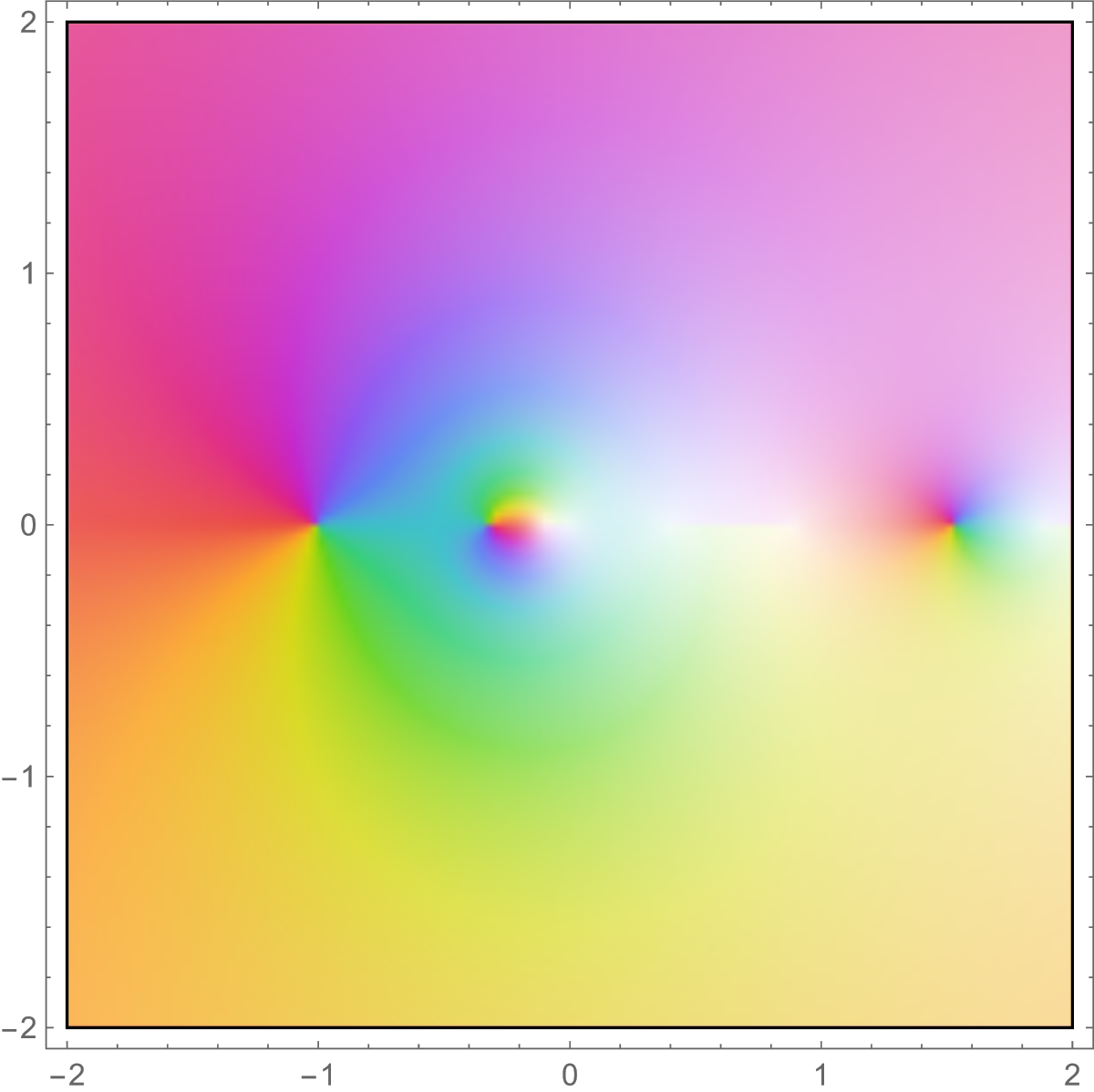}
	\includegraphics[height=4cm]{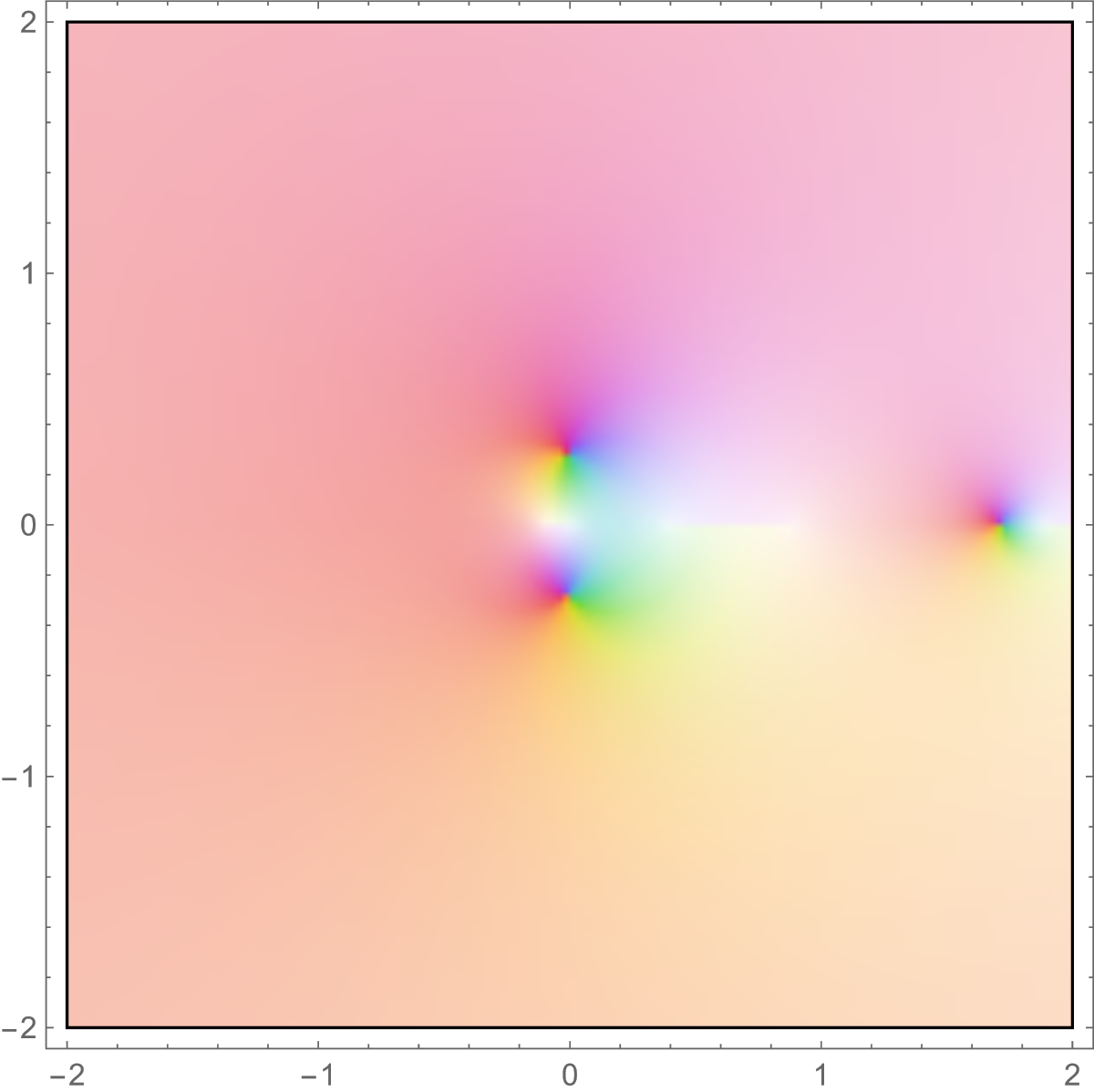}
	\includegraphics[height=4cm]{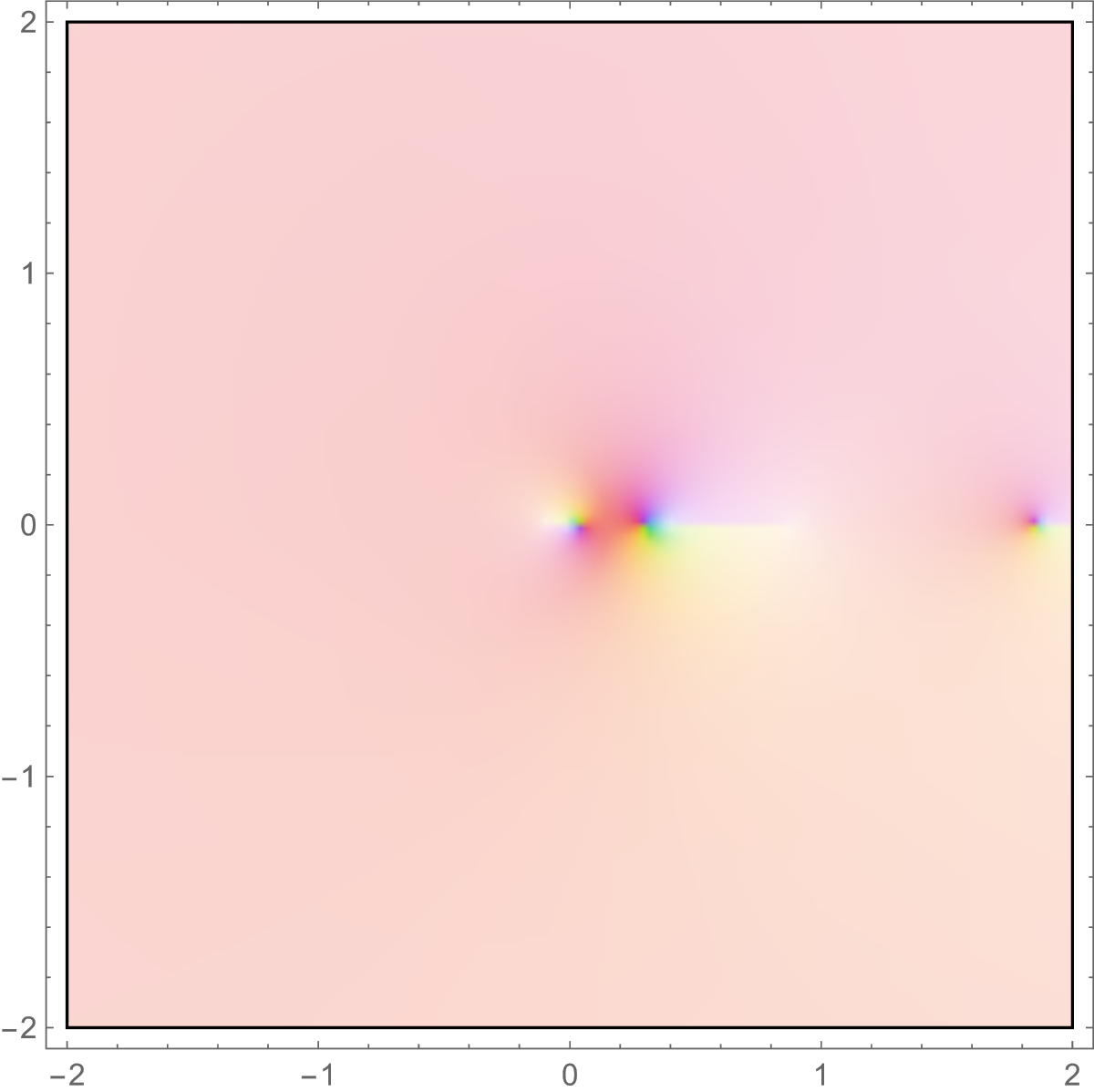}
	\includegraphics[height=4cm]{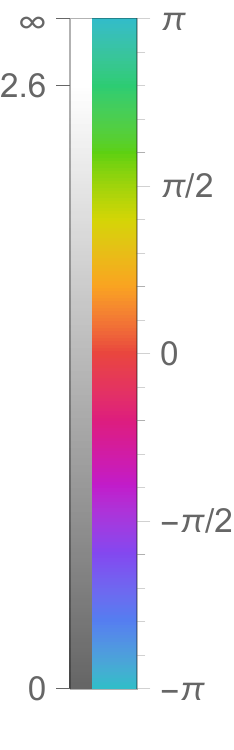}
	
	\vspace{0.2 cm}
	
	\includegraphics[height=0.98cm]{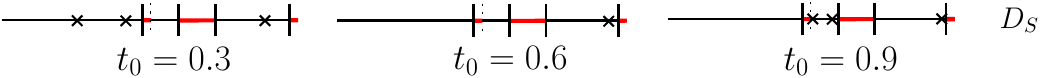}
	\captionsetup{width=\linewidth}
	\caption{\textbf{Top-left:} A Young diagram $\la^0$ with interlacing coordinates $a_0=-2<b_1=-1<a_1=0<b_2=2<a_2=3$. In particular, according to our notation $m=2$ and $\eta=1/\sqrt{ |\la^0|}=1/2$, so that $[\eta\, a_0, \eta\, a_m]=[-1,3/2]$. \textbf{Top-right:} In black the boundary of the region $[\eta\, a_0, \eta\, a_m]\times[0,1]$ for the points $(x_0,t_0)$. In red the frozen boundary of the liquid region $L$ corresponding to the diagram $\la^0$ (the liquid region $L$ is in the interior of the red curve). \textbf{Middle:} The landscapes of the function $\frac{\partial S}{\partial U}(U)$ associated with $\la^0$ from \eqref{eq:critical_eq}. In particular, we fixed $x_0=-0.9$ and we plotted the landscape of $\frac{\partial S}{\partial U}(U)$ for three different points $t_0\in\{0.3,0.6,0.9\}$. The plots are obtained using a cyclic color function over $\mathrm{Arg}\left(\frac{\partial S}{\partial U}(U)\right)$ as explained in the colored column in the legend on the right-hand side. Moreover, the shading of colors is based on $|\frac{\partial S}{\partial U}(U)|$ as explained in the gray column in the legend on the right-hand side. In particular, black points (at the center of the rainbow spirals) correspond to roots of the function $\frac{\partial S}{\partial U}(U)$. \textbf{Bottom:} We copied the portion of the sketch of $D_S$ from \cref{fig:diagram-for-DS} in the case $x_0=-0.9$ and we reported the position of the real roots (black crosses) from the middle pictures.\\
		\cref{lem:critical_points_general} predicts the existence of a real root of the critical
		equation in the interval $( \eta\, a_1-x_0, \eta\, b_2-x_0)=(0.9,1.9)$
		independently of the value of $t_0$, and we can indeed find such a root in all three cases (this is the right-most root).
		For small $t_0$, \cref{prop:t_regions} asserts that there are two additional roots
		in {$(-\infty,\eta\, a_0  -x_0)=(-\infty, -0.1)$}, which seems to be the case
		for $t_0=0.3$.
		For large $t_0$, again from \cref{prop:t_regions}, these two additional roots are
		in $(0,\eta\, b_1-x_0)=(0,0.4)$, which seems to be the case for $t_0=0.9$.
		For $t_0=0.6$, the point $(x_0,t_0)$ is in the liquid region
		and numerically, we indeed see two complex conjugate roots
		of the critical equation.\label{fig:liquid_region_and_landscape}}
\end{figure}

\begin{proof}[Proof of \cref{prop:t_regions}]
Recall that we can restrict ourselves to the case $t_0\in(0,1)$. We subdivide the proof into four main steps.

\smallskip 

\noindent{\emph{\underline{Step 1:} We fix $x_0>0$ and show that there exists $t_{-}^1(x_0)>0$ such that the critical equation~\eqref{eq:critical_intro}
has two real solutions (counted with multiplicities) in the interval $(\eta \,a_m -x_0,+\infty)$
if and only if $t_0 \le t_{-}^1(x_0)$, and no real solutions in this interval otherwise.}}

\smallskip 

First note that, as a consequence of \cref{lem:critical_points_general},
the critical equation~\eqref{eq:critical_intro} cannot have more than two real
solutions on this interval. Now, recall from \eqref{eq:left_right_fct} the functions $L(U)=L(U,x_0)$ and $R(U) = R(U,t_0)$ characterizing the left and right-hand sides of the critical equation~\eqref{eq:critical_intro}, i.e.\ 
\[ L(U)= U \, \prod_{i=1}^m (x_0-\eta\, b_i+U) \qquad\text{and}\qquad R(U)= (1-t_0) \, \prod_{i=0}^m (x_0-\eta\, a_i+U).\]
For $U$ tending to $\eta \,a_m -x_0$ from the right,
$L(U)$ stays positive while $R(U)$ tends to $0$, so that $L(U)-R(U)$ is positive.
For $U$ tending to $+\infty$, the difference $L(U)-R(U)$
is asymptotically equivalent to $t_0 U^{m+1}$, and hence is also positive {(since $t_0>0$)}.
Thus $L(U)-R(U)$ has two zeroes (counted with multiplicities) in the interval $(\eta \,a_m -x_0,+\infty)$
if and only if
\[\theta(t_0):= \inf_{U \geq \eta\,a_m-x_0} \left\{L(U)-R(U,t_0)\right\}, \qquad t_0\in[0,1],\]
is non-positive; and $L(U)-R(U,t_0)$ has no zeroes in this interval otherwise.
{Since $\theta(t_0)$ is continuous and non-decreasing in $t_0$, if we show that $\theta(t_0)$ is non-positive for $t_0$ small enough, this would complete the proof of Step 1.} Note that, as $U$ tends to $+\infty$, we have that
\begin{align*}
 L(U) &= U^{m+1} + \left(mx_0-{\eta}\sum_{i=1}^m b_i\right)U^m+ O(U^{m-1}) ,\\ 
 R(U,0) &= U^{m+1} + \left( (m+1)x_0 - {\eta}\sum_{i=0}^{m} a_i\right)U^m + O(U^{m-1}).
 \end{align*}
 Using the identity $\sum_{i=0}^m a_i = \sum_{i=1}^m b_i$  in \eqref{eq:int_rel}, this implies 
 that when $x_0>0$, $L(U) - R(U,0)$ tends to $-\infty$ as $U$ goes to $+\infty$.
 This implies that $\theta(t_0)<0$ for $t_0$ small enough,
and conclude the proof of Step 1.
 
\smallskip 
 
\noindent{\emph{\underline{Step 2:} We complete the proof of the first item in the proposition statement.}} 

\smallskip 

By symmetry, if $x_0<0$ there exists $t_{-}^2(x_0)>0$ such that the critical equation has two zeroes (counted with multiplicities) in the interval $(-\infty,\eta \,a_0 -x_0)$ if and only if $t_0 \le t_{-}^2(x_0)$, and no real solutions in this interval otherwise.

From \cref{lem:critical_points_general}, the critical equation~\eqref{eq:critical_intro} cannot have at the same time
two solutions in the intervals $(-\infty,\eta \,a_0 -x_0)$ and $(\eta \,a_m -x_0,+\infty)$.
Therefore $t_{-}^1(x_0)=0$ for $x_0<0$ and $t_{-}^2(x_0)=0$ for $x_0>0$.
Setting $t_-(x_0)= t_{-}^1(x_0)$ for $x_0>0$ and $t_-(x_0)= t_{-}^2(x_0)$ for $x_0<0$,
we have proven the first item in the Proposition statement for $x_0 \ne 0$.
The case $x_0 = 0$ follows from continuity arguments.

\smallskip 
 
\noindent{\emph{\underline{Step 3:} We complete the proof of the second item in the proposition statement.}}

\smallskip 

To fix ideas and notation, let us suppose that 
$$\eta\,a_{i_0-1} - x_0 <  0 \le \eta\,b_{i_0}-x_0 < \eta\,a_{i_0} - x_0;$$
 the other case $\eta\,b_{i_0}-x_0 < 0$ being symmetric; 
 while the case $\eta\,a_{i_0}-x_0=0$ needs minor adjustments
 discussed below.
 The left-hand side $L(U)$ vanishes only at the edges of the
 interval $[0,\eta\,b_{i_0}-x_0]$,
 while the  right-hand side $R(U)$  has  sign $(-1)^{m-i_0+1}$.
 Thus the difference $L(U) - R(U)$ has two zeroes on $[0,\eta\,b_{i_0}-x_0]$
 if and only if
 \[\theta(t_0) :=\sup_{0 \leq U \leq \eta\,b_{i_0}-x_0} \left\{(-1)^{m-i_0+1}(L(U) - R(U,t_0) )\right\}\]
 is non-negative, and no zeroes otherwise (it cannot have more than two zeroes,
 as a consequence of \cref{lem:critical_points_general}).
The quantity $\theta(t_0)$ is increasing with $t_0$
and $\theta(1)= \sup_{0 \leq U \leq \eta\,b_{i_0}-x_0} (-1)^{m-i_0+1} L(U) \ge 0$,
which implies the existence of $t_+(x_0)\leq 1$ as in the proposition statement.

If $x_0=\eta\, b_{i_0}$, then, for all $t_0<1$, one has
\[\theta(t_0) =  (-1)^{m-i_0+1} (L(0)-R(0)) = - (1-t_0) |R(0)| <0,\]
proving $t_+(x_0)=1$.
On the other hand, if $x_0 \ne \eta\, b_{i_0}$ one has $\theta(t_0)>0$
as soon as 
\[1-t_0 < \frac{\sup_{0 \leq U \leq \eta\,b_{i_0}-x_0} |L(U)|}{\sup_{0 \leq U \leq \eta\,b_{i_0}-x_0} |R(U)|},\]
proving that $t_+(x_0)<1$.

\smallskip 

\noindent{\emph{\underline{Step 4:} We complete the proof of the third item in the proposition statement.}}

\smallskip 

Note that since $t_0\in(t_-(x_0),t_+(x_0))$, as a consequence of the first two items in the proposition statement, there are no real solutions outside the interval $(\eta\, a_0 -x_0,\eta\, a_m -x_0)$ and inside the interval $(\eta\, a_{i_0-1}-x_0,\eta\, a_{i_0}-x_0)$. Combining this observation with the sign discussion in the proof of \cref{lem:critical_points_general}, we see that the only possibility is that either an interval of the form $(\eta\,b_{j_0} -x_0,\eta\,a_{j_0} -x_0)$ for some $j_0<i_0$ 
or an interval of the form $(\eta\,a_{j_0}-x_0,\eta\,b_{{j_0}+1} -x_0)$ for $j_0>i_0$ contains three roots of the critical equation \eqref{eq:critical_intro}.
\end{proof}

\begin{remark}\label{rem:pat_cases}
	As already mentioned, the case where $x_0=\eta\, a_{i_0}$ is not considered in the above results,
	because it needs some adjustment.
	In this case, the factor $U$ in $L(U)$ cancels out with the factor $(x_0-\eta\, a_{i_0} +U)$
	in $R(U)$ and the critical equation~\eqref{eq:critical_intro} (see \eqref{eq:critical_eq2} for a closer reference) has degree only $m$.
	An analogue of \cref{lem:critical_points_general} states that 
	the critical equation~\eqref{eq:critical_intro} has at least one root in each interval $(\eta\, b_i-x_0,\eta\, a_i-x_0)$
	for $i \in \{1,\dots,i_0-1\}$ 
	and in each interval $(\eta\, a_{i-1}-x_0,\eta\, b_i -x_0)$ for $i \in \{i_0+2,\dots,m\}$.
	
	If $i_0 \notin \{0,m\}$, this locates $m-2$ roots out of the $m$ roots of the equation.
	As in the generic case, the location of the two last roots is 
	partially described by an analogue of \cref{prop:t_regions}.
	The first item of \cref{prop:t_regions} holds true
	in the case $x_0=\eta\, a_{i_0}$ without any modification.
	For the second item, 
	it holds that the critical equation~\eqref{eq:critical_intro}
	has two real solutions (counted with multiplicities) inside the interval 
	{$(\eta\, b_{i_0}-x_0,\eta\, b_{i_0+1}-x_0)$}
	if and only if $t_0$ is at least equal to some $t_+(x_0)$.
	The proof is a simple adaptation of that of \cref{prop:t_regions}.
	
	Finally, if $x_0=\eta\, a_{0}$ (resp.\ $x_0=\eta\, a_m$), the critical equation~\eqref{eq:critical_intro} has degree $m$,
	and has at least one root in each interval $(\eta\, a_{i-1}-x_0,\eta\, b_i -x_0)$ for $i \ge 2$
	(resp.\ $(\eta\, b_i-x_0,\eta\, a_i-x_0)$ for $i \le m-1$). In both case, it has at least $m-1$ real roots, and hence cannot have complex roots.
	The vertical lines $\{\eta\, a_{0}\} \times [0,1]$ and $\{\eta\, a_{m}\} \times [0,1]$ thus entirely lie in the frozen region, {as well as the horizontal lines $[\eta\, a_0,\eta\, a_m] \times \{0\}$ and $[\eta\, a_0,\eta\, a_m] \times \{1\}$ (recall the discussion above \cref{prop:t_regions})}.
\end{remark}

\subsubsection{The shape of the liquid region}

We conclude this section proving \cref{prop:Burgers_eq} from the introduction, and giving an alternative description of the liquid region (see \cref{prop:description_liquid_region} below). These results have all rather standard proofs, which we include for the sake of completeness.

\begin{proof}[Proof of \cref{prop:Burgers_eq}]
  \label{proof:burger}
	Recall that $U_c(x,t)$ is the unique solution with positive imaginary part of the critical equation $\frac{\partial S}{\partial U}(U) = 0$, where $S(U)$ is as in \eqref{eq:action}. Setting\footnote{This substitution is not strictly needed here but makes the computations nicer and also useful for the proof of \cref{prop:description_liquid_region}.} $\widetilde U_c(x,t):=U_c(x,t) + x$, we get that $\widetilde U_c(x,t)$ solves the equation
	\begin{equation}\label{eq:new_crit_eq}
		\log (\widetilde U_c-x) - \log(1-t) - \sum_{i=0}^m \log(\widetilde U_c-\eta\, a_i)
		+\sum_{i=1}^m \log(\widetilde U_c-\eta\, b_i) =0.
	\end{equation}
	Differentiating in $x$ and $t$ the above equation, we obtain that
	\begin{equation*}
		\begin{cases}
          \frac{-1}{\widetilde U_c-x}=( \widetilde U_c)_x \cdot \big(\Sigma(\widetilde U_c) -\tfrac1{\widetilde U_c-x}\big),\\
			\frac{1}{1-t}=(\widetilde U_c)_t \cdot \big(\Sigma(\widetilde U_c)-\tfrac1{\widetilde U_c-x}\big).
		\end{cases}
	\end{equation*}
	where $\Sigma(s)=\sum_{i=0}^m\frac{1}{s-\eta\, a_i}-\sum_{i=1}^m\frac{1}{s-\eta\, b_i}$. 
    The quantity $\Sigma(\widetilde U_c) -\tfrac1{\widetilde U_c-x}$ corresponds to $\frac{\partial^2 S}{\partial U^2}$,
    and is thus nonzero in the liquid region; otherwise $U_c$ would be a double root of the critical equation,
    which is impossible since the latter has at most 2 non-real conjugate solutions, counted with multiplicities (\cref{lem:critical_points_general}).
    Consequently, $\widetilde U_c(x,t)$ solves the equation
	\begin{equation*}
	    -\frac{(\widetilde U_c)_t}{\widetilde U_c-x} = \frac{ (\widetilde U_c)_x}{1-t}.
      \end{equation*}
	Substituting $U_c(x,t)= \widetilde U_c(x,t) - x$ in the equation above, we get  \eqref{eq:burgers}.
\end{proof}

\begin{proposition}\label{prop:description_liquid_region}
	{We have the following equivalent description of the liquid region $L$:
		\begin{equation*}
			L=\left\{(x,t)\in[\eta\, a_0, \eta\, a_m] \times [0,1]\,\middle|\, \Disc_{U}(P_{x,t})<0\right\},
		\end{equation*}
		where $\Disc_{U}(P_{x,t})$ denotes the discriminant\footnote{Since $L$ is defined in terms of the sign of the discriminant, let us recall the standard convention of normalisation of the discriminant: $\Disc_U(P_{x,t}) = (-1)^{\frac{m(m+1)}{2}} t^{-1}\,\mathrm{Res}_U(P_{x,t},P_{x,t}')$, where $\mathrm{Res}$ is the resultant.} of the polynomial $P_{x,t}(U)$ appearing in the critical equation~\eqref{eq:critical_intro}.
		As a consequence, $L$ is an open subset of $[\eta\, a_0, \eta\, a_m] \times [0,1]$ and the boundary of $L$, called \emph{frozen boundary curve}, is given by
		\begin{equation}\label{eq:boundary_liquid}
			\partial L=\left\{(x,t)\in[\eta\, a_0, \eta\, a_m] \times [0,1]\,\middle|\, \Disc_{U}(P_{x,t})=0\right\}.
		\end{equation}
	Moreover, the frozen boundary curve $\partial L$ can be parametrized by $-\infty<s<\infty$, as follows:
	\begin{equation*}
		\begin{cases}
			x(s)=s-\frac{1}{\Sigma(s)},\\
			t(s)=1-\frac{G(s)}{\Sigma(s)},
		\end{cases}
	\end{equation*}
	where $\Sigma(s):=\sum_{i=0}^m\frac{1}{s-\eta\, a_i}-\sum_{i=1}^m\frac{1}{s-\eta\, b_i}$ and $G(s):=\frac{\prod_{i=1}^{m} (s-\eta\, b_i)}{\prod_{i=0}^{m} (s-\eta\, a_i)}$. The tangent vector $(\dot{x}(s),\dot{t}(s))$ to the frozen boundary curve is parametrized by:
	\begin{equation*}
		\begin{cases}
			\dot x(s)=1+\frac{\dot \Sigma(s)}{\Sigma(s)^2},\\
			\dot t(s)=G(s)\left(1+\frac{\dot \Sigma(s)}{\Sigma(s)^2}\right).
		\end{cases}
	\end{equation*}
	In particular, it has slope $\frac{\dot t(s)}{\dot x(s)}=G(s)$.
	}
\end{proposition}

\begin{proof}[Proof of \cref{prop:description_liquid_region}]
	{Recall that for a polynomial with real coefficients, its discriminant is positive (resp.\ zero) if and only if the number of non-real roots is a multiple of 4 (resp.\ the polynomial has a multiple root). 
	Since \cref{lem:critical_points_general} shows that at least $m-1$ roots of the critical equation \eqref{eq:critical_intro} are simple and real, we conclude that the existence of two non-real solutions for the critical equation \eqref{eq:critical_intro} is equivalent to $\Disc_{U}(P_{x,t})<0$, where we recall that $\Disc_{U}(P_{x,t})$ denotes the discriminant of the polynomial $P_{x,t}(U)$ appearing in the critical equation \eqref{eq:critical_intro}. 
	$\Disc_{U}(P_{x,t})$ is itself a polynomial in $(x,t)$, which implies that the
	liquid region $L$ is an open subset of  $[\eta\, a_0, \eta\, a_m] \times [0,1]$ and its boundary is described by \eqref{eq:boundary_liquid}.\medskip

	We now turn to the proof of the claimed parametrization for the frozen boundary curve. Note that, if a point $(x,t)\in [\eta\, a_0, \eta\, a_m] \times [0,1]$ approaches the frozen boundary curve $\partial L$, the solution $U_c(x,t)$ of the critical equation becomes real and merges with $\overline{U_c(x,t)}$. Equivalently, the frozen boundary can be characterized as the set of points $(x,t)\in [\eta\, a_0, \eta\, a_m] \times [0,1]$ such that the action $S$ in \eqref{eq:action} has a double critical point. For computational reasons, it is convenient to take $\widetilde U_c(x,t)=U_c(x,t) + x$ as the parameter to describe $\partial L$, and express $x$ and $t$ in terms of $\widetilde U_c(x,t)$.
	Hence, we need to impose that $\widetilde U_c(x,t)$ solves the equation $\frac{\partial \widetilde S}{\partial U}(U)=0$ (written in \eqref{eq:new_crit_eq}) and the equation $\frac{\partial^2 \widetilde S}{\partial U^2}(U)=0$. We obtain:
	\begin{equation*}
		\begin{cases}
			\frac{\widetilde U_c-x}{1-t}=(G(\widetilde U_c))^{-1},\\
			\frac{1}{\widetilde U_c-x}=\Sigma(\widetilde U_c).
		\end{cases}
	\end{equation*}
	Solving the above linear system for $x$ and $t$, we get the parametrization claimed in the proposition statement (with $s=\widetilde U_c$). The claims for the tangent vector follow from standard computations, noting that $\dot G(s)=-\Sigma(s) \cdot G(s)$.}
\end{proof}

\subsection{The imaginary and real parts of the action on the real line}
\label{ssec:im_real_line}

The imaginary and real part of the action $S$ (introduced in~\eqref{eq:action}) on the real line play a key-role in our analysis in the next sections. Hence we analyze them here.

In what follows, we keep working with the principal branch of the logarithm,
defined on $\mathbb C \setminus \{0\}$ and continuous on $\mathbb C \setminus \mathbb R_-$; with the convention that, for negative real $x$, we have $\log(x)=\log(-x)+\I\pi$.

For $u\in \R$, we have
\begin{equation}
	\Im S(u) = - \pi u^{-} + \sum_{i=0}^m \pi(u+x_0-\eta\, a_i)^{-}
	- \sum_{i=1}^m \pi(u+x_0-\eta\, b_i)^{-},
	\label{eq:im_action}
\end{equation}
where $x^{-} = \max(0,-x)$. 
In particular, $\Im S(u)$
is piecewise affine, has value $-\pi x_0$
for $u<\eta\, a_0-x_0$ (recall from \eqref{eq:int_rel} that $\sum_{i=0}^m a_i- \sum_{i=1}^m b_i=0$),
has slope alternatively $+\pi$ and $0$ for $u<0$,
then alternatively $0$ or $-\pi$ for $u>0$,
and finally takes value $0$ for $u> \eta\, a_m -x_0$. See \cref{fig:Imaginary_action} for an example.

On the other hand, for $u\in \R$, we have
\begin{equation}
	\Re S(u) = \widetilde g(u) - u \cdot \ln(|1-t_0|)- \sum_{i=0}^m \widetilde g (x_0-\eta\, a_i+u) + \sum_{i=1}^m \widetilde g (x_0-\eta\, b_i+u),
	\label{eq:re_action}
\end{equation}
where $\widetilde g(u)=u\cdot\ln (|u|)$ and we use the convention that $\widetilde g(0)=0$. See again \cref{fig:Imaginary_action} for an example. In particular, the map $u \to \Re S(u)$ is well-defined and continuous on the real line. 
Moreover, since $\widetilde g'(u)=\ln(|u|)+1$, the map $u \to \Re S(u)$ is differentiable as a function $\mathbb R \to \mathbb R$
except at the points $\eta\,a_i-x_0$ (where it has a positive infinite slope)
and at the points $0$ and $\eta\,b_i-x_0$ (where it has a negative infinite slope).
Its derivative vanishes exactly when $u$ satisfies
\[
|u|  \prod_{i=1}^m |u+x_0-\eta\, b_i| = (1-t_0)  \prod_{i=0}^m |u+x_0-\eta\, a_i|,
\]
i.e.\ when $u$ satisfies the critical equation~\eqref{eq:critical_intro} or the companion equation
\begin{equation}
	u  \prod_{i=1}^m (u+x_0-\eta\, b_i) = - (1-t_0)  \prod_{i=0}^m (u+x_0-\eta\, a_i).
	\label{eq:companion_critical}
\end{equation}

\begin{figure}
	\includegraphics[width=7cm]{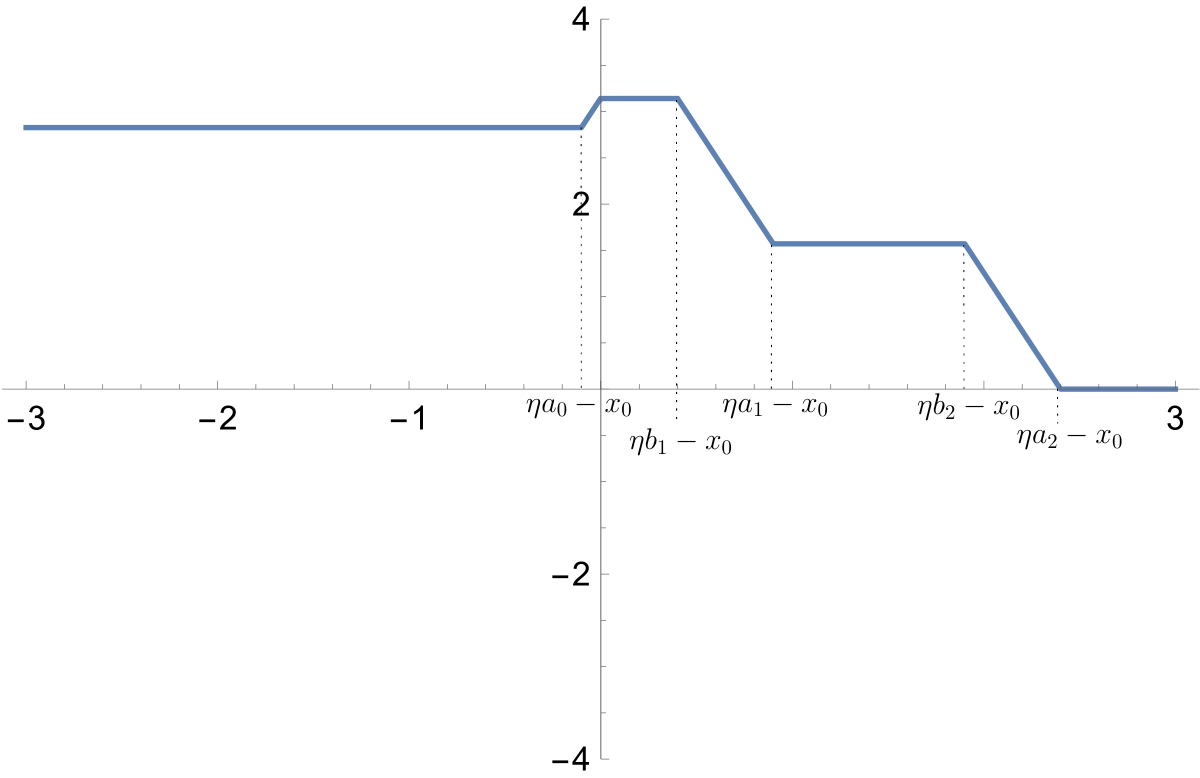}
	\hspace{1cm}
	\includegraphics[width=7cm]{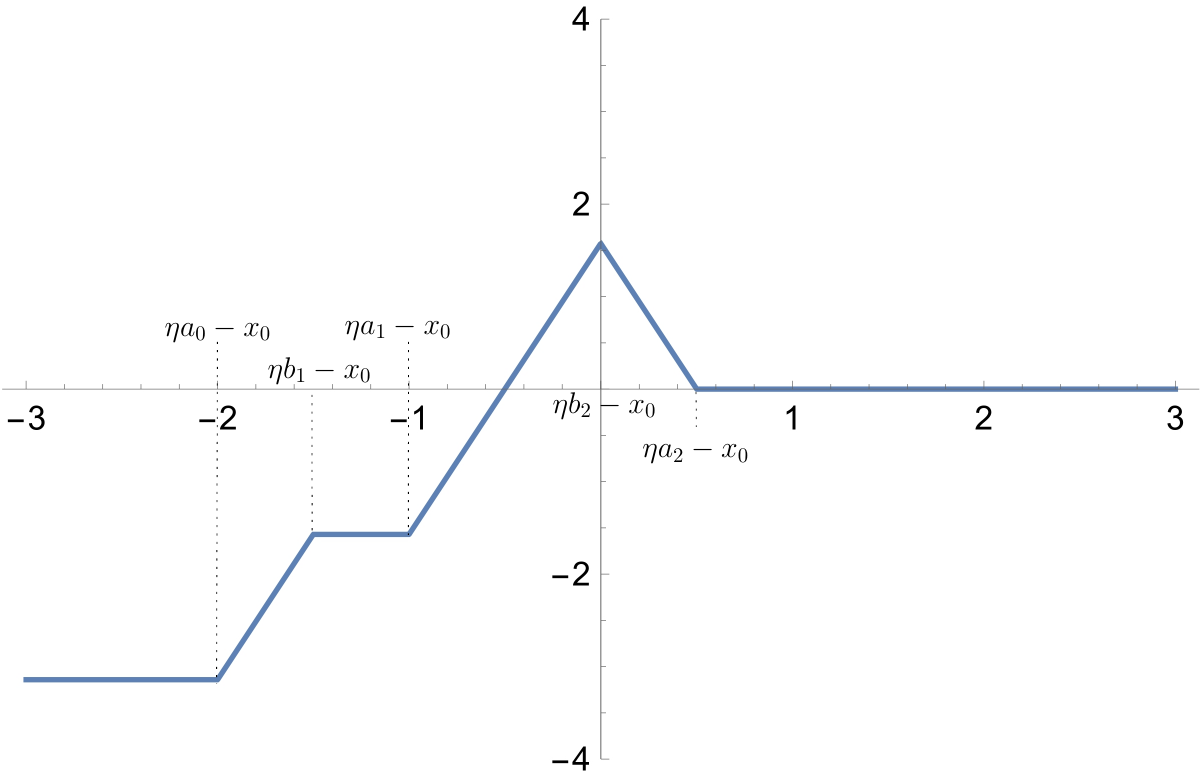}
	\captionsetup{width=\linewidth}
	\includegraphics[width=7cm]{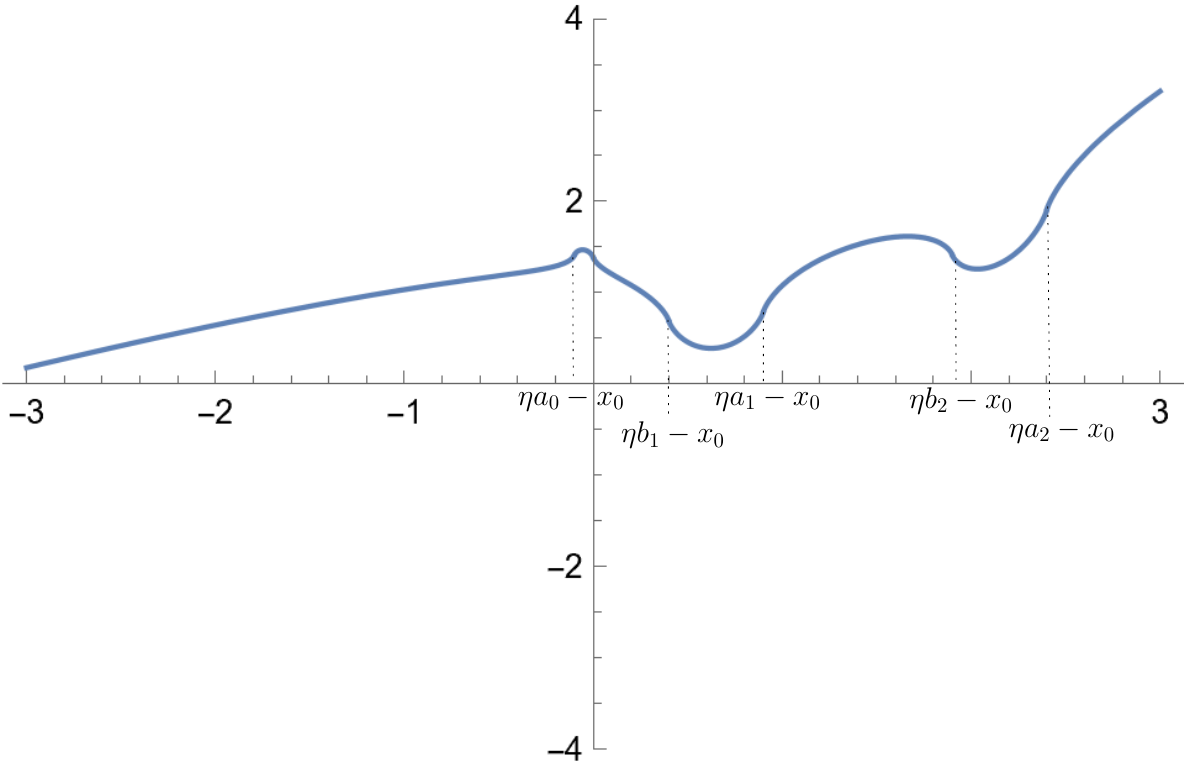}
	\hspace{1cm}
	\includegraphics[width=7cm]{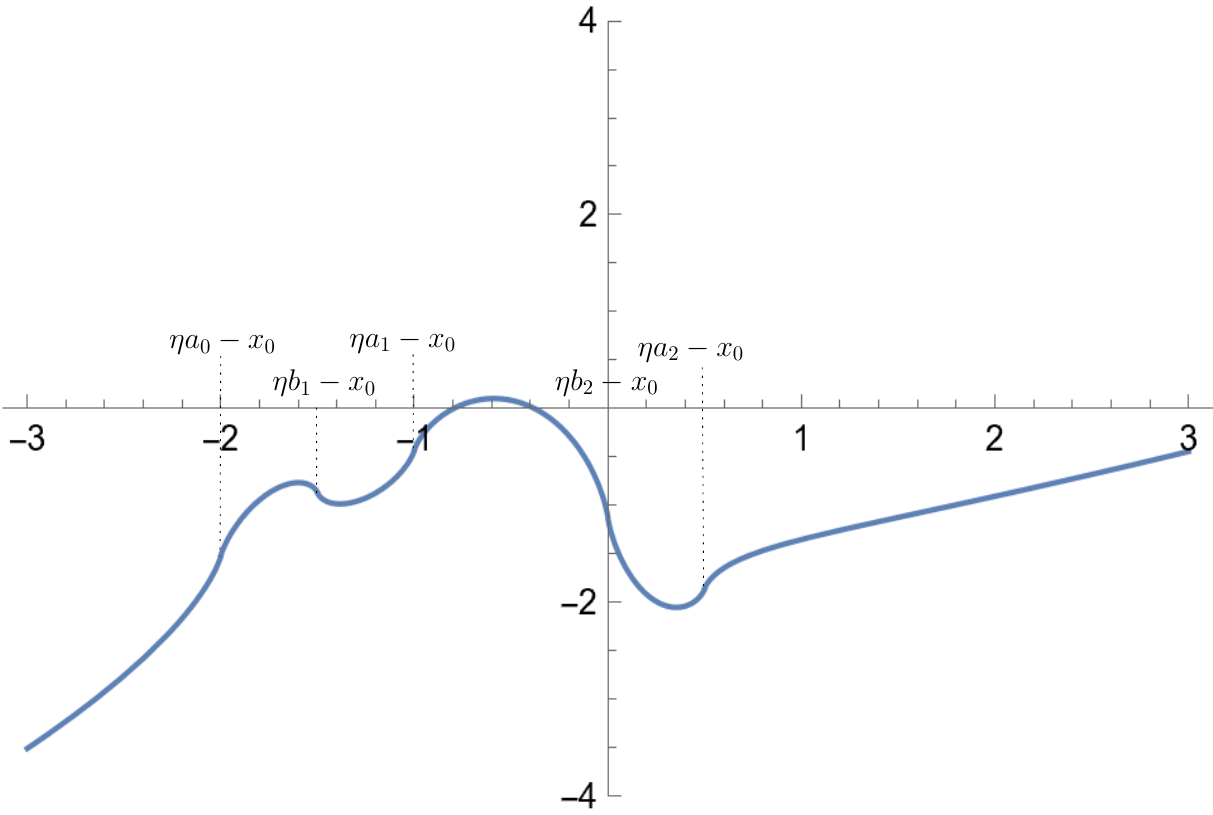}
	\captionsetup{width=\linewidth}
	\caption{\textbf{Top:} Two plots of the function $\Im S(u)$  in \eqref{eq:im_action} in the specific example of \cref{fig:liquid_region_and_landscape}. On the left we set $x_0=-0.9$, while on the right $x_0=1$. 
	\textbf{Bottom:} Two plots of the function $\Re S(u)$  in \eqref{eq:re_action} in the same specific example of \cref{fig:liquid_region_and_landscape}. Again, on the left we set $x_0=-0.9$ and $t_0=0.5$, while on the right $x_0=1$ and $t_0=0.5$.\label{fig:Imaginary_action}}
\end{figure}

\section{Asymptotics for the kernel in the liquid region}\label{sect:liquid_region}
Within this section, we assume that $(x_0,t_0)$ lies inside the liquid region, i.e.\ that the critical equation~\eqref{eq:critical_intro} has two non-real solutions,
denoted by $U_c$ and $\overline{U_c}$ and with the convention that $\Im U_c >0$.

\subsection{Landscape of the action}
 
 As a preparation for the saddle point analysis, we also need to understand to some extent the real part $\Re S(U)$ of the action $S$ introduced in \cref{eq:action}. 
 In particular, we are interested in the shape of the region 
 $$\{\Re S(U) > \Re S(U_c)\}:=\Set{U\in\mathbb C | \Re S(U) > \Re S(U_c)}.$$
 We will use similar notation for various regions below.
 We invite the reader to compare the following discussion with the pictures in \cref{fig:landscape_ReS}.
 
 Since $S$ is analytic around $U_c$ (and more generally on $\mathbb C \setminus (-\infty,\eta \,a_m-x_0)$) and $U_c$ is a simple critical point of $S$,
 we have that, locally around $U_c$
 \[S(U) = S(U_c)+\tfrac{S''(U_c)}2 \,(U-U_c)^2 +O((U-U_c)^3).\]
 In particular,\footnote{Most of the claims in this paragraph follows by analogy
 with the landscape of the complex function $F(U)=U^2$ around $0$. For instance,
 the four curves corresponding to the level sets $\{\Re S(U) = \Re S(U_c)\}$ are
 the analogue of the four curves corresponding to the level sets $\{\Re F(U) =
\Re F(0)\}=\{U=x+iy\in\mathbb C:(x+y)(x-y)=0\}$.} there are \emph{eight special  curves} leaving $U_c$ (see for instance the left-hand side of \cref{fig:landscape_ReS}): four curves corresponding to the level sets $\{\Re S(U) = \Re S(U_c)\}$
 and four other curves corresponding to the level sets $\{\Im S(U) = \Im S(U_c)\}$.
 We will refer to these curves as \emph{real} or \emph{imaginary level lines}, respectively.
 The four real level lines split the neighborhood of $U_c$ into four regions, belonging
 alternatively to the set $\{\Re S(U) > \Re S(U_c)\}$ or $\{\Re S(U) < \Re S(U_c)\}$,
 each of these regions containing one imaginary level line.
 
 Since $S$ is analytic with a non-zero derivative on $\set{V | V \ne U_c,\, \Im(V)>0}$,
  locally around such $V$ with $\Re S(V) = \Re S(U_c)$ (resp.~with $\Im S(V) = \Im S(U_c)$),
 the level set  $\{\Re S(U) = \Re S(U_c)\}$ (resp.~the level set $\{\Im S(U) = \Im S(U_c)\}$) looks like a single {simple} curve.
 Therefore the real and imaginary level lines leaving $U_c$ go either to the real line, or to infinity.
 Moreover, these lines cannot cross in the closed upper half-plane.
 Indeed, following a real level line, the analyticity of $S$ implies
 that $U \mapsto \Im S(U)$ is strictly monotone: a local extremum would violate the open mapping theorem. 
 Therefore starting from $U_c$ and following a real level line in the set  $\{\Re S(u) = \Re S(U_c)\}$,
 we never reach a point $V$ with $\Im S(V) = \Im S(U_c)$.

\begin{lemma}
	Exactly three of the real level lines leaving $U_c$ go to the real line.
\end{lemma}

\begin{proof}
	Since $S(U) \sim |\log(1-t_0)| \,U$ as $|U|$ tends to infinity (as already remarked in \eqref{eq:as_action}), we have that $\Re S(U) < \Re S(U_c)$ when $\Re U$ goes to $-\infty$ with a fixed imaginary part.
	Likewise, $\Re S(U) > \Re S(U_c)$ when $\Re U$ goes to $+\infty$ with a fixed imaginary part. Therefore, there
	is an odd number of real level lines going to the real line.
	
	Assume that there is only one. Then three real level lines are going to infinity.
	Because of the asymptotics $S(U) \sim |\log(1-t_0)| \,U$, real level lines going to infinity
	should be asymptotically inside the region $\Set{y\geq |x|}:=\Set{x+iy | y\geq |x|}$ (otherwise we would have $\lim_{|U| \to +\infty} \Re S(U)=+\infty$ and this would be in contradiction with the definition of real level line).
	These three real lines determine two unbounded regions, each of them containing an imaginary level line.
	In particular, these two imaginary level lines go to infinity, inside the region $\Set{ y\geq |x|}$.
	But inside this region, we have that $\lim_{|U| \to +\infty} \Im S(U)=+\infty$
	since $S(U) \sim |\log(1-t_0)| \,U$. This is in contradiction with the definition of imaginary level line.
	We, therefore, conclude that there are exactly three of the real level lines leaving $U_c$ and going to the real line.
\end{proof}
 
 We call $A \le B \le C$ the intersection points of these three
 real level lines and the real axis.
 Also, let $D$ and $E$ be the intersection points of the imaginary level lines in between
 these three level lines and the real axis.
 Since real and imaginary level lines do not intersect, we have $A<D<B<E<C$.

\begin{lemma}\label{lem:0_inside}
With the above notation, $D<0<E$.
Moreover there is no real $x$ in $(-\infty,A)\cup(C,+\infty)$
such that $\Im S(x) = \Im S(U_c)$.
\end{lemma}
\begin{proof}
Since $D$ and $E$ belong to an imaginary level line leaving $U_c$,
and since $\Im S(u)$ is continuous on the closed upper half-plane,
we have $\Im S(D)=\Im S(E)= \Im S(U_c)$.
Furthermore, 
$\Im S(U_c)$ is distinct from the values $\Im S(A)$, $\Im S(B)$, $\Im S(C)$
since, as explained above, $\Im S$ is strictly monotone
on the real level lines going from $U_c$ to 
$A$, $B$ or $C$.
Looking at the graph of the function $u \mapsto \Im S(u)$ on the real line
(see the computations in \cref{ssec:im_real_line}),
{we see that a necessary condition to have $\Im S(D)=\Im S(E)\neq\Im S(B) $ is that $D<0<E$.}

Moreover, looking again at the shape of $\Im S(U)$ and recalling that 
\[\Im S(A) \ne \Im S(D)=\Im S(U_c)=\Im S(E) \ne \Im S(C),\] we see that
there cannot be an $x$ in either $(-\infty,A)$ or $(C,+\infty)$
such that $\Im S(x) = \Im S(U_c)$.
\end{proof}

As said above, we are interested in describing the regions
$\{\Re S(u) > \Re S(U_c)\}$ or $\{\Re S(u) < \Re S(U_c)\}$.
From the asymptotics $S(U) \sim |\log(1-t_0)|\, U$, 
we know that for $\Re(U)$ negative (resp.\ positive) and large in absolute value
(in particular larger than $K |\Im(U)|$ for some constant $K>0$),
$U$ belongs to the region $\{\Re S(U) < \Re S(U_c)\}$ (resp.\  $\{\Re S(U) > \Re S(U_c)\}$).
Thus, the real level lines leaving $U_c$ split the complex plane as shown on 
the left-hand side of
\cref{fig:landscape_ReS}.

Note that we do not exclude the existence of smaller
islands inside the main regions as shown 
on the right-hand side of \cref{fig:landscape_ReS}.
For the sake of simplicity, we will not draw such potential islands in future figures.

\subsection{Moving contours and asymptotic of the kernel}
We recall from \eqref{eq:double_int_integr}
that the renormalized kernel $\widetilde{K}_{\la_N}^{(x_0,t_0)}$
writes as a double contour integral over contours $\gz$ and $\gw$ (where the integrand is denoted by $\Int_N(W,Z)$).

Our strategy for the asymptotic analysis of $\widetilde{K}_{\la_N}^{(x_0,t_0)}$ goes as follows:
\begin{enumerate}
  \item Move the integration contours from $\gz$ and $\gw$ to
    properly chosen contours $\gzn$ and $\gwn$ given in \cref{lem:move_contour_liquid};
  \item Replace the integrand $\Int_N(W,Z)$ by the asymptotically equivalent expression 
    given in \eqref{eq:equiv_integrand}:
\[
 (\sqrt{N})^{x_2-x_1}
 \E^{\sqrt{N}(S(W) - S(Z))}\,\hh(W,Z).\]
 This second step is performed in \cref{prop:limit_Mla}.
\end{enumerate}
The contours $\gzn$ and $\gwn$ will be constructed in such a way that,
for almost all $(W,Z) \in \gwn \times \gzn$, one has $S(W)<S(Z)$.
In this way, after the second step, the integrand -- and thus the integral -- will converge to $0$.
The asymptotic expansion of the kernel $\widetilde{K}_{\la_N}^{(x_0,t_0)}$ will then be given by the potential residue term
created by the change of contours.

\smallskip

Recalling the discussion below \eqref{eq:double_int_integr}, the integrand has poles for $W=Z$,
for some values of $W$ in the interval $I_W=[\eta\, a_0 -x_0,o(1)]$
and for some values of $Z$ in the interval $I_Z=[o(1),\eta\, a_m -x_0]$. We also recall from \eqref{eq:bound_L} that we chose the contours $\gw=\partial D(0,3L)$ and $\gz=\partial D(0,4L)$ (both followed in counterclockwise order), with $L=\eta \max(|a_0|,a_m)$. (Recall that we restricted our analysis to the case $t_1 \ge t_2$, i.e.\ when $\gamma_W$ is inside $\gamma_Z$.)

\begin{lemma}
 \label{lem:move_contour_liquid}
 There exist two integration contours
 $\gwn$ and $\gzn$ (both followed in counterclockwise direction) such that
  \begin{itemize}
    \item $\gwn$ and $\gzn$ intersect each other only at $U_c$ and $\overline{U_c}$;
    \item $\gwn$ (resp.\ $\gzn$) contains $I_W$ (resp.\ $I_Z$) in its interior;
    \item $\gwn \setminus \{U_c,\overline{U_c}\}$ (resp.\ $\gzn\setminus \{U_c,\overline{U_c}\}$)
      lies inside the region $\{\Re S(U) < \Re S(U_c)\}$ (resp.\ $\{ \Re S(U) >  \Re S(U_c)\}$).
  \end{itemize}
\end{lemma}
We refer the reader to \cref{fig:integration_contour} 
for an illustration of the original and new integration contours.
\begin{figure}
	\[\includegraphics[height=5.5cm]{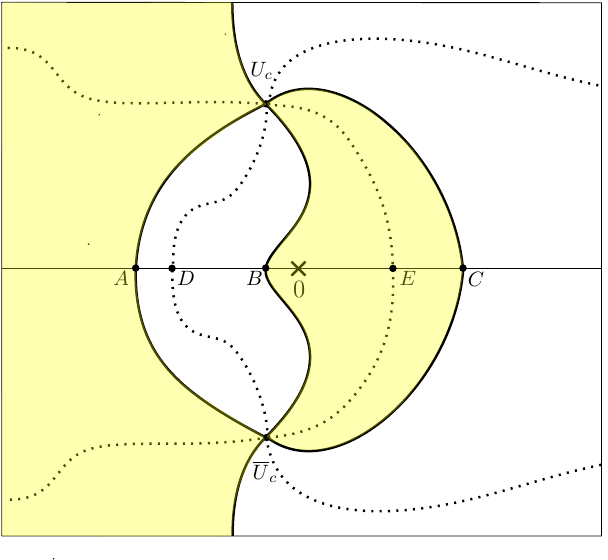}  
	\qquad \includegraphics[height=5.5cm]{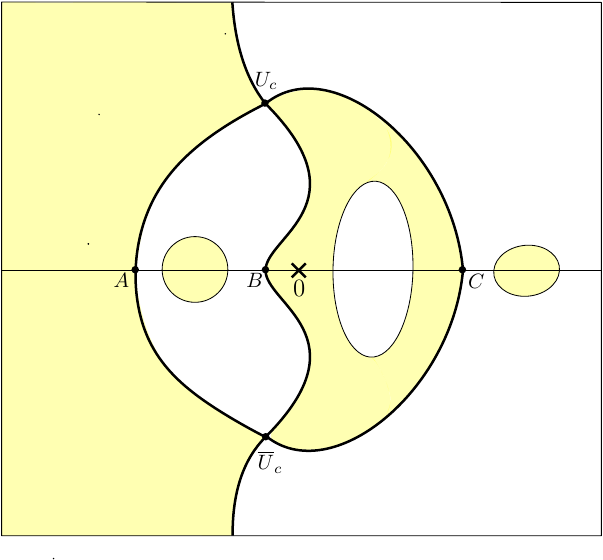}\] 
	\captionsetup{width=\linewidth}
	\caption{Two possible shapes for the landscape of $\Re S$. The black fat lines are the real level lines $\{\Re S(U) = \Re S(U_c)\}$. The yellow regions correspond to $\{\Re S(U) < \Re S(U_c)\}$, while the white regions correspond to $\{\Re S(U) > \Re S(U_c)\}$. On the left-hand side, we have also represented the imaginary level lines $\{\Im S(U) = \Im S(U_c)\}$ in dotted lines. The right-hand side shows another possible landscape of $\Re S$ with a more complicated configuration (to avoid overloading the picture, we did not draw imaginary level lines here). For the definition of the points $A,B,C,D,E$, see the discussion above \cref{lem:0_inside}.\label{fig:landscape_ReS}}

	\[\includegraphics[height=5.5cm]{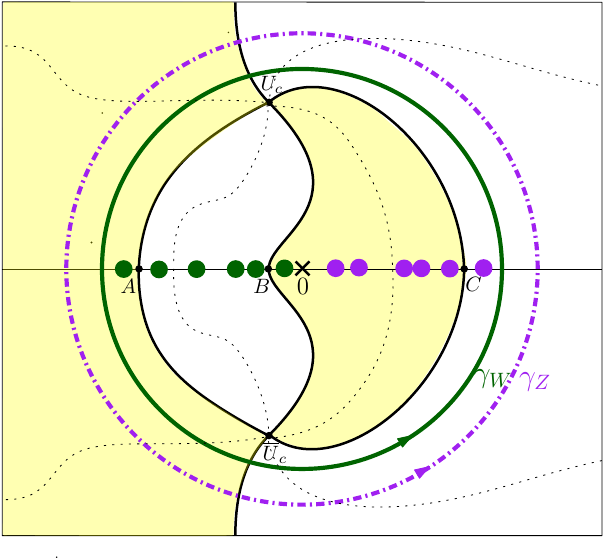} 
	\qquad 
	\includegraphics[height=5.5cm]{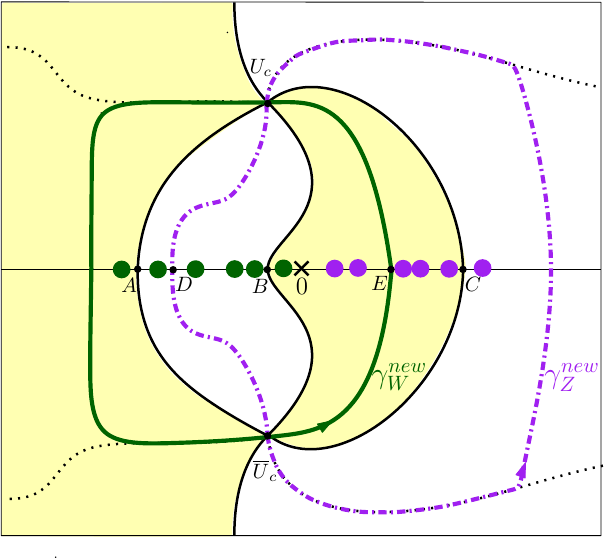} \]
	\captionsetup{width=\linewidth}
	\caption{
	 \textbf{Left:} The picture from \cref{fig:landscape_ReS} together with the original integration contours $\gw$ (in green)
	and $\gz$ (in purple) appearing in the kernel $\widetilde{K}_{\la_N}^{(x_0,t_0)}$. 
	The green and purple dots are respectively the $W$-poles and $Z$-poles of the integrand. We also recall that the yellow regions correspond to $\{\Re S(U) < \Re S(U_c)\}$, 
	while the white regions correspond to $\{\Re S(U) > \Re S(U_c)\}$.
	\textbf{Right:} The same picture with the new contours $\gzn$ and $\gwn$ from \cref{lem:move_contour_liquid}.\label{fig:integration_contour}}
\end{figure}
\begin{proof}
  We first explain how to construct $\gzn$, as the concatenation of two paths going from
  $U_c$ to $\overline{U_c}$ (reversing the right-most one so that we have a counterclockwise contour).

  Recall from the previous section (see also \cref{fig:landscape_ReS}) that there is a portion
  of $\{ \Re S(U) > \Re S(U_c) \}$ between the point $U_c$, $\overline{U_c}$,
  $A$ and $B$ and that this region contains an imaginary level line of $S$ that intersects the real axis at $D<0$ (\cref{lem:0_inside}).
  We choose this imaginary level line as the first path to construct $\gzn$.

  Around $U_c$, we can find another connected area of the region $\{ \Re S(U) > \Re S(U_c) \}$.
  This part also contains an imaginary level line of $S$. 
  This imaginary level line cannot go to the real axis, otherwise the meeting point would
	be an $x$ in $(C,+\infty)$ with $\Im S(x)=\Im S(U_c)$, contradicting \cref{lem:0_inside}. 
	 Because of the asymptotics $S(U) \sim |\log(1-t_0)|\,U$,
	this imaginary line should go to infinity inside the region  $\{y\leq |x|\}$
    and therefore meets the region $\{ x \ge M\}$, for any $M>0$.
    Moreover, we can choose $M>0$ such that $M \ge \max I_Z$ and such that if $|x| \ge M$ and $|y| \le x$ then $\Re S(x+iy) > \Re S(U_c)$.
    The second path used to construct $\gzn$ is then defined as follows:
    we first follow the imaginary level line from $U_c$ until it meets 
    the region $\{|x| \ge M,\, y\leq |x|\}$, and then join the real axis inside this region.
    We complete the path by symmetry until it hits $\overline{U_c}$.

    By construction, $\gzn$ goes through $U_c$ and $\overline{U_c}$
    and $\gzn\setminus \{U_c,\overline{U_c}\}$ lies in $\{ \Re S(U) >  \Re S(U_c)\}$.
    Also, the intersection points of $\gzn$ with the real axis are $D<0$ on the one side
    and some number larger than $ M > \max I_Z$ on the other side, so that $\gzn$ encloses $I_Z$ as wanted.
    We construct $\gwn$ in a symmetric way, completing the proof of the lemma.
\end{proof}

We are now ready to compute the asymptotics of the kernel $\widetilde{K}_{\la_N}^{(x_0,t_0)}$ 
in \eqref{eq:double_int_integr} when $(x_0,t_0)$ lies inside the liquid region
and hence compute the limit of the associated bead process $\widetilde{M}_{\la_N}^{(x_0,t_0)}$.

\begin{proposition}
  \label{prop:limit_Mla}
  Assume that $(x_0,t_0)$ is in the liquid region.
  The bead process $\widetilde{M}_{\la_N}^{(x_0,t_0)}$ converges in distribution
  to a determinantal process with correlation kernel
  \begin{equation}\label{eq:Kinfty}
    K^{(x_0,t_0)}_\infty\big((x_1,t_1),(x_2,t_2)\big):= 
     \frac{1}{2\I\pi}\int_{\gamma}  \frac{1}{1-t_0}\, \E^{\frac{W}{1-t_0}(t_1-t_2)} \, \left(\frac{W}{1-t_0}\right)^{x_2-x_1} \dd{W},
   \end{equation}
     where $\gamma$ is a path going from $\overline{U_c}$ to $U_c$ and passing on the left of $0$
     for $t_1 \ge t_2$ and on the right of $0$ for $t_1 <t_2$
     (the point 0 is a pole of the integrand in the case $x_1 < x_2$).
\end{proposition}

\begin{proof}
We first treat the case $t_1 \ge t_2$, so that the original contours in \eqref{eq:double_int_integr} are such that $\gamma_W$ is inside $\gamma_Z$.
We move the latter contours to the new contours $\gwn$ and $\gzn$ constructed in \cref{lem:move_contour_liquid}.
Since $I_W$ (resp.\ $I_Z$) is inside both $\gamma_W$ and $\gwn$ (resp.\ $\gamma_Z$ and $\gzn$), we do not cross any of the $W$-poles (resp.\ $Z$-poles).
However, since the relative positions of the contours is not the same
($\gwn$ is not inside $\gzn$), we have to take care of the pole at $Z=W$.

Fix $Z \in \gz$. {Possibly enlarging the contour $\gz$,} we can assume that the point $Z$ is outside both $\gw$ and $\gwn$.
Therefore, recalling that $\Int_N(W,Z)$ denotes the integrand in \eqref{eq:double_int_integr},
we have
\[   \oint_{\gw} \Int_N(W,Z)\, {\dd{W}} =  \oint_{\gwn} \Int_N(W,Z)\, {\dd{W}}  ,\]
which implies that
\[\widetilde{K}_{\la_N}^{(x_0,t_0)}((x_1,t_1),(x_2,t_2))
=  \frac{-1}{(2\I\pi)^2} \oint_{\gz} \oint_{\gw} \Int_N(W,Z) \dd{W} \dd{Z} 
= \frac{-1}{(2\I\pi)^2} \oint_{\gz} \oint_{\gwn} \Int_N(W,Z) \dd{W} \dd{Z}.\]

Fix now $W$ in $\gwn \setminus \{U_c,\overline{U_c}\}$.
Note that $\gwn \setminus \{U_c,\overline{U_c}\}=\gwnl \sqcup \gwnr$,
where $\gwnl$ and $\gwnr$ denote respectively the left and right parts of the contour $\gwn$; c.f.\ \cref{fig:integration_contour}.
We then move the $Z$-contour from $\gamma_Z$ (left-hand side of \cref{fig:integration_contour})
to $\gzn$ (right-hand side of \cref{fig:integration_contour}). If $W$ is in $\gwnr$,
then deforming the $Z$-contour from $\gz$ to $\gzn$ can be done without crossing any pole.
On the other hand, if $W$ is in $\gwnl$,
when deforming the $Z$-contour from $\gz$ to $\gzn$, we cross the pole $Z=W$.
By the residue theorem, for $W$ fixed in $\gwn \setminus \{U_c,\overline{U_c}\}$, we have
\[   \frac{1}{2\I\pi}   \oint_{\gz} \Int_N(W,Z) \dd{Z} 
	=  \frac{1}{2\I\pi} \oint_{\gzn} \Int_N(W,Z) \dd{Z} + \delta_{W \in \gwnl}\,
     \Res\big(\Int_N(W,Z),Z=W\big),\]
where $\Res\big(\Int_N(W,Z),Z=W\big)$ denotes the residue of $\Int_N(W,Z)$ corresponding to the pole $Z=W$.
Integrating over $W$ in $\gwn$ gives
\begin{multline}\label{eq:Kla_new_contours_liquid}
  \widetilde{K}_{\la_N}^{(x_0,t_0)}((x_1,t_1),(x_2,t_2))
= 
-\frac{1}{(2\I\pi)^2} \oint_{\gwn} \oint_{\gzn} \Int_N(W,Z) \dd{Z} \dd{W} \\
- \frac{1}{2\I\pi} \oint_{\gwnl} \!  \Res\big(\Int_N(W,Z),Z=W\big) \dd{W},
\end{multline}
where the set $\gwnl$ is interpreted as a path going from $U_c$ to $\overline{U_c}$;
note that it passes on the left of $0$.
Note that because of the pole in $Z=W$ of $\Int_N(W,Z)$,
the term $\oint_{\gzn} \Int_N(W,Z) \dd{Z}$ has a logarithmic singularity when $W$
tends to $U_c$ or $\overline{U_c}$; this singularity is integrable
so that the double integral above is well-defined.\medskip

Recalling that $\Int_N(W,Z)=\frac{F_{\la_N}(\sqrt{N}(Z+x_0))}{F_{\la_N}(\sqrt{N}(W+x_0))}\,\frac{\Gamma(\sqrt{N}W- x_1+1)}{\Gamma(\sqrt{N}Z - x_2+1)} \cdot \frac{(1-\widetilde{t}_2)^{\sqrt{N}Z - x_2}\,(1-\widetilde{t}_1)^{-\sqrt{N}W + x_1-1}}{(Z-W)}$, the right term in the latter displayed equation can be computed as follows:
\begin{align*}
  \Res\big(\Int_N(W,Z),Z=W\big) &= \frac{\Gamma(\sqrt N W-x_1+1)}{\Gamma(\sqrt N W -x_2+1)}  
\, (1-\widetilde{t}_2)^{\sqrt{N}W - x_2}\,(1-\widetilde{t}_1)^{-\sqrt{N}W + x_1-1} \\
&\simeq (\sqrt N W)^{x_2-x_1} (1-t_0)^{x_1-x_2-1} \exp\left( \frac{W}{1-t_0}\, (t_1-t_2) \right),
\end{align*}
where we used \cref{eq:asympt1} in the last estimation.
The estimate is uniform for $W$ in compact subsets of $\mathbb C \setminus \{0\}$,
implying
\begin{multline*}
	-\frac{1}{2\I\pi} \oint_{\gwnl} \!  \Res\big(\Int_N(W,Z),Z=W\big) \dd{W} \\
	\simeq\,-\frac{(\sqrt N)^{x_2-x_1}}{2\I\pi} \oint_{\gwnl} \!   \frac{1}{1-t_0}\, \E^{\frac{W}{1-t_0}(t_1-t_2)} \left(\frac{1-t_0}{W}\right)^{x_1-x_2}
	\dd{W}.
\end{multline*}
We now consider the double integral term $\oint_{\gwn} \oint_{\gzn} \Int_N(W,Z) \dd{Z} \dd{W}$ in \eqref{eq:Kla_new_contours_liquid}.
Call $Z_-$ and $Z_+$ (resp.~$W_-$ and $W_+$) the intersections points of the contour $\gzn$ (resp.~$\gwn$)
with the real axis, with the convention $Z_-<0<Z_+$ (resp.~$W_-<0<W_+$). 
By construction, $W_-$ and $Z_+$ can be chosen outside $[\eta a_0-x_0,\eta a_m-x_0]$ and hence belong
to the set $\DS$ (see its definition before \cref{lem:Asymp_Fla}).
Moreover, up to deforming slightly $\gzn$ and $\gwn$ (still keeping the properties of \cref{lem:move_contour_liquid} true),
we may assume that $W_+$ and $Z_-$ are distinct from all $\eta\, a_i-x_0$ and $\eta\, b_i-x_0$.\medskip

Fixing $\eps>0$, we have that:
\begin{itemize}
	\item Using the estimate in \eqref{eq:equiv_integrand} (which holds uniformly for $Z$ and $W$ in compact subsets of $\DS$), for $Z \in \gzn \setminus D(Z_-,\eps)$ 
	and $W \in \gwn \setminus D(W_+,\eps)$, the integrand $\Int_N(W,Z)$ is bounded by the function $2\hh(W,Z)$ in \eqref{eq:action} for all $N$ large enough,
	and tends to $0$ pointwise when $N$ tends to $+\infty$ because $\Re S(Z) > \Re S(W)$ by the third property in \cref {lem:move_contour_liquid} for $\gzn$ and $\gwn$.  (This is true except when $\{Z,W\} \subseteq \{U_c, \overline{U_c}\}$,
	where the bound by $\hh(W,Z)$ still holds but the integrand $\Int_N(W,Z)$ does not converge to zero;
    since this is a zero measure subset,
	this will not be problematic.)
	\item Using the second and third parts of \cref{lem:Asymp_Fla}, the bound by $\hh(W,Z)$ 
	and the convergence of $\Int_N(W,Z)$ to $0$ also hold when $Z \in D(Z_-,\eps)$ 
	or $W \in D(W_+,\eps)$ (or both), for $\eps$ small enough.
    Indeed, the extra factor $\exp(\eps \pi \sqrt N)$ in this case 
    is compensated by the factor $\exp(\sqrt N (S(W)-S(Z)))$ as soon as $S(W) <S(Z)-\pi \eps$,
   which happens for $W \in D(W_+,\eps)$ and $Z \in \gzn$ or for $W\in \gwn$ and $Z \in D(Z_-,\eps)$,
   if $\eps$ is small enough.
	\item The function $\hh(W,Z)$ is integrable for $(W,Z)$ on the double contour $(\gwn,\gzn)$. Indeed, it has singularities for $W=Z=U_c$ and $W=Z=\overline{U_c}$ but behaves as                 
	$O((W-Z)^{-1})$ near these singularities,
	and a standard computation shows that $(W-Z)^{-1}$ is integrable
	on $\gwn \times \gzn$, since the paths cross non-tangentially.
\end{itemize}
Hence, using the dominated convergence theorem, we know that
the double contour integral $\oint_{\gwn} \oint_{\gzn} \Int_N(W,Z) \dd{Z} \dd{W}$ 
goes to $0$ as $N$ tends to infinity.\medskip

Letting $\gamma$ to be the reverse path of $\gwnl$, we get
\begin{equation} 
  \label{eq:asymp_Ktilde}
	\lim_{N \to +\infty} 
  \frac{\widetilde{K}_{\la_N}^{(x_0,t_0)}((x_1,t_1),(x_2,t_2))}{(\sqrt{N})^{x_2-x_1}}
	= 
     \frac{1}{2\I\pi}\int_{\gamma} \!  \frac{1}{1-t_0}\, \E^{\frac{W}{1-t_0}(t_1-t_2)} \left(\frac{1-t_0}{W}\right)^{x_1-x_2} \dd{W}.
\end{equation}
Note that $\gamma$ is indeed a path from $\overline{U_c}$ to $U_c$ passing on the left of $0$, as required in the case $t_1 \ge t_2$.

\medskip

Let us now consider the case $t_1 < t_2$. We recall that in this case the initial contour
$\gamma_W$ and $\gamma_Z$ are swapped. An argument similar to the one above shows that 
\eqref{eq:asymp_Ktilde} still holds, but with the path $\gamma=\gwnr$ passing on the right of $0$.

Recall that, up to the factor $(\sqrt{N})^{x_2-x_1}$, the left-hand side of \eqref{eq:asymp_Ktilde}
is the kernel of the bead process $\widetilde{M}_{\la_N}$.
But the factor $(\sqrt{N})^{x_2-x_1}$ is a conjugation factor (see the last paragraph in \cref{sec:bead}):
adding or removing it does not change the associated bead process.

The expression on the right-hand side of \eqref{eq:asymp_Ktilde} depends continuously on $(x_1,t_1)$ and $(x_2,t_2)$
and is therefore locally bounded. Moreover, all estimates above and in particular
the convergence \eqref{eq:asymp_Ktilde} are locally uniform in these variables.
Therefore, from \cref{prop:dpp_conv}, the bead process $\widetilde{M}_{\la_N}^{(x_0,t_0)}$ converges in distribution to a determinantal point process with kernel $K^{(x_0,t_0)}_\infty$.
\end{proof}

\subsection{Recovering the bead kernel}
\label{sec:RecoveringBeadKernel}
We now want to compare the limit kernel $K^{(x_0,t_0)}_\infty$ from \cref{prop:limit_Mla}
with that of the random infinite bead process $J_{\a,\be}$ from \cref{thm:bead_kernel_cedric}. Recall from \eqref{eq:parm_liq} that, given the critical point $U_c$ associated to $(x_0,t_0)$ in the liquid region, $\a=\frac{\Im U_c}{1-t_0}$ and $\be=\frac{\Re U_c}{|U_c|}$.

\begin{lemma}\label{lem:recovering_bead_kernel}
There exists a function $g$ such that, for $t_1 \ne t_2$
  \begin{equation}\label{eq:Kinfty_Bead}
    K^{(x_0,t_0)}_\infty \big((x_1,t_1),(x_2,t_2)\big) = \frac{g(x_1,t_1)}{g(x_2,t_2)}\,
  J_{\a,\be} \big((x_1,t_1),(x_2,t_2)\big).
\end{equation}
\end{lemma}

Before proving the lemma, we discuss its consequences.
We recall that a conjugation factor of the form $\frac{g(x_1,t_1)}{g(x_2,t_2)}$
in a kernel does not affect the associated point process.
Also changing the kernel on a set of measure $0$, 
e.g.~$\left\{\big((x_1,t_1),(x_2,t_2)\big), t_1 = t_2\right\}$,
does not change the point process either.
Thus, the above lemma
implies that $K_\infty$ and $J_{\a,\be}$ are the kernels of the same point process,
i.e.\ the random infinite bead process of intensity $\a$ and skewness $\be$ introduced in \cref{defn:bead_process}.
Together with \cref{prop:limit_Mla}, this proves the first item of \cref{thm:cv_bead_process}.
\begin{proof}[Proof of \cref{lem:recovering_bead_kernel}]
  First assume $x_2 \ge x_1$ (but we do not assume any comparison between $t_1$ and $t_2$,
  the forthcoming argument works in both cases).
  Then the integrand of the kernel $K^{(x_0,t_0)}_\infty \big((x_1,t_1),(x_2,t_2)\big)$ in \eqref{eq:Kinfty} has no poles,
and we can replace $\gamma$ by any path from $\overline{U_c}$ to $U_c$.
We will take a vertical segment, which we parametrize as 
\[\gamma(u):=R(\cos \theta + \I u\sin\theta),\qquad u \in (-1,1),\]
where $R=|U_c|$ and $\theta=\Arg(U_c)$.
We get 
\begin{equation*}
	K_\infty\big((x_1,t_1),(x_2,t_2)\big)=\frac{1}{1-t_0}\,\frac{g(x_1,t_1)}{g(x_2,t_2)}
	\int_{-1}^{1} \E^{\I u( t_1-t_2)\frac{R\sin\theta}{1-t_0}}
    (\cos\theta + \I u\sin\theta)^{x_2-x_1} \frac{R \sin\theta \dd{u}}{2\pi},
\end{equation*}
with $g(y,\eps)=\exp(\frac{\eps R\cos\theta}{1-t_0}) \,\left(\frac{R}{1-t_0}\right)^{y}$.
Comparing with \eqref{eq:BeadKernel} where we take $\a=\frac{R\sin\theta}{1-t_0}=\frac{\Im U_c}{1-t_0}$
and $\be=\cos\theta=\frac{\Re U_c}{|U_c|}$,
we obtain \eqref{eq:Kinfty_Bead} in the case $x_2 \ge x_1$.
\medskip

Assume now $x_2 <x_1$ and $t_1 > t_2$. 
In this case, in \cref{prop:limit_Mla},
we integrate over the following path $\gamma$ going from $\overline{U_c}$ to $U_c$ as shown in \cref{fig:IntegrationRectangle}
(for $A\geq 1$ large enough this path passes on the left of $0$, as needed):
\begin{enumerate}
  \item take a vertical path $\gamma_1$ going down from $\overline{U_c}$ to $\Re U_c - \I A$;
  \item take a horizontal path $\gamma_2$ going left from  $\Re U_c - \I A$ to  $\Re U_c -(\log A)^2 - \I A$;
  \item take a vertical path $\gamma_3$ going up from  $\Re U_c -(\log A)^2 - \I A$
    to $\Re U_c -(\log A)^2 + \I A$;
  \item take a horizontal path $\gamma_4$ going right from  $\Re U_c -(\log A)^2 +\I A$ to  $\Re U_c + \I A$;
  \item and finally take a vertical path $\gamma_5$ going down from $\Re U_c + \I A$ to $U_c$.
\end{enumerate}

\begin{figure}[h]
	\includegraphics[height=6cm]{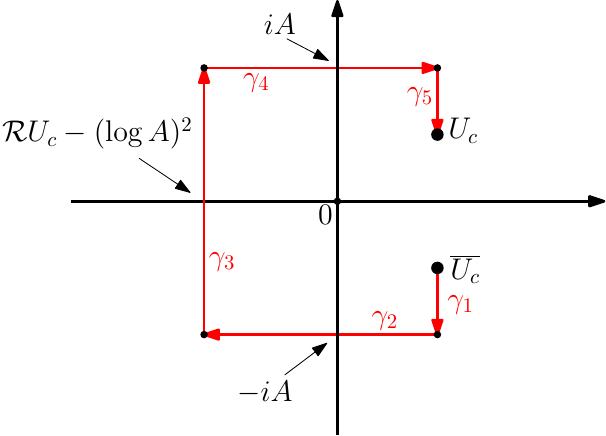}
	\captionsetup{width=\linewidth}
	\caption{The path $\gamma$ going from $\overline{U_c}$ to $U_c$ considered in the proof of \cref{lem:recovering_bead_kernel}.
	}
	\label{fig:IntegrationRectangle}
\end{figure}

Since $t_1 -t_2> 0$ and $x_1-x_2>0$,
the integrals over the second and fourth paths are bounded by
\[ \frac{(\log A)^2}{2\pi(1-t_0)}\,\E^{\frac{\Re U_c}{1-t_0}(t_1-t_2)} \left(\frac{1-t_0}{A}\right)^{x_1-x_2}. \]
This upper bound tends to $0$ as $A$ tends to infinity.
Similarly, when $t_1  > t_2$ and $x_1-x_2>0$, the integral over the third path is bounded by
\[ \frac{2A}{2\pi(1-t_0)} \,\E^{\frac{\Re U_c-(\log A)^2}{1-t_0}(t_1-t_2)} \left(\frac{1-t_0}{A}\right)^{x_1-x_2},\]
which also tends to $0$ as $A$ tends to infinity.
Consider the integral on the first path.
Reversing its direction (which yields a minus sign), it can be parametrized by 
\[\gamma_1(u)=R(\cos \theta + \I u\sin\theta),\quad u \in (-A/\Im U_c,-1).\]
Thus, doing the same computation as in the case $x_2 \ge x_1$ above,
the integral over $\gamma_1$ is equal to
\begin{equation*}
	-\frac{1}{1-t_0}\,\frac{g(x_1,t_1)}{g(x_2,t_2)}
	\int_{-A/\Im U_c}^{-1} \E^{\I u( t_1-t_2)\frac{R\sin\theta}{1-t_0}}
    (\cos\theta + \I u\sin\theta)^{x_2-x_1} \frac{R \sin\theta \dd{u}}{2\pi},
\end{equation*}
As $A$ tends to infinity, this quantity tends to 
\[-\frac{1}{1-t_0}\,\frac{g(x_1,t_1)}{g(x_2,t_2)}
	\int_{-\infty}^{-1} \E^{\I u( t_1-t_2)\frac{R\sin\theta}{1-t_0}}
    (\cos\theta + \I u\sin\theta)^{x_2-x_1} \frac{R \sin\theta \dd{u}}{2\pi}.\]
    Similarly the integral over $\gamma_5$ tends to
\[ -\frac{1}{1-t_0}\,\frac{g(x_1,t_1)}{g(x_2,t_2)}
    \int_1^\infty \E^{\I u( t_1-t_2)\frac{R\sin\theta}{1-t_0}}
    (\cos\theta + \I u\sin\theta)^{x_2-x_1} \frac{R \sin\theta \dd{u}}{2\pi}.\]
Comparing with \eqref{eq:BeadKernel}, we obtain \eqref{eq:Kinfty_Bead} in the case $x_2 <x_1$ and $t_1 > t_2$.
The case $x_2 < x_1$ and $t_1<t_2$ is similar, taking a path $\gamma$ going through $\Re U_c - \I A$,
$\Re U_c +(\log A)^2 \pm \I A$ and $\Re U_c +\I A$
(recall that for $t_1<t_2$, the path $\gamma$ in \cref{prop:limit_Mla} should pass on the right of $0$).
\end{proof}

\section{Asymptotics for the kernel in the frozen region}\label{sect:frozen_region}

\subsection{The small \texorpdfstring{$t$}{t} region}

In this section, we fix $x_0 \in [\eta\,a_0,\eta\, a_m]$ and we let $t_0 \le t_-$,
where $t_-=t_-(x_0)$ is given by \cref{prop:t_regions}.
We first assume that $x_0 \in (\eta\,a_0,0)$; the necessary modification for the cases $x_0=\eta\,a_0$ and $x_0\geq 0$ will be explained afterwards.

From \cref{prop:t_regions} and \cref{rem:pat_cases},
we know that in this regime the critical equation~\eqref{eq:critical_intro} has only real solutions and that
the smallest two, which will be denoted by $U_c \le U'_c$, are in $(-\infty,\eta\,a_0-x_0)$.
In \cref{lem:max_or_min,lem:move_contour_frozen1}, we assume $U_c < U'_c$ and discuss the case of a double critical $U_c=U'_c$ only in the proof
of the main result of the section, i.e.\ \cref{prop:convergence_kernel_frozen1}.

\begin{lemma}
  The critical points $U_c$ and $U'_c$ are respectively
  a local maximum and a local minimum of the function $u \to \Re S(u)$ on the real line
  and it holds that $\Re S(U_c) > \Re S(U'_c)$.
  \label{lem:max_or_min}
\end{lemma}
\begin{proof}
  Recall from \cref{ssec:im_real_line} that the map $u \mapsto \Re S(u)$;
  \begin{itemize}
  	\item is well-defined and continuous on the real line:
  	\item is differentiable as a function from $\mathbb R$ to  $\mathbb R$,
  except at the points $\eta\,a_i-x_0$ (where it has a positive infinite slope)
  and at the points $0$ and $\eta\,b_i-x_0$ (where it has a negative infinite slope);
  \item has a vanishing derivative exactly when $u$ satisfies the critical equation~\eqref{eq:critical_intro} or the companion equation
 \begin{equation}
    u  \prod_{i=1}^m (u+x_0-\eta\, b_i) = - (1-t_0)  \prod_{i=0}^m (u+x_0-\eta\, a_i).
   \label{eq:companion_critical2}
 \end{equation}
  \end{itemize} 
 This companion equation always has exactly $m+1$ solutions,
 in the intervals $[\eta\,a_{i-1}-x_0,\eta\, b_i-x_0]$ for $i \le i_0$
 and $[\eta\, b_i-x_0, \eta\, a_i-x_0]$ for $i \ge i_0$, where $i_0$ is such that $0 \in  (\eta \,a_{i_0-1} -x_0, \eta\, a_{i_0} -x_0]$.
 In particular, it does not have solutions for $u<\eta\, a_0-x_0$,
 and so $U_c$ and $U'_c$ are the only local extrema of $\Re S(u)$ in $(-\infty,\eta\,a_0-x_0)$.
 Observing that {$\lim_{u \to -\infty} \Re S(u)=-\infty$} and $\Re S(u)$ has a positive infinite slope at $\eta\,a_0-x_0$,
 this ends the proof of the lemma.
\end{proof}

We now recall from \cref{eq:double_int_integr,eq:equiv_integrand}
that the renormalized kernel $\widetilde{K}_{\la_N}^{(x_0,t_0)}$
writes as
\begin{equation}\label{eq:equiv_integrand3}
	\widetilde{K}_{\la_N}^{(x_0,t_0)}((x_1,t_1),(x_2,t_2)) 
    = - \frac{1}{(2\I \pi)^2}\oint_{\gz}\!\oint_{\gw} \Int_N(W,Z)\dd{W}\dd{Z},
\end{equation}
where $\Int_N(W,Z)$ can be approximated by
$(\sqrt{N})^{x_2-x_1} \E^{\sqrt{N}(S(W) - S(Z))}\,\hh(W,Z)$
uniformly on compact subsets of $\DS$ (defined above \cref{lem:Asymp_Fla}).
We recall also that the integrand $\Int_N(W,Z)$ has poles for $W=Z$,
for some values of $W$ in an interval $I_W=[\eta\, a_0 -x_0,o(1)]$
and for some values of $Z$ in an interval $I_Z=[o(1),\eta\, a_m -x_0-1]$.
\begin{lemma}
  \label{lem:move_contour_frozen1}
  There exist two integration contours
  $\gwn$ and $\gzn$ (both followed in counterclockwise order) such that, 
  \begin{itemize}
    \item $\gwn$ lies in the interior of $\gzn$;
    \item $\gwn$ (resp.\ $\gzn$) contains $I_W$ (resp.\ $I_Z$) in its interior;
    \item $\gwn$ (resp.\ $\gzn$) lies inside the region $\{\Re S(U) \le  \Re S(U'_c)\}$ (resp.\ $\{\Re S(U) \ge  \Re S(U_c)\}$).
  \end{itemize}
\end{lemma}
These new contours are shown in the left-hand side of \cref{fig:new_contours_frozen}.

\begin{figure}
\[\includegraphics[height=49mm]{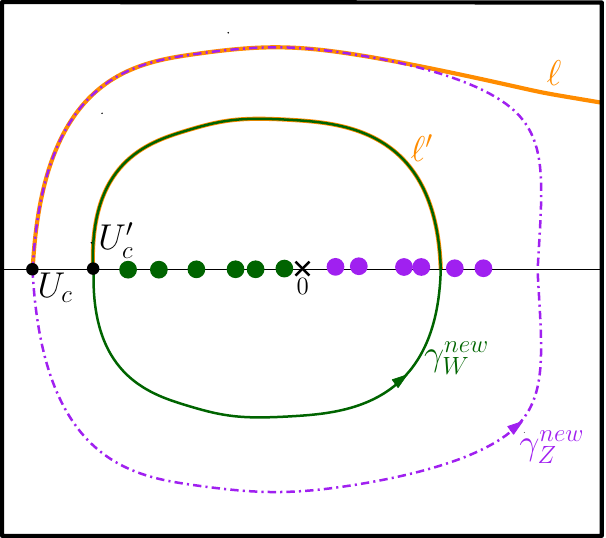}
\quad \includegraphics[height=49mm]{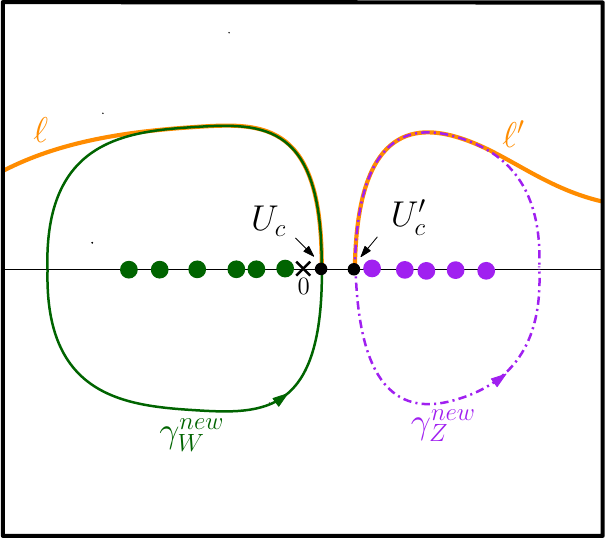}
\quad \includegraphics[height=49mm]{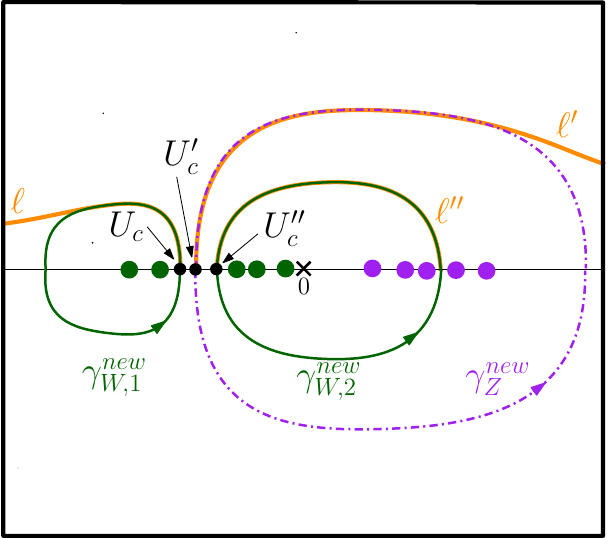}\]
\captionsetup{width=\linewidth}
\caption{\textbf{From left to right:} In green and purple, the new integration contours $\gwn$ and $\gzn$ constructed in \cref{lem:move_contour_frozen1,lem:move_contour_frozen2,lem:move_contour_frozen3}. The three situations correspond to the small $t$ region $\{0\leq t_0 \le t_-(x_0)\}$, the large $t$ region $\{t_+(x_0) \leq t_0 \leq 1\}$, and the intermediate $t$ region $\{t_-(x_0)<t_0<t_+(x_0)\}$, respectively. In orange, we highlight the level lines considered in the proofs of the three lemmas.}
\label{fig:new_contours_frozen}
\end{figure}

\begin{proof}
We start by noting that the action $S$ can be analytically continued in the upper half-plane on a
neighborhood of $U_c\in \R$ and behaves like 
\begin{equation*}
	S(U)=S(U_c)+ \frac{S''(U_c)}{2}\,(U-U_c)^2 +O((U-U_c)^3).
\end{equation*}
 Since $\Im S(u)$ is piecewise affine and $\Re S(u)$, seen as a function of a real variable $u$,
 reaches a local maximum at $U_c$ (\cref{lem:max_or_min}), 
 the coefficient $S''(U_c)$ must be real and negative.
 Since $S''(U_c)$ is real, by comparison with the map $F(U)=U^2$, there is an imaginary level line of $S$ leaving from $U_c$ orthogonally
 to the real axis and going in the upper half plane; call it $\ell$. 
 Moreover, since $S''(U_c)$ is negative, the real part of $S$ increases along $\ell$ (recall that along an imaginary level line, the function $\Re S(U)$ is strictly monotone by the open mapping theorem).
 In particular, $\ell$ is included in the region $\{\Re S(U) \ge  \Re S(U_c)\}$.
 Since $S$ does not have any critical point in the upper half-plane,
 $\ell$ cannot end inside the upper half-plane, hence it either reaches again the real axis
 or it goes to infinity.
 If it goes to infinity, we claim that $\ell$ can only go to infinity in the positive real direction. Indeed, the estimate $S(U) \sim |\log(1-t_0)|\, U$  for large $|U|$ (recall \eqref{eq:as_action}) forces $\Im U$ to stay bounded along $\ell$ and $\Re U$ to stay bounded from below along $\ell$
 (by definition $\Im S(U) $ is bounded along $\ell$
 and $\Re S(U) \ge  \Re S(U_c)$ along $\ell$).

 Similarly one can consider the imaginary line $\ell'$ of $S$ leaving from $U_c'$ 
 and going in the upper half plane. The real part of $S$ is decreasing along $\ell'$
 so that $\ell'$ stays in the region $\{ \Re S(U) \le  \Re S(U'_c)\}$.
 In particular $\ell$ and $\ell'$ cannot cross.
 Similarly as above, one can see that $\ell'$ either reaches again the real axis at some point
 or go to infinity in the negative real direction.
  \smallskip

  To determine the behavior of $\ell$ and $\ell'$, it is useful to look
  at the imaginary part of the action of the real line.
  We recall from \eqref{eq:im_action} and the discussion below it, that $\Im S(U_c) = \Im S(U'_c) = -\pi x_0$,
  and that $-\pi x_0 \in (0, \Im S(0))$.
  Since $\Im(S(u))=0$ for large real positive values of $u$, there exists $\bx>0$
  such that $\Im(S(\bx))=-\pi x_0$.
  Considering the shape of the real map $u \mapsto \Im (S(u))$ (\cref{fig:Imaginary_action}), 
  there are two possibilities:
  \begin{itemize}
    \item 
  If $\bx$ is an interval of the type $(\eta\,b_{k}-x_0,\eta\,a_{k}-x_0)$,
  then $u \mapsto \Im(S(u))$ is decreasing at $\bx$, and $\bx$ is uniquely determined
  by the conditions $\bx>0$ and  $\Im(S(\bx))=-\pi x_0$.
    \item 
      If $\bx$ is an interval of the type $(\eta\,a_{k-1}-x_0,\eta\,b_{k}-x_0)$,
      then $u \mapsto \Im(S(u))$ is constant on that interval, and any point of this interval
      can be chosen as $\bx$. 
      We then choose $\bx$ to be the unique solution of the critical equation in that interval
      (the existence being ensured by \cref{lem:critical_points_general} and
      the uniqueness by our assumption that the critical equation has two solutions in $(-\infty,\eta \,a_0-x_0)$).
  \end{itemize}
\smallskip

 \noindent\emph{\underline{Claim:} The imaginary level lines $\ell$ and $\ell'$ can only come back to the real axis at $\bx$.}
 \smallskip
 
 First note that $\ell$ and $\ell'$ can only come back to the real axis at a point $u\in\R\setminus\{U_c,U'_c\}$  such that  $\Im S(u)=\Im S(U_c) = \Im S(U'_c).$
 Moreover, let $y$ be a non-critical point around which $u \mapsto \Im S(u)$ is constant for real $u$.
 Since $y$ is non-critical, the set $\{\Im S(z)=\Im S(y)\}$ is a single curve around $y$ and thus coincide
 with the real axis. In particular $\ell$ (or $\ell'$) cannot reach the real axis at such $y$.

 The case when $y=\eta a_i-x_0$ or $y=\eta b_i-x_0$ is a bit more delicate. Locally around $y$, the function $S$
 looks like $\pm (z-y) \log (z-y)$, and thus the set $\{\Im S(z)=\Im S(y)\}$ looks locally like a half-line.
 Thus it is locally included in the real line (we know from \cref{fig:Imaginary_action} that at each
 point $y=\eta a_i-x_0$ or $y=\eta b_i-x_0$, there is a direction along the real line in which $\Im S(u)$ is constant).
 Again, $\ell$ (or $\ell'$) cannot reach the real axis at such $y$.  
 
 Finally, we observe that the interval $(-\infty,\eta a_0-x_0)$ contains no other critical point than $U_c$ and $U'_c$,
 so, from the previous discussion, $\ell$ and $\ell'$ cannot come back to the real axis inside this interval.
 Altogether, this proves the claim, i.e.~that $\ell$ and $\ell'$ can only come back to the real axis at $\bx$.
 \smallskip

 A first consequence is that at most one of $\ell$ and $\ell'$ can come back
 to the real axis. But they cannot go both to infinity, otherwise they would cross.
 So one of $\ell$ or $\ell'$ should go to the real axis, and the other to infinity.
 The non-crossing constraint and the fact that $\bx>0$ forces that
 $\ell$  goes to infinity, while $\ell'$ goes to $\bx$.
 \smallskip
 
 The contours of the lemma are now obtained as follows. For $\gwn$,
 we simply follow $\ell'$ and its mirror image in the lower half plane.
 By construction, it lies in the region  $\{ \Re S(U) \le  \Re S(U'_c)\}$.
Also, since the return point $\bx$ to the real axis satisfies $\bx>0$,
$\gwn$ encloses $I_W$.
 For $\gzn$, we follow $\ell$ until $\Re(U)$ is sufficiently large, and then go to the real axis
 following any smooth curve. We then go back to $U_c$ with the mirror image
 of the first part of the contour.
 If we follow $\ell$ until when $\Re(U)$ is sufficiently large,
 we can ensure that this contour encloses $I_Z$ and 
 lies in the region $\{ \Re S(U) \ge  \Re S(U_c)\}$.
 Finally note that $\gwn$ lies in the interior of $\gzn$ by construction. 
\end{proof}

\begin{proposition}
\label{prop:convergence_kernel_frozen1}
  Let $x_0<0$ and $0\leq t_0 \le t_-(x_0)$. Then, locally uniformly for $(x_1,t_1) \in \Z \times \R$, we have that
  \begin{equation}
    \label{eq:K_tend_to_0}
    \lim_{N \to +\infty} \widetilde{K}_{\la_N}^{(x_0,t_0)}((x_1,t_1),(x_1,t_1)) =0.
  \end{equation}
  As a consequence, $\widetilde{M}_{\la_N}^{(x_0,t_0)}$ tends in probability to the empty set.
\end{proposition}
\begin{proof}
  We are interested in \eqref{eq:equiv_integrand3} for $x_1=x_2$ and $t_1=t_2$.
  In this case, the contours are such that 
  $\gw$ contains $I_W$, $\gz$ contains $I_Z$ and $\gw$ lies in the interior of $\gz$. 
  Therefore,
  we can transform the contours $\gw$ and $\gz$ to the contours $\gwn$ and $\gzn$
  from \cref{lem:move_contour_frozen1} without crossing any poles.
Thus, by the residue theorem, we get
\begin{equation}\label{eq:Kla_new_contours_frozen1}
    \widetilde{K}_{\la_N}^{(x_0,t_0)}((x_1,t_1),(x_1,t_1)) 
	= - \frac{1}{(2\I \pi)^2} \oint_{\gzn}\!\oint_{\gwn} \Int_N(W,Z)\dd{W}\dd{Z}.
\end{equation}
We recall that $\Int_N(W,Z) \simeq (\sqrt{N})^{x_2-x_1} \E^{\sqrt{N}(S(W) - S(Z))}\,\hh(W,Z)$ on $\DS$
and that $\Re S(W) \le  \Re S(U_c') < \Re S(U_c) \le \Re S(Z)$ for $(W,Z) \in \gwn \times \gzn$
(see \cref{lem:max_or_min,lem:move_contour_frozen1}).
A difficulty comes from the fact that the {right-most} intersection point $W^+$ of $\gwn$ with the real axis {might lie} outside $\DS$ {($W^+$ coincides with the point $\bx$ in the proof of \cref{lem:move_contour_frozen1}). Nevertheless,} the integrand can be controlled around this point thanks to the third part of \cref{lem:Asymp_Fla}.
The same arguments as in the proof of \cref{prop:limit_Mla} shows that $\Int_N(W,Z)$ 
converges pointwise to $0$ and that the integrand is bounded by the integrable function $\hh(W,Z)$.
Dominated convergence applies, proving~\eqref{eq:K_tend_to_0}. 

The convergence in distribution to the empty set then follows from the general identity for determinantal point processes
$$\esper\!\left[\widetilde{M}_{\la_N}^{(x_0,t_0)}(A)\right] = \int_A \widetilde{K}_{\la_N}^{(x_0,t_0)}((x_1,t_1),(x_1,t_1)) \dd{\la}(x_1,t_1),$$
for any bounded subset $A \subset \Z \times \R$, where we recall that $\widetilde{M}_{\la_N}^{(x_0,t_0)}(A)$ denotes the number of beads of $\widetilde{M}_{\la_N}^{(x_0,t_0)}$ contained in the set $A$. By dominated convergence, these expected numbers of beads go to $0$, so  $(\widetilde{M}_{\la_N}^{(x_0,t_0)}(B_i))_{1\leq i \leq k}$ converges in probability to $(0,\ldots,0)$ for any collection of bounded subsets $(B_i)_{1\leq i \leq k}$. The discussion from \cref{sec:bead} shows that this is equivalent to the convergence in distribution towards the empty set.

We now discuss the case where the critical equation has a double root $U_c=U'_c$ 
in the interval $(-\infty,\eta\, a_0-x_0)$.
Then $u \mapsto \Re S(u)$ is increasing on $(-\infty,\eta\, a_0-x_0)$
(this is an immediate analogue of \cref{lem:max_or_min}),
and the action writes locally as
$S(z)=S(U_c)+\frac{S'''(U_c)}{6} (z-U_c)^3 +O\left( (z-U_c)^4 \right)$,
with $S'''(U_c)>0$. \cref{fig:double_critical_point} shows the imaginary level lines
of $S$, and the $\{\Re S(U) < \Re S(U_c)\}$ and $\{\Re S(U) > \Re S(U_c)\}$ regions around $U_c$.
In particular, we can find new integration contours
as in \cref{lem:move_contour_frozen1} 
except that they now meet at $U_c=U'_c$.
Note that, in this case, the integrand in \eqref{eq:Kla_new_contours_frozen1} has a singularity
at $Z=W=U_c$, where it behaves as $O( (Z-W)^{-1})$. 
This is similar to the setting of \cref{prop:limit_Mla}
and we can therefore apply dominated convergence using the same arguments.
We conclude that 
the renormalized correlation kernel tends also to $0$, finishing the proof of the proposition.
\end{proof}

\begin{figure}[t]
\[\begin{array}{c}\includegraphics[height=5cm]{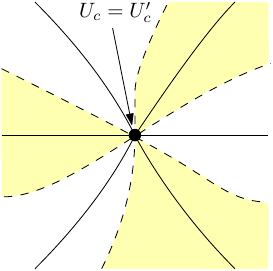}\end{array}
\qquad \begin{array}{c}\includegraphics[height=5cm]{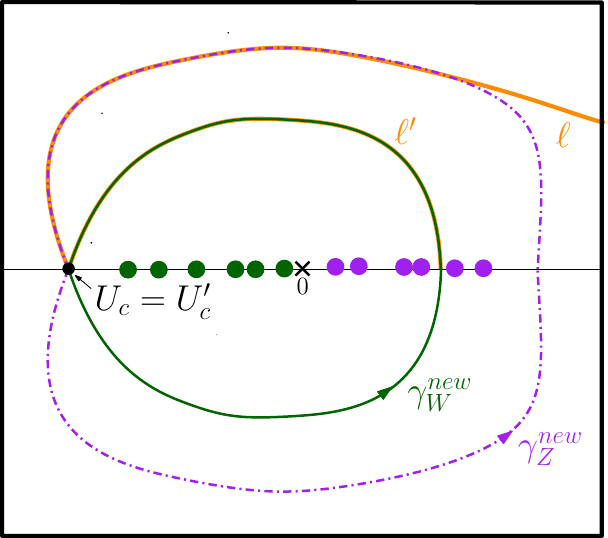}\end{array}\]
	\captionsetup{width=\linewidth}
\caption{\textbf{Left:} The landscape of the action $S$ around a double critical point
on the real line in the case of the small $t$ region $\{0\leq t_0 \le t_-(x_0)\}$. Plain lines are imaginary level lines, dotted lines are the real level lines.
We also indicated the alternation of different regions: the yellow regions correspond to $\{\Re S(U) < \Re S(U_c)\}$, 
while the white regions correspond to $\{\Re S(U) > \Re S(U_c)\}$.
\textbf{Right:} The new integration contour in the case of a double critical point
(to be compared with the left-hand side of \cref{fig:new_contours_frozen}).}
\label{fig:double_critical_point}
\end{figure}

The case $x_0 \in (0, \eta\, a_{m})$ is essentially treated in the same way, with the following modifications.
By \cref{prop:t_regions}, the two relevant real critical points $U_c$ and $U_c'$ live in $(\eta\,a_m-x_0,+\infty)$.
To simplify the discussion we assume $U_c \ne U'_c$.
Using the convention $U_c<U_c'$, \cref{lem:max_or_min} still holds.
An analog of \cref{lem:move_contour_frozen1} also holds, with the important change
that in this case, $\gzn$ lies in the interior of $\gwn$ (and not the opposite).
Since the contours in \eqref{eq:equiv_integrand} satisfy that $\gw$ is inside $\gz$ (for $t_1=t_2$),
moving them to  $\gzn$ and $\gwn$ yield a residue term related to the pole $Z=W$
(as in \eqref{eq:Kla_new_contours_liquid}, except that the residue term appears for any $W$ in $\gwn$).
For  $x_1=x_2$, $t_1=t_2$, recalling \eqref{eq:double_int_integr}, the residue of the integrand in \eqref{eq:Kla_new_contours_frozen1}
related to the pole $Z=W$ is simply $\frac{1}{1-t_0-\frac{t_1}{\sqrt N}}$ .
Therefore, \cref{eq:Kla_new_contours_frozen1} is replaced by:
\begin{align*}
    \widetilde{K}_{\la_N}^{(x_0,t_0)}((x_1,t_1),(x_1,t_1))
    = - \frac{1}{(2\I \pi)^2} \oint_{\gzn}\!\oint_{\gwn}\Int_N(W,Z) \dd{W}\dd{Z} 
- \frac{1}{2\I \pi}\oint_{\gwn} \frac{1}{1-t_0-\frac{t_1}{\sqrt N}} \dd{W}.
\end{align*}
But the second integral is obviously $0$, so that \eqref{eq:K_tend_to_0} holds also in this case.
We conclude as above and obtain that also in this regime, the bead process $\widetilde{M}_{\la_N}^{(x_0,t_0)}$ tends in distribution to the empty set.

The same conclusion also holds for the following remaining three cases: when $x_0= 0$ 
(in this case, by \cref{prop:t_regions}, the only admissible value for $t_0$ is $t_0=0$), when $x_0=\eta\,a_0$, and when $x_0=\eta\,a_m$. We leave to the reader the details of these three remaining cases, pointing out that one needs some little modifications in the same spirit as \cref{rem:pat_cases}.

Combining all the results in this section, we obtain the following final result for the small $t$ region.

\begin{proposition}\label{prop:convergence_kernel_frozen1bis}
{Let $x_0\in [\eta\,a_0,\eta\, a_m]$ and $0\leq t_0 \leq t_-(x_0)$.} Then the bead process $\widetilde{M}_{\la_N}^{(x_0,t_0)}$ tends in probability to the empty set.
\end{proposition}

\subsection{The large \texorpdfstring{$t$}{t} region}
In this section, we fix $x_0 \in [\eta\,a_0,\eta\,a_m]$ and we let $t_+ \leq t_0 \leq 1$,
where $t_+=t_+(x_0)$ is given by \cref{prop:t_regions}.
{For the sake of brevity, we restrict ourselves to the case when there exists $i_0\geq 1$ such that $x_0 \in (\eta\,a_{i_0-1},\eta\,b_{i_0})$.}
The case $x_0 \in (\eta\,b_{i_0},\eta\,a_{i_0})$ for some $i_0\geq 1$ and the case $x_0=\eta\,a_{i_0}$ for some $i_0\geq 0$
are similar (for the case $x_0=\eta\,a_{i_0}$, we refer to the discussion in \cref{rem:pat_cases}).
{Finally, for $x_0=\eta\,b_{i_0}$, the only admissible value for $t_0$ in the large $t$ region is $t_0=1$ (since \cref{prop:t_regions} states that $t_+(\eta\,b_{i_0})=1$); also in this case we omit the simple necessary modifications.}

From \cref{prop:t_regions},
we know that the critical equation~\eqref{eq:critical_intro} has only real solutions,
two of which, say $U_c \le U'_c$, are in the interval $(0,\eta\, b_{i_0}-x_0)$.
We will discuss here the case $U_c < U'_c$, the case of a double critical
point $U_c=U'_c$ being obtained with simple modifications similar to those discussed in the previous subsection.

\begin{lemma}
  The critical point $U_c$ and $U'_c$ are respectively
  a local minimum and a local maximum of the function $u \to \Re S(u)$ on the real line
  and we have $\Re S(U_c) < \Re S(U'_c)$.
  \label{lem:max_or_min2}
\end{lemma}
\begin{proof}
The lemma follows from the fact that $u \to \Re S(u)$ 
has negative infinite slopes at the points $0$ and $\eta\,b_{i_0}-x_0$ and
that there are no local extrema other than $U_c$ and $U'_c$ in the interval 
$(0,\eta\, b_{i_0}-x_0)$.
\end{proof}

We now see how to move integration contours. To this end, let us first remark that when $x_0 \in (\eta\,a_{i_0-1},\eta\,b_{i_0})$ the $Z$ poles of $\Int_N(W,Z)$
all lie in a strictly smaller sub-interval of the interval $I_Z=[o(1),\eta\, a_m -x_0-1]$ considered so far. Indeed, from \eqref{eq:double_int_integr}, one can see that the only $Z$-poles of $\Int_N(W,Z)$ are a subset of the poles of $F_{\la_N}(\sqrt N (Z +x_0))$. But, from \eqref{eq:f_lambda} and \cref{fig:poles}, we know that the latter function has no poles in the intervals of the form $(\eta\,a_{i-1}-x_0, \eta\,b_{i}-x_0)$ and so, in particular, the $Z$ poles of $\Int_N(W,Z)$ all lie in $\widetilde I_Z=[\eta b_{i_0} -x_0,\eta\, a_m -x_0-1] \subset I_Z$.

\begin{lemma} 
  \label{lem:move_contour_frozen2}
 There exist two integration contours $\gwn$ and $\gzn$ (both followed in counterclockwise order) such that
  \begin{itemize}
    \item $\gwn$ and $\gzn$ have disjoint interiors;
    \item $\gwn$ (resp.\ $\gzn$) contains $I_W$ (resp.\ $\widetilde I_Z$) in its interior;
    \item $\gwn$ (resp.\ $\gzn$) lies inside the region $\{ \Re S(U) \le  \Re S(U_c)\}$ (resp.\ $\{ \Re S(U) \ge  \Re S(U'_c)\}$).
  \end{itemize}
\end{lemma}
These contours are shown in the middle picture of~\cref{fig:new_contours_frozen}.
\begin{proof}
The proof is similar to that of \cref{lem:move_contour_frozen1}.
We consider the imaginary level lines $\ell$ and $\ell'$ leaving from $U_c$ and $U'_c$ orthogonally to the real line and in the upper-half plane.
The level line $\ell$ (resp.\ $\ell'$) lies inside the region
$\{ \Re S(U) \le  \Re S(U_c)\}$ (resp.\ $\{ \Re S(U) \ge  \Re S(U'_c)\}$).
Moreover, they cannot go back to the the real line. 
Indeed, the value $\Im S(z)$ on the real line is maximal on the interval $(0,\eta\, b_{i_0}-x_0)$ (see \cref{ssec:im_real_line} and \cref{fig:Imaginary_action}) 
so that if $\ell$ or $\ell'$ goes back to the real line, it has to be within this interval.
But this would mean that $S$ has a third critical point in this interval, which is impossible by \cref{lem:critical_points_general} and \cref{prop:t_regions}.
We conclude that both $\ell$ and $\ell'$ go to infinity. Since $S(U) \sim |\log(1-t_0)| \, U$ for large $|U|$ by \eqref{eq:as_action}, necessarily $\ell$ goes to infinity in the negative
real direction, while $\ell'$ goes to infinity in the positive real direction.

It is then possible to follow $\ell$ long enough, and then join the negative real axis while staying in the region $\{ \Re S(U) \le  \Re S(U_c)\}$.
Completing the path by symmetry gives the contour $\gwn$.
The contour $\gzn$ is constructed similarly from $\ell'$.

We finally note that the claim in the second item follows from the discussion above the lemma statement.
\end{proof}

\begin{proposition}
\label{prop:convergence_kernel_frozen2}
Let  $x_0 \in [\eta\,a_0,\eta\,a_m]$
and $t_+(x_0) \leq t_0 \leq 1$.  Then, locally uniformly for $(x_1,t_1) \in \Z \times \R$, we have that
 \[\lim_{N \to +\infty} \widetilde{K}_{\la_N}^{(x_0,t_0)}((x_1,t_1),(x_1,t_1)) =0.\]
  As a consequence, $\widetilde{M}_{\la_N}^{(x_0,t_0)}$ 
  tends in probability to the empty set.
\end{proposition}
\begin{proof}
{As already explained, we give details only in the case when there exists $i_0\geq 1$ such that $x_0 \in (\eta\,a_{i_0-1},\eta\,b_{i_0})$.}
 Again, we are interested in \eqref{eq:equiv_integrand} for $x_1=x_2$ and $t_1=t_2$.
 The original contours are such that $\gw$ lies in the interior of $\gz$. 
 As in the case $x_0>0$ and $0\leq t_0\leq t_-(x_0)$ in the previous section, moving the contour to those of \cref{lem:move_contour_frozen2}
 yields the following:
\[ \widetilde{K}_{\la_N}^{(x_0,t_0)}((x_1,t_1),(x_1,t_1)) 
 = - \frac{1}{(2\I \pi)^2} \oint_{\gzn}\!\oint_{\gwn} \Int_N(W,Z)\dd{W}\dd{Z} 
 -\frac{1}{2\I \pi} \oint_{\gwn} \frac{1}{1-t_0-\frac{t_1}{\sqrt N}} \dd{W}.\]
The second integral is identically $0$,
while the first tends to $0$ as $N$ tends to $+\infty$ using the same argument as before.
This ends the proof of the proposition.
\end{proof}

\subsection{The intermediate \texorpdfstring{$t$}{t} region}
\label{ssec:frozen-intermediate}
Fix $x_0 \in [\eta\,a_0,\eta\,a_m]$.
As seen in \cref{ssec:examples}, it might happen that there is some $t_0 \in (t_-(x_0),t_+(x_0))$
in the frozen region, i.e.\ such that the critical equation
 \eqref{eq:critical_intro} has only real roots.
{In this case, thanks to the third item of \cref{prop:t_regions}, there are three roots --  denoted by $U_c \le U_c'\le U''_c$ -- that are either all inside a negative interval $(\eta\,b_{j_0} -x_0,\eta\,a_{j_0} -x_0)$ for some $j_0<i_0$,
or all inside a positive interval $(\eta\,a_{j_0}-x_0,\eta\,b_{{j_0}+1} -x_0)$ for some $j_0\geq i_0$.
These two cases are treated similarly, 
we will therefore assume that the three roots are in $(\eta\,b_{j_0} -x_0,\eta\,a_{j_0} -x_0)$ for some $j_0<i_0$.}
Also, we assume $U_c<U_c'< U''_c$, and let the reader convince himself
that multiple critical points do not create additional difficulties.
The following lemma is proven similarly as before.
\begin{lemma}
  The critical point $U_c$, $U'_c$ $U''_c$ are respectively
  local minimum, maximum and minimum of the function $u \to \Re S(u)$ on the real line
  and we have both $\Re S(U_c) < \Re S(U'_c) > \Re S(U''_c)$.
  \label{lem:max_or_min3}
\end{lemma}

This helps us to construct integration contours as follows.
With an argument similar to the one used above \cref{lem:move_contour_frozen2}, we note that there are no $W$-poles in the interval $(\eta\,b_{j_0} -x_0,\eta\,a_{j_0} -x_0)$.
A major difference is that now the $W$ integration contour is split into two disjoint parts,
the union of their interiors containing all $W$ poles.

 \begin{lemma} 
  \label{lem:move_contour_frozen3}
 There exist integration contours $\gwno$, $\gwnt$ and $\gzn$  (all followed in counterclockwise order) such that
  \begin{itemize}
    \item {$\gwno$ and $\gzn$ have disjoint interiors, and $\gwnt$ lies inside $\gzn$};
    \item $\gwno$ (resp.\ $\gwnt$  and $\gzn$) contains $(\eta\,a_0 -x_0,\eta\,b_{j_0} -x_0)$
    (resp.\ $(\eta\,a_{j_0} -x_0,0)$ and $I_Z$) in its interior;
    \item $\gwno$ (resp.\ $\gwnt$) lies inside the region $\{ \Re S(U) \le  \Re S(U_c)\}$
    (resp.\ $\{ \Re S(U) \le  \Re S(U''_c)\}$);
    \item $\gzn$ lies inside the region $\{ \Re S(U) \ge  \Re S(U'_c)\}$.
  \end{itemize}
\end{lemma}

These contours are shown in the right-hand side of \cref{fig:new_contours_frozen}.

\begin{proof}
The strategy of proof is again the same. We consider the three imaginary lines $\ell$,
$\ell'$ and $\ell''$ leaving from $U_c$, $U'_c$ and $U''_c$. One can prove that $\ell''$ goes back to the real axis at some point $x>0$, while $\ell$ and $\ell'$ go to infinity respectively in the negative and positive directions. 
Symmetrizing $\ell''$ gives the contour $\gwnt$. On the other hand, following $\ell$ and $\ell'$ for long enough and then joining the real line (plus symmetrizing) gives $\gwno$ and $\gzn$.
\end{proof}

\begin{proposition}
\label{prop:convergence_kernel_frozen3}
Let  $x_0 \in [\eta\,a_0,\eta\,a_m]$ and $t_0\in(t_-(x_0),t_+(x_0))$ 
such that the critical equation~\eqref{eq:critical_intro} has only real roots.
Then, locally uniformly for $(x_1,t_1) \in \Z \times \R$, we have that
 \[ \lim_{N \to +\infty} \widetilde{K}_{\la_N}^{(x_0,t_0)}((x_1,t_1),(x_1,t_1)) =0.\]
  As a consequence, $\widetilde{M}_{\la_N}^{(x_0,t_0)}$ 
  tends in probability to the empty set.
\end{proposition}
\begin{proof}
 Again, we consider \eqref{eq:equiv_integrand} for $x_1=x_2$ and $t_1=t_2$.
 The original contours are such that $\gw$ lies in the interior of $\gz$. 
 Moving the contour to those of \cref{lem:move_contour_frozen3}
 (note that $\gwno$ and $\gwnt$ together enclose the same $W$-poles as $\gw$)
 yields a residue term which is an integral over $\gwno$.
We get
\begin{equation*}
	\widetilde{K}_{\la_N}^{(x_0,t_0)}((x_1,t_1),(x_1,t_1)) 
	= - \frac{1}{(2\I \pi)^2} \oint_{\gzn}\!\oint_{\gwn} {\Int_N(W,Z)} \dd{W}\dd{Z}
	- \frac{1}{2\I \pi} \oint_{\gwno} {\frac{1}{1-t_0-\frac{t_1}{\sqrt N}}} \dd{W}.
\end{equation*}
Again, the second integral is identically $0$,
while the first tends to $0$ as $N$ tends to $+\infty$.
This ends the proof of the proposition.
\end{proof}

Propositions \ref{prop:convergence_kernel_frozen1bis}, \ref{prop:convergence_kernel_frozen2}, and \ref{prop:convergence_kernel_frozen3} complete the proof of the second item of \cref{thm:cv_bead_process}.

\section{The limiting height function and the continuity of the limiting surface}\label{sec:LimitShape}

In this section, we prove \cref{thm:limiting_surface_formula,thm:continuity}. Recall (from \cref{thm:limiting_surface_existence}) that $H^\infty: [\eta\, a_0, \eta\, a_m] \times [0,1]\to \R$ is the deterministic limiting height function of the sequence of random height functions $\frac{1}{\sqrt N}
H_{\lambda_N}(\lfloor x \sqrt N \rfloor, t)$ coming from the bead processes associated with a uniform random (Poissonized) Young tableaux of fixed shape $\lambda^0$. We want to show that
\begin{equation*}
	H^\infty(x,t) = \frac{1}{\pi} \int_0^t \alpha(x,s)\dd{s}, \quad \text{for all }(x,t)\in [\eta\, a_0, \eta\, a_m] \times[0,1].
\end{equation*}

\begin{proof}[Proof of \cref{thm:limiting_surface_formula}]
	We fix $(x,t)\in [\eta\, a_0, \eta\, a_m] \times[0,1]$. Recalling the boundary conditions in \eqref{eq:boundary_conditions}, we note that
	 \[ \frac{1}{\sqrt N}
	 H_{\lambda_N}(\lfloor x \sqrt N \rfloor, t)\leq\frac{1}{2{\sqrt N}} \left(\omega_{\lambda_N}(\lfloor x \sqrt N \rfloor) - |\lfloor x \sqrt N \rfloor|\right)=\frac12 (\omega_{\eta \la^0}(x)-|x|), \quad \text{for all } N>0.\]
	Therefore, by dominated convergence theorem, the convergence in \cref{thm:limiting_surface_existence} implies that for every fixed $(x,t)\in [\eta\, a_0, \eta\, a_m] \times[0,1]$,
	\[H^\infty(x,t) = \lim_{N\to+\infty}\frac{1}{\sqrt N}
	\mathbb E\big[ H_{\lambda_N}(\lfloor x \sqrt N \rfloor, t)\big]. \]
	We will prove that the right-hand side converges to $\frac{1}{\pi} \int_0^t \alpha(x,s) ds$,
	implying \cref{thm:limiting_surface_formula} by uniqueness of the limit.
	By definition, $H_{\lambda_N}(\lfloor x \sqrt N \rfloor, t)$ is the number of beads 
	in the bead process $M_{\lambda_N}$ which lie on the thread at position $\lfloor x \sqrt N \rfloor$ and with height in $[0,t]$.
	Since $M_{\lambda_N}$ is a determinantal point process with kernel $K_{\lambda_N}$, we have
	\[ \esper\big[ H_{\lambda_N}(\lfloor x \sqrt N \rfloor, t) \big]
	= \int_0^t  K_{\lambda_N}\big( (\lfloor x \sqrt N \rfloor, s), (\lfloor x \sqrt N \rfloor, s) \big) ds.\]
	With the notation in \eqref{eq:renorm_kern}, we get
	\[ \frac1{\sqrt N} K_{\lambda_N}\big( (\lfloor x \sqrt N \rfloor, s), (\lfloor x \sqrt N \rfloor, s) \big)
	=\widetilde{K}^{(x,s)}_{\lambda_N}((0,0),(0,0)).\]
	Recall now that $\a(x,s):=\frac{\Im U_c}{1-s}\mathds{1}_{(x,s)\in L}$, where $L$ denotes the liquid region. If $(x,s)$ is in the liquid region, \cref{eq:Kinfty,eq:asymp_Ktilde} imply that
	\[ \lim_{N\to +\infty} \widetilde{K}^{(x,s)}_{\lambda_N}((0,0),(0,0))=K_\infty((0,0),(0,0)) 
	=\frac{1}{2i\pi} \int_{\overline{U_c}}^{U_c} \frac{dW}{1-s}
	=\frac{1}{\pi} \frac{\Im U_c}{1-s} =\frac{\alpha(x,s)}{\pi},
	\]
	where, for the third last equality, we used a vertical path from $\overline{U_c}$ to $U_c$, as in the beginning of the proof of \cref{lem:recovering_bead_kernel}.
	On the other hand, if $(x,s)$ is outside the liquid region, we have, by
	\cref{prop:convergence_kernel_frozen1bis,prop:convergence_kernel_frozen2,prop:convergence_kernel_frozen3},
	\[ \lim_{N\to +\infty} \widetilde{K}^{(x,s)}_{\lambda_N}((0,0),(0,0))=0=\frac{\alpha(x,s)}{\pi}.\]
	Hence, we get that $\lim_{N\to +\infty} \widetilde{K}^{(x,s)}_{\lambda_N}((0,0),(0,0))=\frac{\alpha(x,s)}{\pi}$,
	for all $(x,s)\in[\eta\, a_0, \eta\, a_m] \times [0,1]$. Moreover, since integration contours,
	integrands and critical points are continuous function of $(x,s)$, and since
	all asymptotic estimates given are uniform on compact sets,
	the kernel convergence is also uniform on compact sets. 
	We conclude that 
	\[ 
	\lim_{N\to+\infty}\frac{1}{\sqrt N}\,
	\mathbb E\big[ H_{\lambda_N}(\lfloor x \sqrt N \rfloor, t)\big]=\lim_{N\to +\infty} 
	 \int_0^t  \widetilde{K}^{(x,s)}_{\lambda_N}((0,0),(0,0)) ds
	=\frac{1}{\pi} \int_0^t \alpha(x,s) ds,\]
	proving \cref{thm:limiting_surface_formula}.
\end{proof}

We now turn to the proof of the continuity criterion for the limiting surface $T^\infty$ stated in \cref{thm:continuity}.

\begin{proof}[Proof of \cref{thm:continuity}]
	Recall the definition of the quantities $T^\infty_-(x,y)$ and $T^\infty_+(x,y)$ from \eqref{eq:Tplusminus} and recall also the definition of $t_-(x)$ and $t_+(x)$ from \cref{prop:t_regions}, noting that $t_-(x)=\inf\left\{\partial L\cap(\{x\}\times [0,1])\right\}$ and $t_+(x)=\sup\left\{\partial L\cap(\{x\}\times [0,1])\right\}$.
	
	From \cref{rem:cont_points}, we know that the limiting surface $T^\infty$ is continuous in its domain $D_{\lambda^0}$ if and only if $T^\infty_-(x,y)=T^\infty_+(x,y)$ for all $(x,y)\in D_{\lambda^0}$. The latter condition is equivalent to the following property of the liquid region $L$: for all $x\in[\eta\, a_0 , \eta\, a_m]$,
	\begin{equation}\label{eq:cond}
		\text{$[t_-(x),t_+(x)]\setminus \left\{L\cap(\{x\}\times [0,1])\right\}$ does not contain any non-empty open interval.}
	\end{equation}
	Indeed, if this is not true, i.e.\ for some $x\in[\eta\, a_0 , \eta\, a_m]$ there exists an non-empty open interval $I$ contained in $[t_-(x),t_+(x)]\setminus \left\{L\cap(\{x\}\times [0,1])\right\}$, then there exists some constant $c>0$ such that $H^\infty(x,t)=c$ for all $t\in I$ and so $T^\infty_-(x,2c-|x|) < T^\infty_+(x,2c-|x|)$.
	
	Now recall the parametrization of the frozen boundary curve $\partial L$ given in \cref{prop:description_liquid_region}. We observe that a necessary and sufficient condition such that  \eqref{eq:cond} holds for all $x\in[\eta\, a_0 , \eta\, a_m]$ is that the tangent vector $(\dot{x}(s),\dot{t}(s))$ at the cusp points\footnote{We recall that for a plane curve defined by  analytic parametric equations $x(t)=f(t)$ and $y(t)=g(t)$, a \emph{cusp} is a point where both derivatives of $f$ and $g$ are zero, and the directional derivative, in the direction of the tangent, changes sign.} of $\partial L$ is vertical.
	Requesting that the tangent vector $(\dot{x}(s),\dot{t}(s))$ is vertical at a certain point $(x(s),t(s))$ of the curve $\partial L$ is equivalent to impose that its (absolute) slope is zero, that is $|G(s)|=+\infty$. The solutions to the latter equation are $s=\eta\, a_i$ for $i=0,\dots, m$.
	Moreover, the cusp points of $\partial L$ are given by the solutions to the equations:
	\begin{equation*}
		\begin{cases}
			\dot x(s)=1+\frac{\dot \Sigma(s)}{\Sigma(s)^2}=0,\\
			\dot t(s)=G(s)\left(1+\frac{\dot \Sigma(s)}{\Sigma(s)^2}\right)=0,
		\end{cases}
	\end{equation*}
	where the directional derivative, in the direction of the tangent, changes sign. One can check that $(x(s),t(s))$, for $s=\eta\, a_0$ and $s=\eta\, a_m$, are not cusp points. Indeed, these two points are the two points where the frozen boundary curve $\partial L$ is tangent to the two vertical boundaries of $[\eta\, a_0,\eta\, a_m]\times[0,1]$. 
	Hence, it is enough to impose that for all $i=1,\dots,m-1$,
	\begin{equation*}
		\begin{cases}
			\dot x(\eta\, a_i)=0,\\
			\dot t(\eta\, a_i)=0.
		\end{cases}
	\end{equation*}
	With some standard computation, one can note that $	\dot x(\eta\, a_i)=\lim_{s\to \eta\, a_i}1+\frac{\dot \Sigma(s)}{\Sigma(s)^2}=0$  for all $i=1,\dots,m-1$. Therefore, the only non-trivial condition is to impose that $\dot t(\eta\, a_i)=0$ for all $i=1,\dots,m-1$. Again, with some standard computations, one can check that 
	\begin{equation*}
		\dot t(\eta\, a_i)=\lim_{s\to \eta\, a_i}G(s)\left(1+\frac{\dot \Sigma(s)}{\Sigma(s)^2}\right)=\frac{2\left(\sum_{j=0,
					j\neq i}^m \frac{1}{\eta\,a_{i}-\eta\,a_j}-\sum_{j=1}^m \frac{1}{\eta\,a_{i}-\eta\,b_j}\right)\prod_{j=1}^m(\eta\,a_{i}-\eta\,b_j)}{\prod_{j=0,
					j\neq i}^m(\eta\,a_{i}-\eta\,a_j)},
	\end{equation*}
	concluding that the necessary and sufficient conditions for the continuity of $T^\infty$ are the ones in \eqref{eq:cont_cond}.
\end{proof}

\section{Applications}\label{sect:applications}
In this section we prove the applications discussed in \cref{sect:appl_intro}.

\begin{proof}[Proof of \cref{prop:limit_shape}]
	We first prove the formula for $H^\infty_r(x,t)$.
	From \cref{thm:limiting_surface_formula}, we have
	$$H^\infty_r(x,t) = \frac{1}{\pi}\int_0^t \alpha_r(x,s)\dd{s},$$ 
	where $\a_r(x,s)=\frac{\Im U_c}{1-s}\mathds{1}_{(x,s)\in L}$.
	In this case the critical equation becomes
	\[U \, \big(x-\sqrt r +\tfrac1{\sqrt r}+U \big)
	   = (1-t) \big(x+\tfrac1{\sqrt r}+U \big) \big(x-\sqrt r +U \big).\]
	Solving this quadratic equation and integrating the imaginary part
	of the solution gives the formula \eqref{eq:limit_H_rectangular}
	for $H^\infty_r$.
	
	The fact that the limiting surface $T^\infty_r(x,y)$ is continuous for all $r>0$ 
	is a consequence of \cref{thm:continuity}. 
\end{proof}

\begin{proof}[Proof of \cref{prop:phase_diagram}]
	From \cref{thm:continuity} we know that the surface $T^\infty_{p,q,r}$ is continuous in its domain if and only if 
	\begin{equation*}
		 \frac{1}{a_{1}-a_0}+\frac{1}{a_{1}-a_2}= \frac{1}{a_{1}-b_1}+\frac{1}{a_{1}-b_2},
	\end{equation*}
	with the coefficients $a_0,a_1,a_2,b_1,b_2$ as in \cref{eq:param_int}. That is,
	\begin{equation*}
		\frac{1}{1 + p}-\frac{2}{2 + p - q}-\frac{2}{p + q - 2 r}+\frac{1}{p - r}=0.
	\end{equation*}
	Solving the above equation, one gets the solutions claimed in the proposition statement.
\end{proof}

\section{Local limits for random Young tableaux}\label{sect:local_syt}

The goal of this section is to complete the proof of \cref{corol:local_lim_standard Young tableau}.

\subsection{Local topology for standard Young tableaux}
\label{sect:local_topo_standard Young tableau}
We start with some definitions.
We recall that a \emph{marked standard Young tableau} is a triple $(\la,T,(x,y))$ where $(\la,T)$ is a standard Young tableau of shape $\la$ and $(x,y)$ are the coordinates of a distinguished box in $\la$ (see the left-hand side of \cref{fig:Restrictions} for an example). A \emph{marked Poissonized Young tableau} is defined analogously.

We introduce a family of \emph{restriction functions} for marked standard Young tableaux/Poissonized Young tableaux. Fix $h\in \N$. Given a marked standard Young tableau/Poissonized Young tableau $(\la,T,(x,y))$, we denote by $r_h\left(\la,T,(x,y)\right)$ the standard Young tableau $(\diamond_h,R)$, where $\diamond_h$ is a square Young diagram of size $h^2$ and $R$ is a filling of the boxes of $\diamond_h$ obtained as follows:\footnote{We highlight the fact that by definition the restriction functions $r_h$ give always back a standard Young tableau (even in the case of Poissonized Young tableaux). Note that the rescaling procedure in step (2) is well-defined also when $(\la,T,(x,y))$ is a marked Poissonized Young tableau since we restrict to the case where the mapping $T$ has distinguished values.}
\begin{enumerate}
	\item first, we define $\widetilde R(x',y')=T(x+x',y+y')$ for all $(x',y')\in\diamond_h$, where we set $R(x',y')=*$ if $T(x+x',y+y')$ is not well-defined (see the middle picture of \cref{fig:Restrictions});
	\item second, we rescale the values of the map $\widetilde R$ obtaining a new map $R$ so that $(\diamond_h,R)$ is a standard Young tableau and the values of $R$ have the same relative order as the values of $\widetilde R$, and the boxes filled by $*$ are kept as they are (see the right-hand side of \cref{fig:Restrictions}).
\end{enumerate}

\begin{figure}[ht]
	\includegraphics[scale=0.8]{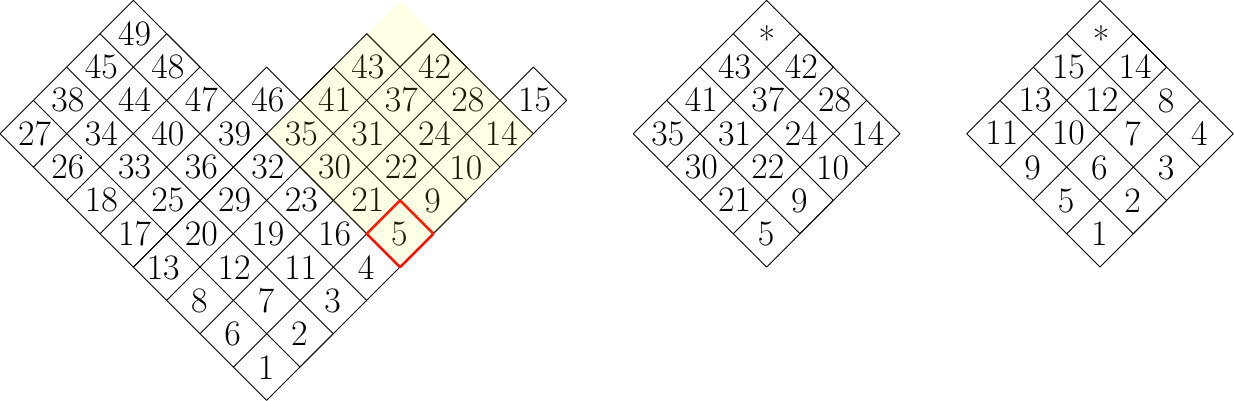}
	\captionsetup{width=\linewidth}
	\caption{\textbf{Left:} A marked standard Young tableau $\left(\la,T,(x,y)\right)$ with the marked box highlighted in red at position $(x,y)=(4,5)$. \textbf{Middle:} The pair $(\diamond_4,\widetilde R)$ obtained from the first step in the definition of the restriction function $r_4\left(\la,T,(x,y)\right)$. \textbf{Right:} The pair $(\diamond_4,R)$ obtained from the second step in the definition of the restriction function $r_4\left(\la,T,(x,y)\right)$.}
	\label{fig:Restrictions}
\end{figure}

An \emph{infinite standard Young tableau} is a pair $(\diamond_{\infty},T)$ where $\diamond_{\infty}$ is the infinite Young diagram formed by all the boxes at positions $(x,y)\in\Z^2$ such that $x+y$ is odd and $y> |x|$
(recall we are using the Russian notation; see the left-hand side of \cref{fig:interlacing1}, p.\ \pageref{fig:interlacing1}) and $T:\diamond_{\infty}\to\Z_{\geq 1}$ is a bijection that is increasing along rows and columns. An infinite standard Young tableau is always (implicitly) marked at the box $(0,1)$. An example of an infinite standard Young tableau is given in \cref{fig:InfiniteSYT}.

\begin{figure}[ht]
	\includegraphics[scale=0.8]{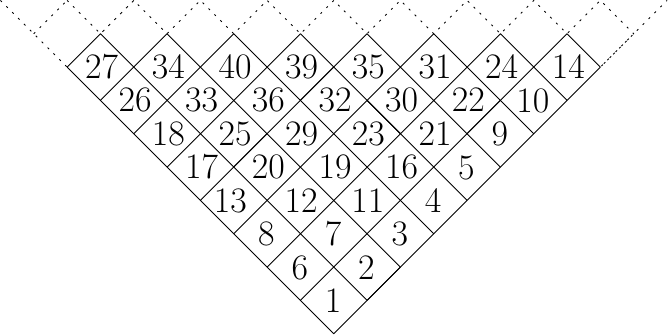}
	\captionsetup{width=\linewidth}
	\caption{An example of an infinite standard Young tableau.}
	\label{fig:InfiniteSYT}
\end{figure}

\begin{definition}\label{defn:local_conv}
	Given a sequence of marked standard Young tableaux/Poissonized Young tableaux $(\la_n,T_n,(x_n,y_n))_n$ and an infinite standard Young tableau $(\diamond_{\infty},T)$, we say that $(\la_n,T_n,(x_n,y_n))_n$ \emph{locally converges} to $(\diamond_{\infty},T)$ if for every $h\in\N$,
	$$r_h(\la_n,T_n,(x_n,y_n))\xrightarrow[n\to\infty]{}r_h\left(\diamond_{\infty},T\right).$$
\end{definition}

The space of marked standard Young tableaux/Poissonized Young tableaux and infinite standard Young tableaux with the topology induced by the local convergence defined in \cref{defn:local_conv} is a metrizable Polish space, so one can consider convergence in distribution with respect to this local topology.

\subsection{Local convergence for Poissonized Young tableaux}
We first prove that the random infinite standard Young tableau from \cref{defn:isyt} is well-defined.
Recall the definitions of the functions $H_\be$ and $R_\be$ of the infinite bead process $M_\beta$ from the discussion above \cref{defn:isyt}.

\begin{proposition}\label{prop:inifinite_tableau_well_defined}
Fix a parameter $\be \in (-1,1)$. Then the heights $$\left\{H_{\be}(x,y), \, (x,y) \in \diamond_{\infty}\right\}$$
are a.s.\ all distinct and they have a.s.\ no accumulation points.
\end{proposition}
\begin{proof}
The infinite bead process $M_\beta$ is a determinantal point process. Hence the distribution of the heights has a joint density
with respect to the Lebesgue measure and so, it puts zero measure
on vectors with two equal coordinates. Therefore the  heights $\left\{H_{\be}(x,y), \, (x,y) \in \diamond_{\infty}\right\}$ are a.s.\ all distinct.

We now prove that there are a.s.\ no accumulation points. It is enough to prove that $H_{\be}(k,k+1)$ tends to $+\infty$ a.s., as $k$ tends to infinity.
Indeed, assume this is true. By symmetry $H_{\be}(-k,k+1)$ also tends to $+\infty$.
But the interlacing condition implies that there are at most $k^2$ pairs $(x,y)$ in $\diamond_{\infty}$ with $H_{\be}(x,y) \le \min(H_{\be}(k,k+1),H_{\be}(-k,k+1))$. 
Therefore, if $\min(H_{\be}(k,k+1),H_{\be}(-k,k+1))$ tends to $+\infty$,
then the set $\{H_{\be}(x,y), \, (x,y) \in \diamond_{\infty}\}$ cannot have accumulation points.

To prove that $H_{\be}(k,k+1)$ tends to $+\infty$ a.s., we write 
\[ H_{\be}(k,k+1) -H_{\be}(0,1) = \sum_{j=1}^k H_{\be}(j,j+1)-H_{\be}(j-1,j).\]
Since bead models are translation invariant,
we can use Birkhoff's ergodic theorem and conclude that, a.s.,
\[ \lim_{k\to\infty} \frac{H_{\be}(k,k+1) -H_{\be}(0,1)}{k} = \mathbb E\big[ H_{\be}(1,2)-H_{\be}(0,1)\big]>0.\qedhere\]
\end{proof}

We now complete the proof of of \cref{corol:local_lim_standard Young tableau}.

\begin{proof}[Proof of \cref{corol:local_lim_standard Young tableau}]
  Fix $h\in\N$. In order to prove the {result} it is enough to show that 
$r_h\left(\la_N,T_N,\Box_N\right)$ converges in distribution to
$r_h\left(\diamond_{\infty},R_{\be}\right)$.
Recall from \cref{ssec:local_tableau_intro} that $(\diamond_{\infty},R_{\be})$
is constructed starting from an infinite bead process $M_\beta$ of skewness $\be$ (and intensity $\alpha=1$).
Let $X_h$ be the height of the $h$-th bead above $0$ in the thread of index $0$
in $M_\beta$. 
We note that the restriction $r_h(\diamond_{\infty},R_{\be})$
is completely determined by the relative positions of beads in the window
$\{-h+1,\dots,-1,0,1,\dots, h-1\} \times [0,X_h]$.

Using the Skorohod representation theorem,
let us assume that the convergence in \cref{thm:cv_bead_process}
holds almost surely. Then, for any fixed $A>0$,
the positions of the beads of $\widetilde M^{(x_0,t_0)}_{\la_N}$ in
 $\{-h+1,\dots,-1,0,1,\dots, h-1\} \times (0,A)$ converge
 to the positions of the same beads of $M_\beta$.
 If $X_h<A$, the relative positions of these latter beads a.s.\ determine (thanks to \cref{prop:inifinite_tableau_well_defined}), for $N$ large enough (but random), both
tableaux
\[r_h(\la_N,T_N,\Box_N)  \quad\text{and}\quad r_h (\diamond_{\infty},R_{\be}),\]
where we recall that $\oblong_N$ is the box corresponding to
the first bead of $M_{\la_N}$ above height $t_0$ in the $x_0 \sqrt{N}$-th thread, i.e. the first bead of $\widetilde M^{(x_0,t_0)}_{\la_N}$ above height zero in the zero thread.

Therefore, conditionally on $X_h<A$, for $N$ large enough (but random), we a.s.\ have
 \[ r_h(\la_N,T_N,\Box_N )= r_h(\diamond_{\infty},R_{\be}).\]
Since this holds for all $A>0$ and since $X_h$ is a.s.\ finite, we have the almost sure convergence of
$r_h(\la_N,T_N,\Box_N )$ to $r_h(\diamond_{\infty},R_{\be})$
in the probability space constructed by the Skorohod representation theorem.
Almost sure convergence implies convergence in distribution of $r_h(\la_N,T_N,\Box_N )$ to $r_h(\diamond_{\infty},R_{\be})$,
which concludes the proof.
\end{proof}

\bibliographystyle{bibli_perso}
\bibliography{bibli}

\def\cprime{$'$}
\begin{thebibliography}{AHRV07}

\bibitem[Agg23]{aggarwal2019universality}
A.~Aggarwal.
\newblock Universality for lozenge tiling local statistics.
\newblock {\em Ann. Math. (2)}, 198(3):881--1012, 2023.

\bibitem[AHRV07]{angel2007sorting}
O.~Angel, A.~E. Holroyd, D.~Romik, and B.~Vir{\'a}g.
\newblock Random sorting networks.
\newblock {\em Adv. Math.}, 215(2):839--868, 2007.

\bibitem[ANvM14]{adler2014minor}
M.~Adler, E.~Nordenstam, and P.~van Moerbeke.
\newblock The {Dyson} {Brownian} minor process.
\newblock {\em Ann. Inst. Fourier}, 64(3):971--1009, 2014.

\bibitem[Bia98]{Biane1998}
P.~Biane.
\newblock {Representations of symmetric groups and free probability}.
\newblock {\em Adv. Math.}, 138(1):126--181, 1998.

\bibitem[Bia01]{Biane2001}
P.~Biane.
\newblock {Approximate factorization and concentration for characters of
  symmetric groups}.
\newblock {\em Internat. Math. Res. Notices}, 4:179--192, 2001.

\bibitem[Bia03]{Biane2003}
P.~Biane.
\newblock {Characters of symmetric groups and free cumulants}.
\newblock In {\em {Asymptotic combinatorics with applications to mathematical
  physics (St. Petersburg, 2001)}}, volume 1815 of {\em {Lecture Notes in
  Math.}}, pages 185--200. Springer, Berlin, 2003.

\bibitem[BK08]{borodin2008unitary}
A.~Borodin and J.~Kuan.
\newblock Asymptotics of {Plancherel} measures for the infinite-dimensional
  unitary group.
\newblock {\em Adv. Math.}, 219(3):894--931, 2008.

\bibitem[BMW20]{banderier2020tableaux}
C.~Banderier, P.~Marchal, and M.~Wallner.
\newblock Periodic {P{\'o}lya} urns, the density method and asymptotics of
  {Young} tableaux.
\newblock {\em Ann. Probab.}, 48(4):1921--1965, 2020.

\bibitem[BOO00]{BorodinOkounkovOlshanski2000}
A.~Borodin, A.~Okounkov, and G.~Olshanski.
\newblock {Asymptotics of Plancherel measures for symmetric groups}.
\newblock {\em J. Amer. Math. Soc}, 13:481--515, 2000.

\bibitem[Bou05]{boutillier2005modeles}
C.~Boutillier.
\newblock {\em Mod{\`e}les de dim{\`e}res: comportements limites}.
\newblock PhD thesis, Universit{\'e} Paris Sud-Paris XI, 2005.

\bibitem[{Bou}09]{boutillier2009bead}
C.~{Boutillier}.
\newblock {The bead model and limit behaviors of dimer models}.
\newblock {\em {Ann. Probab.}}, 37(1):107--142, 2009.

\bibitem[BR10]{baryshnikov_romik2010}
Y.~Baryshnikov and D.~Romik.
\newblock Enumeration formulas for {Young} tableaux in a diagonal strip.
\newblock {\em Isr. J. Math.}, 178:157--186, 2010.

\bibitem[DJM16]{DuseJohanssonMetcalfe2016cuspAiry}
E.~Duse, K.~Johansson, and A.~Metcalfe.
\newblock The cusp-{Airy} process.
\newblock {\em Electron. J. Probab.}, 21:50, 2016.
\newblock Id/No 57.

\bibitem[DVJ08]{daley2008introduction}
D.~Daley and D.~Vere-Jones.
\newblock {\em An Introduction to the Theory of Point Processes. Volume II:
  General Theory and Structure}.
\newblock Springer, 2008.

\bibitem[Elk03]{elkies2003}
N.~D. Elkies.
\newblock On the sums {{\(\sum_{k=-\infty}^\infty (4k+1)^{-n}\)}}.
\newblock {\em Am. Math. Mon.}, 110(7):561--573, 2003.

\bibitem[FFN12]{fleming2012finitization}
B.~J. Fleming, P.~J. Forrester, and E.~Nordenstam.
\newblock A finitization of the bead process.
\newblock {\em Probab. Theory Relat. Fields}, 152(1-2):321--356, 2012.

\bibitem[FRT54]{HookLengthFormula1954}
J.~S. Frame, G.~d.~B. Robinson, and R.~M. Thrall.
\newblock {The hook graphs of the symmetric group}.
\newblock {\em Canadian Journal of Mathematics}, 6:316--324, 1954.

\bibitem[GNW82]{greene1982probabilistic}
C.~Greene, A.~Nijenhuis, and H.~Wilf.
\newblock A probabilistic proof of a formula for the number of {Y}oung tableaux
  of a given shape.
\newblock In {\em Young Tableaux in Combinatorics, Invariant Theory, and
  Algebra}, pages 17--22. Elsevier, 1982.

\bibitem[Gor20]{gordenko2020limit}
A.~Gordenko.
\newblock Limit shapes of large skew young tableaux and a modification of the
  tasep process.
\newblock {\em Preprint arXiv:2009.10480}, 2020.

\bibitem[GR19]{GR19}
V.~{Gorin} and M.~{Rahman}.
\newblock {Random sorting networks: local statistics via random matrix laws}.
\newblock {\em {Probab. Theory Relat. Fields}}, 175(1-2):45--96, 2019.

\bibitem[GX22]{gorin2022random}
V.~Gorin and J.~Xu.
\newblock Random sorting networks: edge limit.
\newblock Preprint arXiv:2207.09000, 2022.

\bibitem[HKPV09]{hough2009determinantal}
J.~B. Hough, M.~Krishnapur, Y.~Peres, and B.~Vir{\'a}g.
\newblock {\em Zeros of {Gaussian} analytic functions and determinantal point
  processes}, volume~51 of {\em Univ. Lect. Ser.}
\newblock Providence, RI: American Mathematical Society (AMS), 2009.

\bibitem[Hor16]{HoraBookYoung}
A.~Hora.
\newblock {\em The limit shape problem for ensembles of {Young} diagrams},
  volume~17 of {\em SpringerBriefs Math. Phys.}
\newblock Tokyo: Springer, 2016.

\bibitem[IO02]{IvanovOlshanski2002}
V.~Ivanov and G.~Olshanski.
\newblock {Kerov's central limit theorem for the {P}lancherel measure on
  {Y}oung diagrams}.
\newblock In {\em {Symmetric functions 2001: surveys of developments and
  perspectives}}, volume~74 of {\em {NATO Sci. Ser. II Math. Phys. Chem.}},
  pages 93--151. Kluwer Acad. Publ., Dordrecht, 2002.

\bibitem[Ker00]{kerov2000anisotropic}
S.~V. Kerov.
\newblock Anisotropic {Young} diagrams and {Jack} symmetric functions.
\newblock {\em Funct. Anal. Appl.}, 34(1):41--51, 2000.

\bibitem[KO07]{KenyonOkounkov2007Burgers}
R.~Kenyon and A.~Okounkov.
\newblock Limit shapes and the complex {Burgers} equation.
\newblock {\em Acta Math.}, 199(2):263--302, 2007.

\bibitem[KP22]{KenyonPrause2022variational}
R.~Kenyon and I.~Prause.
\newblock Gradient variational problems in {{\(\mathbb{R}^2\)}}.
\newblock {\em Duke Math. J.}, 171(14):3003--3022, 2022.

\bibitem[KV86]{KerovVershik1986}
S.~Kerov and A.~M. Vershik.
\newblock {The characters of the infinite symmetric group and probability
  properties of the {R}obinson-{S}chensted-{K}nuth algorithm}.
\newblock {\em SIAM J. Algebraic Discrete Methods}, 7(1):116--124, 1986.

\bibitem[LS77]{LoganShepp1977}
B.~F. Logan and L.~A. Shepp.
\newblock {A variational problem for random {Y}oung tableaux}.
\newblock {\em Advances in Math.}, 26(2):206--222, 1977.

\bibitem[LV21]{li_vuletic2022}
Z.~Li and M.~Vuletić.
\newblock Asymptotics of pure dimer coverings on rail-yard graphs, 2021.
\newblock Preprint arXiv:2110.11393.

\bibitem[Mar16]{marchal2016Young}
P.~Marchal.
\newblock Rectangular {Young} tableaux and the {Jacobi} ensemble.
\newblock In {\em Proceedings of the 28th international conference on formal
  power series and algebraic combinatorics, FPSAC 2016, Vancouver, Canada, July
  4--8, 2016}, pages 839--850. Nancy: The Association. Discrete Mathematics \&
  Theoretical Computer Science (DMTCS), 2016.

\bibitem[Mkr14]{mkrtchyan2014}
S.~Mkrtchyan.
\newblock Plane partitions with two-periodic weights.
\newblock {\em Lett. Math. Phys.}, 104(9):1053--1078, 2014.

\bibitem[M{\'S}22]{sniady2022evacuation}
{\L}.~Ma{\'s}lanka and P.~{\'S}niady.
\newblock Second class particles and limit shapes of evacuation and sliding
  paths for random tableaux.
\newblock {\em Doc. Math.}, 27:2183--2273, 2022.

\bibitem[OR03]{okounkov2003SchurProcess}
A.~Okounkov and N.~Reshetikhin.
\newblock Correlation function of {S}chur process with application to local
  geometry of a random 3-dimensional {Y}oung diagram.
\newblock {\em J. Amer. Math. Soc.}, 16(3):581--603, 2003.

\bibitem[OR06]{OkounkovReshetikhin2006birth}
A.~Okounkov and N.~Reshetikhin.
\newblock The birth of a random matrix.
\newblock {\em Mosc. Math. J.}, 6(3):553--566, 2006.

\bibitem[OR07]{OkounkovReshetikhin2007Pearcey}
A.~Okounkov and N.~Reshetikhin.
\newblock Random skew plane partitions and the {Pearcey} process.
\newblock {\em Commun. Math. Phys.}, 269(3):571--609, 2007.

\bibitem[{Pet}14]{petrov2014tilings}
L.~{Petrov}.
\newblock {Asymptotics of random lozenge tilings via Gelfand-Tsetlin schemes}.
\newblock {\em {Probab. Theory Relat. Fields}}, 160(3-4):429--487, 2014.

\bibitem[PR07]{pittel2007young}
B.~Pittel and D.~Romik.
\newblock Limit shapes for random square {Young} tableaux.
\newblock {\em Adv. Appl. Math.}, 38(2):164--209, 2007.

\bibitem[Pra24]{Prause2023tableaux}
I.~Prause.
\newblock {Random Young tableaux and the tangent plane method}.
\newblock In preparation, 2024+.

\bibitem[Rom06]{romik2006RandomES}
D.~Romik.
\newblock Permutations with short monotone subsequences.
\newblock {\em Adv. Appl. Math.}, 37:501--510, 2006.

\bibitem[Rom15]{RomikLIS}
D.~Romik.
\newblock {\em The surprising mathematics of longest increasing subsequences},
  volume~4 of {\em Institute of Mathematical Statistics Textbooks}.
\newblock Cambridge University Press, 2015.

\bibitem[Sta86]{stanley1986polytopes}
R.~Stanley.
\newblock Two poset polytopes.
\newblock {\em Discrete Comput. Geom.}, 1:9--23, 1986.

\bibitem[Sun18]{sun2018dimer}
W.~Sun.
\newblock Dimer model, bead model and standard {Y}oung tableaux: finite cases
  and limit shapes, 2018.
\newblock Preprint arXiv:1804.03414.

\bibitem[VK77]{VershikKerov1977}
A.~M. Vershik and S.~V. Kerov.
\newblock {Asymptotic behavior of the {P}lancherel measure of the symmetric
  group and the limit form of {Y}oung tableaux}.
\newblock {\em Dokl. Akad. Nauk SSSR}, 233(6):1024--1027, 1977.

\end{thebibliography}

\end{document}